\documentclass[a4paper,11pt]{amsart}
\usepackage{chngcntr, amsmath, amsfonts, amssymb, textcomp, mysymb-private,  latexsym,paralist,soul, tikz}

\usepackage[utf8]{inputenc}
\usepackage[all]{xy}
\usepackage{hyperref}
\usepackage{wasysym}
\usepackage[shortcuts]{extdash}
\usepackage{mathrsfs}

\counterwithin{equation}{section}
\allowdisplaybreaks

\theoremstyle{plain}

\newtheorem{thm}{Theorem}[section]
\newtheorem{lem}[thm]{Lemma}
\newtheorem{cor}[thm]{Corollary}
\newtheorem{prop}[thm]{Proposition}

\theoremstyle{definition}

\newtheorem{defn}[thm]{Definition}
\newtheorem{eg}[thm]{Example}
\newtheorem{rmk}[thm]{Remark}

\newtheorem{constr}[thm]{Construction}

\begin{document}

\newcommand{\shom}{\starhyphen homomorphism}
\newcommand{\fm}{Fredholm module}

\newcommand{\ind}{\iind}

\newcommand{\suspgrad}[1]{\salg\gradtensor{#1}}

\newcommand{\alg}{\aelem}

\newcommand{\clifft}[2][]{\contfualg{#1}\left(#2, \cliffc \tanbndl{#2} \right)}
\newcommand{\clifftpare}[2][]{\contfualg{#1}\left(#2, \cliffc \tanbndl{(#2)} \right)}
\newcommand{\clifftrel}[3][]{\contfualg{#1}\left(#2, \cliffc \tanbndl_{#2}({#3}) \right)}

\newcommand{\spcofmet}{\mathrm{Met}}

\newcommand{\lrep}[1][\gam]{L_{#1}}

\newcommand{\arep}[2][\rho]{#1: \aalg \to \bh[#2]}

\newcommand{\ray}[2][]{\mathfrak{R}_{#1}^{#2}}
\newcommand{\rayinfty}[3][]{(\ray[#1]{#2}{#3})|_\infty}

\newcommand{\ksubeveng}[2][\ggrp]{\kfunctr_0^{#1}(#2)}
\newcommand{\ksuboddg}[2][\ggrp]{\kfunctr_1^{#1}(#2)}
\newcommand{\ksupeveng}[2][\ggrp]{\kfunctr^0_{#1}(#2)}
\newcommand{\ksupoddg}[2][\ggrp]{\kfunctr^1_{#1}(#2)}
\newcommand{\kk}{\kkfunctr}

\newcommand{\rkkggrp}[1][*]{\kkfunctr_{#1}^\ggrp}
\newcommand{\rkkgam}[1][*]{\kkfunctr_{#1}^\gamgrp}

\newcommand{\AofM}[1][\mnf]{\aalg({#1})}
\newcommand{\AevenofM}[1][\mnf]{\aalg_{\mathrm{ev}}({#1})}
\newcommand{\AofMrel}[1]{\aalg(\mmnf,#1)}
\newcommand{\AevenofMrel}[1]{\aalg_{\mathrm{ev}}(\mmnf,#1)}

\newcommand{\cliffmult}[1][\xpt_0]{C_{#1}}
\newcommand{\botthom}[1][\xpt_0]{\bottmap_{#1}}

\newcommand{\alfinalfind}{\alfind\in\alfindset}

\newcommand{\bh}[1][]{\boprs({#1})}
\newcommand{\kh}[1][]{\koprs({#1})}

\newcommand{\csds}{C*-dynamical system}
\newcommand{\crep}{covariant representation}
\newcommand{\cp}{crossed product}
\newcommand{\rcp}{reduced crossed product}

\newcommand{\lam}{\lambda}
\newcommand{\Lam}{\Lambda}
\newcommand{\act}[1][\gam]{\alpha_{#1}}
\newcommand{\acsds}{(\aalg\,,\Gam,\act[])}
\newcommand{\acrep}{(\rho,\uopr)}
\newcommand{\urep}[1][\gam]{\uopr_{#1}}
\newcommand{\ging}{{\gam\in\Gam}}
\newcommand{\ling}{{\lam\in\Gam}}
\newcommand{\lact}[1][\xsp]{\Gam\curvearrowright#1}

\newcommand{\trep}{\tilde{\rho}}
\newcommand{\lgh}{\lsp{\Gam}\otimes\hil}
\newcommand{\aacp}[1][]{\aalg\rtimes_{#1alg}\Gam}
\newcommand{\acp}[1][]{\aalg\rtimes_{#1}\Gam}
\newcommand{\arcp}[1][]{\aalg\rtimes_{#1r}\Gam}
\newcommand{\cpr}[1][]{\rtimes_{#1r}}

\newcommand{\chaf}[1][(\gam_i\Lam)]{X_{#1}}

\renewcommand{\salg}{{C_0(\rbbd)}}


\newcommand{\hhs}{{Hilbert\-/Hadamard space}}
\newcommand{\admhhs}{{admissible Hilbert-Hadamard space}}


\title[\tiny The Novikov conjecture, diffeomorphisms \& Hilbert-Hadamard spaces]{The Novikov conjecture, the group of volume preserving diffeomorphisms and Hilbert-Hadamard spaces}

\author{Sherry Gong}\address{S.~Gong: Department of Mathematics, Stanford University, Stanford, CA, USA}\email{sgong2@stanford.edu}

\author{Jianchao Wu}\address{J.~Wu: Department of Mathematics, Texas A\&M University, College Station, TX, USA}\email{j.wu@tamu.edu}

\author{Guoliang Yu}\address{G.~Yu: Department of Mathematics, Texas A\&M University, College Station, TX, USA, and Shanghai Center for Mathematical Sciences, Shanghai, China}\email{guoliangyu@math.tamu.edu}

\thanks{This work was partially supported by the grants NSF~1564401, NSF~1564398, NSF~1700021, NSF~2000082, NSFC~11420101001 and the Simons Fellows Program.}

\date{}

\maketitle

\begin{abstract}
	We prove that the Novikov conjecture holds for any discrete group admitting an isometric and metrically proper action on an admissible Hilbert-Hadamard space. 
	Admissible Hilbert-Hadamard spaces are a class of (possibly infinite\-/dimensional) non-positively curved metric spaces that contain dense sequences of closed convex subsets isometric to Riemannian manifolds. 
	Examples of admissible Hilbert-Hadamard spaces include Hilbert spaces, certain simply connected and non-positively curved Riemannian-Hilbertian manifolds and infinite\-/dimensional symmetric spaces. 
	Thus our main theorem can be considered as an infinite\-/dimensional analogue of Kasparov's theorem on the Novikov conjecture for groups acting properly and isometrically on complete, simply connected and non-positively curved manifolds.
	As a consequence, we show that the Novikov conjecture holds for geometrically discrete subgroups of the group of volume preserving diffeomorphisms of a closed smooth manifold. 
	This result is inspired by Connes' theorem that the Novikov conjecture holds for higher signatures associated to the Gelfand-Fuchs classes of groups of diffeormorphisms.
\end{abstract}

\tableofcontents

\section{Introduction}

\newcommand{\Diff}{\operatorname{Diff}}

A central problem in manifold topology is the Novikov conjecture. The Novikov conjecture states that the higher signatures of closed oriented smooth manifolds are invariant under orientation preserving homotopy equivalences. 
In the case of aspherical manifolds, the Novikov conjecture is an infinitesimal version of the Borel conjecture which states that all closed aspherical manifolds
are topologically rigid, i.e., if another closed manifold $N$ is homotopy equivalent to the given closed aspherical manifold $M$, then $N$ is homeomorphic to $M$. To make this precise, recall that a deep theorem of Novikov says that the rational Pontryagin classes are invariant under orientation preserving homeomorphisms. Since for aspherical manifolds, information about higher signatures is equivalent to that of rational Pontryagin classes, the Novikov conjecture for closed aspherical manifolds follows from the Borel conjecture and Novikov's theorem. 

Noncommutative geometry provides a powerful approach to the Novikov conjecture (cf.\,\cite{connes, baumconnes88, baumconneshigson, kasparov95, miscenko74}). The Novikov conjecture follows from the \emph{rational strong Novikov conjecture}, which states that the rational Baum-Connes assembly map is injective.
Using this approach, the Novikov conjecture has been proved when the fundamental group of the manifold is in one of the following classes: 
\begin{enumerate}
	\item groups acting properly and isometrically on simply connected and non-positively curved manifolds (cf.\,\cite{kasparov1}), 
	\item groups acting properly and isometrically on Hilbert spaces (cf.\,\cite{higsonkasparov}),
	\item hyperbolic groups (cf.\,\cite{connesmoscovici}), 
	\item torsion-free groups acting properly on locally compact buildings (cf.\, \cite{kasparovskandalis91}),
	\item groups acting properly and isometrically on ``bolic'', weakly geodesic metric spaces of bounded geometry (cf.\,\cite{kasparovskandalis2}), 
	\item groups with finite asymptotic dimension (cf.\,\cite{yu2}),
	\item groups coarsely embeddable into Hilbert spaces (cf.\,\cite{yu3, higson00, skandalistuyu}), 
	\item groups coarsely embeddable into Banach spaces with property (H) (cf.\,\cite{kasparovyu12}), 
	\item all linear groups and subgroups of all almost connected Lie groups (cf.\,\cite{guentnerhigsonweinberger}), 
	\item mapping class groups (cf.\,\cite{hamenstadt, kida}), and 
	\item $\operatorname{Out}(F_n)$, the outer automorphism groups of free groups (cf.\,\cite{bestvinaguirardelhorbez}).
\end{enumerate}
In the first three cases, an isometric action of a discrete group $\Gam$ on a metric space $X$ is said to be \emph{proper} if $d(x, g \cdot x) \to \infty$ as $g \to \infty$ for some (equivalently, all) $x$ in $X$.

The next most natural class of groups to study for the Novikov conjecture is that of groups of diffeomorphisms, which can be highly nonlinear in nature. 
Connes \cite{connes1986} and Connes-Gromov-Moscovici \cite{connesgromovmoscovici} proved a very striking theorem that the Novikov conjecture holds for higher signatures associated to Gelfand-Fuchs classes of groups of diffeomorphisms. The proof of this result is a technical tour de force and uses the full power of noncommutative geometry. One important construction in this proof is the space of metrics on a smooth manifold.

Inspired by this work, we prove the following theorem and apply it to the rational strong Novikov conjecture for geometrically discrete subgroups of the volume preserving diffeomorphism group on a closed smooth manifold.

\begin{thm} \label{thm:main}
The rational strong Novikov conjecture holds for groups acting properly and isometrically on an admissible Hilbert-Hadamard space. More precisely, if a countable discrete group $\Gamma$ acts properly and isometrically on an admissible Hilbert-Hadamard space, then 
the rational Baum-Connes assembly map
$$\mu \colon K_*(B \Gamma)\otimes_{\zbbd}\qbbd \to \kfunctr_{*}(\cstar_{r}\Gam)\otimes_{\zbbd}\qbbd$$
is injective, where $B \Gamma$ is the classifying space for free and proper $\Gamma$-actions, $K_*(B \Gamma)$ is the $K$-homology group of $B \Gamma$ with compact supports, and $\kfunctr_{*}(\cstar_{r}\Gam)$ is the $K$-theory of the reduced group $C^*$-algebra of $\Gam$.
\label{admissible_hilbert_hadamard_thm}
\end{thm}

\emph{{\hhs}s} are a type of non-positively curved (i.e., CAT(0)) metric spaces that include Hilbert spaces, complete connected and simply connected (possibly infinite\-/dimensional) Riemannian-Hilbertian manifolds with non-positive sectional curvature, and certain infinite\-/dimensional symmetric spaces. 
The precise definition of Hilbert-Hadamard spaces is given in Section~\ref{sec:hhs}.

We say that a {\hhs} $M$ is \textit{admissible} if 
it has an increasing sequence of closed and convex subsets $M_n$, whose union is dense in $M$, such that each $M_n$, seen with its inherited metric from $M$, is isometric to a finite\-/dimensional Riemannian manifold. 
Here a subset is convex if it contains the geodesic segment between every pair of points in the subset. For example, a Hilbert space is clearly admissible; so are some interesting infinite\-/dimensional symmetric spaces (see Section~\ref{sec:L2metrics}). Gromov asked the third author when a Hilbert-Hadamard space $M$ is admissible. 
In particular, two related open questions are whether every infinite\-/dimensional symmetric space is admissible and whether every complete connected and simply connected Riemannian-Hilbertian manifold with non-positive sectional curvature is admissible. 

Theorem~\ref{thm:main} can be viewed as an infinite\-/dimensional analogue of the aforementioned theorem of Kasparov \cite{kasparov1} on the Novikov conjecture for groups acting properly and isometrically on complete, simply connected and non-positively curved manifolds. In the case of Hilbert spaces, this theorem also follows from \cite{higsonkasparov}. 

Our main theorem can be applied to study the Novikov conjecture for \emph{geometrically discrete subgroups} of the group of volume preserving diffeomorphisms of a closed smooth manifold $N$. More precisely, we fix a density $\omega$ on $N$, which we regard as a measure on $N$, which is, in each smooth chart, equivalent to the Lebesgue measure with a smooth Radon-Nikodym derivative. For an orientable smooth manifold, a density is just a volume form without its orientation. Each Riemannian metric on $N$ induces a density in the same way as an inner product on $\rbbd^n$ induces a volume form on $\rbbd^n$. 
Let $\Diff(N,\omega)$ denote the group of diffeomorphisms on $N$ that fix $\omega$. 
We remark that, up to isomorphism, the group $\Diff(N,\omega)$ is independent of the choice of $\omega$. Indeed, by a result of Moser \cite{Moser1965}\footnote{Although the main theorem in Moser \cite{Moser1965} refers to volume forms, the second footnote on the first page supplies a remark by Calabi that the result extends to de Rham's \emph{odd forms}, which are what we call densities.}, any other density $\omega'$ on $N$ is related to $\omega$ through a diffeomorphism and a rescaling, that is, we have $\omega' = k \cdot \varphi_* \omega$, where $k$ is a positive number and $\varphi_* \omega$ is the push-forward of $\omega$ by a suitable diffeomorphism $\varphi$. Hence the groups $\Diff(N,\omega)$ and $\Diff(N,\omega')$ are isomorphic (through conjugation by said diffeomorphism $\varphi$). 

In order to define the concept of geometrically discrete subgroups of $\Diff(N,\omega)$, let us fix a Riemannian metric on $N$ and define a length function $\lambda$ on $\Diff(N,\omega)$ by taking, for all $\varphi \in \operatorname{Diff}(N,\omega)$, 
\[ 
	\lambda_+(\varphi) = \left(\int_N (\log(\|D\varphi\|))^2 d\omega\right)^{1/2}
\]
and
\[ 
	\lambda(\varphi) =  \max\left\{\lambda_+(\varphi),\lambda_+(\varphi^{-1})\right\} \; ,
\]
where $D\varphi$ is the Jacobian of $\varphi$, and the norm $\|\cdot \|$ denotes the operator norm, computed using the chosen Riemannian metric on $N$. Intuitively speaking, this length function measures how much a diffeomorphism $\varphi$ deviates from an isometry in an $L^2$-sense.

\begin{defn}\label{defn:geometrically-discrete}
A countable subgroup $\Gamma$ of $\Diff(N, \omega)$ is said to be a geometrically discrete subgroup 
if $\lambda(\gam) \rightarrow \infty$ when $\gam \to \infty$ in $\Gam$, i.e., for any $R>0$, there exists a finite subset $F\subset \Gamma$ such that $\lambda(\gam)\geq R$ if $\gam \in \Gamma \setminus F$. 
\end{defn}

We point out that this notion of geometric discreteness does not depend on the above choice of the Riemannian metric, even though the length function $\lambda$ does.
The following result is a consequence of our main theorem.

\begin{thm} \label{thm:diffeo}
	Let $N$ be a closed smooth manifold 
	and let $\Diff(N, \omega)$ be the group of all volume preserving diffeomorphisms of $N$ (for some density $\omega$). Then the rational strong Novikov conjecture holds for any geometrically discrete subgroup of $\Diff(N,\omega)$. 
\label{volume_preserving_diff_thm}
\end{thm}

Observe that if $\varphi \in \Diff(N,\omega)$ preserves the Riemannian metric we chose in the definition of $\lambda$, then $\lambda(\varphi)=0$. This suggests that geometrically discrete subgroups of $\Diff(N, \omega)$ are, in a sense, conceptual antitheses to subgroups of isometries. We remark that since the group of isometries of a closed Riemannian manifold is a Lie group, all its countable subgroups satisfy the rational strong Novikov conjecture by \cite{guentnerhigsonweinberger}. This gives hope for a unified approach to prove the rational strong Novikov conjecture for all countable subgroups of $\Diff(N, \omega)$. 

Theorem~\ref{volume_preserving_diff_thm} can be derived from Theorem~\ref{admissible_hilbert_hadamard_thm} as follows. A key point is to model the geometry of $\Diff(N,\omega)$ after a certain infinite-dimensional symmetric space, which is a natural example of an admissible {\hhs}, thus allowing us to apply Theorem~\ref{admissible_hilbert_hadamard_thm}. More precisely, $\Diff(N,\omega)$ acts isometrically on the infinite-dimensional symmetric space
\[
	L^2(N,\omega, \operatorname{SL}(n, \mathbb{R})/\operatorname{SO}(n)) 
\]
of \emph{$L^2$-Riemannian metrics} on an $n$-dimensional closed smooth manifold $N$ with a fixed density $\omega$. 
This infinite-dimensional symmetric space can be defined as the completion of the space of all bounded Borel maps from $N$ to $\operatorname{SL}(n, \mathbb{R})/\operatorname{SO}(n)$ with regard to the following metric:
\[
	d(\xi, \eta) = 	\left( \int_{y \in N} (d_X(\xi(y),\eta(y)))^2 \, d\omega(y) \right)	^{\frac{1}{2}} \text{ for two such maps } \xi \text{ and } \eta \; , 
\]
where $d_X$ is the standard Riemannian metric on the symmetric space $X = \operatorname{SL}(n, \mathbb{R})/\operatorname{SO}(n)$. Observe that this symmetric space parametrizes all inner products on $\rbbd^n$ with a fixed volume form. Thus Riemannian metrics on $N$ that induce $\omega$ correspond to the smooth sections of an $X$-bundle over $N$. Upon taking a Borel trivialization of this bundle, these smooth sections are embedded into the space of all bounded Borel maps from $N$ to $X$, and thus also into $L^2(N,\omega, X)$, with a dense image. This explains the terminology ``$L^2$-Riemannian metrics''. 

The fact that $L^2(N,\omega, X)$ constitutes a {\hhs} essentially follows from the fundamental point that $X$ is non-positively curved as a symmetric space. This construction is key to Theorem~\ref{volume_preserving_diff_thm} since the group $\Diff(N,\omega)$ acts isometrically on $L^2(N,\omega, X)$ in the same way as it permutes Riemannian metrics on $N$ (Construction~\ref{constr:diff-isom}). For a countable subgroup $\Gamma$ of $\Diff(N,\omega)$, the property of being geometrically discrete corresponds precisely to the properness of the natural action $\Gamma \curvearrowright L^2(N,\omega, X)$ (cf.~Proposition~\ref{prop:SLSO-proper-length}). This allows us to invoke Theorem~\ref{admissible_hilbert_hadamard_thm}.

We remark on two aspects in which our paper differs from previous known results on the rational strong Novikov conjecture: 
\begin{enumerate}
	\item Geometrically, most previous cases can be proved by coarse embeddings into Hilbert spaces; however, groups of diffeomorphisms can be highly nonlinear in nature, and Hilbert spaces seems inadequate to model the large-scale geometry of these groups (see Remark~\ref{rmk:SLSO-proper-L2-vs-uniform}). 
	\item $K$-theoretically, most previous cases can be proved by Kasparov's Dirac-dual-Dirac method; however, our proof takes a different route since we do not know how to construct a Dirac element for infinite-dimensional curved spaces. 
\end{enumerate}

To overcome these difficulties in our situation, we make use of the following new technical tools in the proof of Theorem~\ref{admissible_hilbert_hadamard_thm}:
\begin{enumerate}
	\item A construction of a $C^*$-algebra $\AofM$ associated to the Hilbert-Hadamard space $M$ (see Section~\ref{sec:AofM} and, in particular, Definition~\ref{defn_AofM}), which generalizes the algebra $\AofM[\hil]$ constructed by Higson and Kasparov \cite{higsonkasparov} for a Hilbert space $\hil$ (cf.~Remark~\ref{rmk:AofM-complemented}) and is analogous to the one constructed by Kasparov and Yu \cite{kasparovyu12} for Banach spaces with property (H).  
	\item A technique of deforming any isometric action on a Hilbert-Hadamard space $M$ into a trivial action on a ``bigger'' (typically infinite\-/dimensional) Hilbert-Hadamard space (Proposition~\ref{prop:isom-01-nilhomotopic}). This deformation technique allows us to show the induced action on the $C^*$-algebra $\AofM$ can be ``trivialized'' in $KK$-theory (see Section~\ref{sec:proof}). 
	We remark that this deformation technique is only accessible in the framework of \emph{infinite\-/dimensional} spaces. 
	\item $KK$-theory with real coefficients, developed recently by Antonini, Azzali, and Skandalis \cite{antoniniazzaliskandalis2014, antoniniazzaliskandalis}. This theory allows us to deal with groups with torsion. See Lemma~\ref{lem:KKR-EGam-inj}.
\end{enumerate}

This paper is organized as follows. In Section~\ref{sec:hhs}, we introduce the notion of (admissible) {\hhs}s and prove that any isometric action on a {\hhs} can be deformed into a trivial action on a ``bigger'' {\hhs}. The latter procedure makes use of the general construction of a continuum product of a {\hhs} over a finite measure space, a topic that we revisit in the appendix (Section~\ref{sec:appendix}) in order to complete a few technical proofs. In Section~\ref{sec:L2metrics}, we focus on our main example of admissible {\hhs}s, namely the space of $L^2$-Riemannian metrics of a smooth manifold with a fixed volume form (or more generally for non-orientable manifolds, a fixed density), and discuss how volume-preserving (or density-preserving) diffeomorphisms give rise to isometries of this {\hhs} and when a group of these diffeomorphisms acts metrically properly on the space of $L^2$-Riemannian metrics. In Section~\ref{sec:AofM}, we construct a noncommutative $C^*$-algebra $\AofM$ from a {\hhs} $M$, which plays a key role in the proof of our main theorem. A discussion of how isometries of $M$ act on $\AofM$ takes place in Section~\ref{sec:automorphisms}, and one about the $K$-theory of $\AofM$ ensues in Section~\ref{sec:K-theory}. Finally, Section~\ref{sec:proof} completes the proof of our main theorem. 

\subsubsection*{Acknowledgement}The authors would like to thank the anonymous referee for the very helpful and detailed comments.

\section{Preliminaries}\label{sec:prelim}

\subsection{{\cstaralg}s}
A basic tool in the operator-theoretic approach to the Novikov conjecture is the notion of a \emph{$C^*$-algebra}, defined to be a Banach space over the field $\cbbd$ of complex numbers, equipped with a compatible product structure and a conjugate-linear involutive self-map, called the \emph{adjoint} or the $*$\-/operation, satifying
\begin{enumerate}
	\item\label{plain:cstaralg-norm-multiplication} $\| a b \| \leq \|a \| \, \|b\|$, 
	\item\label{plain:cstaralg-adjoint-anti} $(ab)^* = b^* a^*$, and
	\item\label{plain:cstaralg-cstar-identity} $\| a a^* \| = \|a\|^2$ (the \emph{$C^*$-identity}) 
\end{enumerate}
for all elements $a$ and $b$. Removing the role of the norm gives us the notion of a $*$\-/algebra. 

A $C^*$-algebra is \emph{unital} if it contains a (necessarily unique) multiplicative {identity}, denoted by $1$. It is called \emph{separable} if it contains a countable dense subset. 
An element $a$ in a {\cstaralg} $A$ is \emph{normal} if $aa^* = a^* a$, \emph{self-adjoint} if $a = a^*$, and \emph{unitary} if $A$ is unital and $a a^* = a^* a = 1$. The \emph{center} ${Z}(A)$ of a $C^*$-subalgebra consisting of elements that commute with every other element in $A$. 
First examples of {\cstaralg}s include:  
\begin{itemize}
	\item $B(\hhil)$, the algebra of bounded linear operators from a complex Hilbert space $\hhil$ into itself, equipped with operator multiplication and the operation of taking Hermitian adjoints, and 
	\item $C_0(X)$ for a locally compact Hausdorff space $X$, consisting of all continuous functions $f$ from $X$ to $\cbbd$ that \emph{vanish at infinity} (i.e., given any $\varepsilon>0$, there is a compact set $K \subset X$ such that $f(x) < \varepsilon$ for $x \in X \setminus K$) equipped with the pointwise multiplication and conjugation. 
\end{itemize}
A fundamental result of Gelfand and Naimark asserts that every $C^*$-algebra can be realized as a closed subalgebra of $B(\hhil)$ that is closed under taking adjoint operators, and moreover, every \emph{commutative} $C^*$-algebra $A$ can also be realized as $C_0(X)$ for an up-to-homeomorphism unique locally compact Hausdorff space $X$, which can be identified with the \emph{dual} $\widehat{A}$ of the commutative $C^*$-algebra $A$, defined to be the space of maximal ideals of $A$ with the \emph{hull-kernel topology}. A consequence of this is that \emph{$*$\-/homomorphisms} between $C^*$-algebras, i.e., maps that preserve multiplication and the $*$\-/operation, are norm semi\-/decreasing. It is also through the lens of this result that the study of noncommutative {\cstaralg}s is sometimes called noncommutative topology. 

Observe that given a normal element $a$ in a {\cstaralg} $A$, the $C^*$-subalgebra \emph{generated by} $a$ is commutative. It can be identified with $C_0(X \setminus \{0\})$ where $X$ is a closed subset of $\cbbd$ called the \emph{spectrum} of $a$, the element $a$ corresponds to the inclusion map $X \hookrightarrow \cbbd$, and $C_0(\varnothing)$ is understood to be $\{0\}$. Given a continuous function $f \colon X \to \cbbd$, we denote its corresponding element in $A$ by $f(a)$. Such a correspondence is called \emph{continuous functional calculus}.

In this paper, we make use of the following elementary constructions of $C^*$-algebras: 
\begin{enumerate}
	\item Generalizing the construction of $C_0(X)$, we define, for a $C^*$-algebra $A$ and a locally compact space $X$, the $C^*$-algebra $C_0(X, A)$ to consist of all continuous functions $f$ from $X$ to $A$ that {vanish at infinity}, equipped with the pointwise algebraic and $*$\-/operations. This construction is covariant in $A$ with respect to $*$\-/homomorphisms and contravariant in $X$ with respect to proper continuous maps, by means of composition of maps. When $X$ is compact, the vanishing condition is vacuous and thus we simply write $C(X,A)$ (or simply $C(X)$ when $A = \cbbd$). 
	\item Fix a discrete group $\Gamma$. Given a $C^*$-algebra $A$, embedded as a subalgebra of $B(\hhil)$, and a (left) \emph{action} $\alpha \colon \Gamma \curvearrowright A$ (i.e., a homomorphism from $\Gamma$ to $\autgrp(A)$, the \emph{group of $*$\-/automorphisms} of $A$) we can construct the \emph{reduced crossed product} $A \rtimes_{\operatorname{r}, \alpha} \Gamma$ as a $C^*$-subalgebra of $B(\ell^2(G) \otimes \hhil)$ generated by 
	\begin{itemize}
		\item the unitaries $u_g \colon \delta_h \otimes \xi \mapsto \delta_{g h} \otimes \xi$, for $g \in \Gamma$, and 
		\item the operators $\lambda_a \colon \delta_h \otimes \xi \mapsto \delta_{h} \otimes \alpha_h^{-1} (a) \cdot \xi$, for $a \in A$, 
	\end{itemize}
	where $\delta_h$ is the Dirac delta function at $h \in \Gamma$ and $\xi \in \hhil$. The generators thus satisfy the \emph{covariance condition} $u_g \lambda_a u_g^* = \lambda_{\alpha_g (a)}$. Up to $*$\-/isomorphism, this construction does not depend on the embedding $A \hookrightarrow B(\hhil)$. This construction is covariant in $A$ with respect to \emph{$\Gamma$-equivariant} $*$\-/homomorphisms.
	On the other hand, we also construct the \emph{maximal crossed product} $A \rtimes_{\alpha} \Gamma$ as the completion of the $*$\-/algebra generated by the $u_g$'s and $\lambda_a$'s by the largest possible $C^*$-norm, that is, the norm of an element is given by supremum of the norms of its images under all possible $*$\-/homomorphisms into $B(\khil)$ for all possible Hilbert spaces $\khil$. This construction is covariant not only in $A$ with respect to \emph{$\Gamma$-equivariant} $*$\-/homomorphisms, but also more generally in both variables with respect to pairs of the form $\left( \varphi, \psi \right)$ where $\varphi$ is a $*$\-/homomorphism, $\psi$ is a group homomorphism and the pair is equivariant in the sense that it intertwines the group actions.
	We often drop $\alpha$ from the notation if the action is understood, and we view $A$ as embedded in $A \rtimes_{\operatorname{r}} \Gamma$ and $A \rtimes \Gamma$ by identifying $\lambda_a$ with $a$. 
	A prominent special case is when $A = \cbbd$, whereby we write $C^*_{\operatorname{r}} \Gamma$ for $\cbbd \rtimes_{\operatorname{r}} \Gamma$ and call it the \emph{reduced group \cstaralg} of $\Gamma$. 
	\item Given a \emph{real} Hilbert space $\hhil$, we write $\cliffc \hhil$ for the complex Clifford $C^*$-algebra generated by $\hhil$. More precisely, consider the \emph{antisymmetric Fock space} 
	\[
	\expwr^*\hil_{\cbbd} := \bigoplus_{k=0}^\infty \expwr^k\hil_{\cbbd}
	\]
	where $\expwr^k\hil_{\cbbd}$, the \emph{$k$-th complex exterior power} of $\hil$, is defined to be the complexification of the quotient of the real tensor vector space $\bigotimes^k\hil$ by equating $\xi_1\otimes\cdots\otimes\xi_k$ with $\sgn(\sigma)\ \xi_{\sigma(1)}\otimes\cdots\otimes\xi_{\sigma(k)}$ for any $k$-permutation $\sigma$, with the equivalence class denoted by $\xi_1\wedge\cdots\wedge\xi_k$. For each $\eta\in\hil$, we may define its \emph{creation operator} $\copr(\eta): \expwr^*\hil_{\cbbd} \to \expwr^*\hil_{\cbbd}$ by
	\[
	\copr(\eta)(\xi_1\wedge\cdots\wedge\xi_k) = \eta\wedge\xi_1\wedge\cdots\wedge\xi_k \; .
	\]
	If we define a self-adjoint operator 
	$$\hat{\eta}:= \copr(\eta)+\copr^*(\eta)$$
	for each $\eta\in\hil$, then we have the relation
	\begin{equation}\label{cliffordrelation}
	\hat{\eta}\hat{\xi}+\hat{\xi}\hat{\eta} =2 \left\langle \eta,\xi \right\rangle
	\end{equation}
	for any $\eta,\xi \in\hil$. In particular,
	\begin{equation}
	\hat{\eta}^2=\|\eta\|^2  \; ,
	\end{equation}
	which is a scalar multiplication. The \emph{(complex) Clifford algebra} $\cliffc(\hil)$ of $\hil$ is the subalgebra of $\bh[\expwr^*\hil_{\cbbd}]$ generated by $\{\hat{\eta}:\ \eta \in\hil\}$. We remark that the assignment $\hhil \mapsto \cliffc \hhil$ is functorial with regard to isometric linear embeddings of Hilbert spaces and $*$\-/homomorphisms and it also preserves direct limits in the respective categories. In particular, the involutive isometry on $\hhil$ that takes each $\xi$ to $-\xi$ induces a distinguished involutive $*$\-/automorphism of $\cliffc(\hhil)$, which turns the latter into a \emph{{graded} $C^*$-algebra}. All $*$\-/homomorphisms induced from isometric linear embeddings of Hilbert spaces also preserve the grading. 
\end{enumerate}

The first construction above provides a prototype for the following notion. Given a locally compact Hausdorff space $X$, we say a {\cstaralg} $B$ is an \emph{$X$-{\cstaralg}} if there is a continuous map from $\widehat{Z(B)}$, the dual of the center of $B$, to $X$. 
If, in addition, a group $\Gamma$ acts on $B$ by automorphisms and on $X$ by homeomorphisms 
and the map $\widehat{Z(B)} \to X$ is equivariant under the induced action of $\Gamma$ on $\widehat{Z(B)}$, then we say $B$ is a \emph{$\Gamma$-$X$-{\cstaralg}}. For example, $C_0(X, A)$ is an $X$-{\cstaralg} and if $\Gamma$ acts on both $X$ and $A$, then $C_0(X, A)$, with the diagonal action $(g \cdot f) (x) = g \cdot f(g^{-1} \cdot x)$, becomes a $\Gamma$-$X$-{\cstaralg}. Given an \emph{$X$-{\cstaralg}} $B$ and $x \in X$, the \emph{fiber} of $B$ at $x$ is given by the quotient $B / \left( C_0(X \setminus \{x\}) \cdot B \right)$. This is useful since, for example, to prove an $X$-$C^*$-subalgebra $B'$ of $B$ is equal to $B$, we just need to check that $B'$ maps surjectively onto each fiber. 

From a noncommutative geometric point of view, a desirable condition for an action of a discrete group $\Gamma$ on a locally compact Hausdorff space $X$ is being \emph{proper}, that is, for any compact subset $K$ in $X$, we have $K \cap g K = \varnothing$ for all but finitely many $g \in \Gamma$. This topological notion of proper actions coincides with the metric notion mentioned before Theorem~\ref{thm:main} in the case of an isometric action on a \emph{proper} metric space, i.e., one where all the closed balls are compact. We say that a $\Gamma$-{\cstaralg} is \emph{proper} if the induced action of $\Gamma$ on the dual of its center is proper.  

Two {\shom}s $\varphi_0, \varphi_1 \colon A \to B$ are called \emph{homotopic} if there is a {\shom} $\psi \colon A \to C([0,1], B)$ such that for any $a \in A$ and $i = 0,1$, the image $\varphi_i (a)$ is the evaluation of $\psi(a)$ at $i \in [0,1]$. Two {\cstaralg}s $A$ and $B$ are called \emph{homotopy-equivalent} if there are {\shom}s $\varphi \colon A \to B$ and $\psi \colon B \to A$, called \emph{homotopy-equivalences}, such that $\psi \circ \varphi$ and $\varphi \circ \psi$ are homotopic to the identity maps on $A$ and $B$, respectively. The notion of \emph{equivariant homotopy equivalence} is defined similarly. 

We refer to reader to \cite{davidson} for a detailed account of the {\cstaralg} theory. 

\subsection{$KK$-theory}
An extremely potent tool in noncommutative geometry, particularly in relation with the Novikov conjecture, is Kasparov's equivariant $KK$-theory (cf.\,\cite{kasparov1, kasparov95}), which associates to a locally compact and $\sigma$-compact group $\Gamma$ and two separable $\Gamma$-$C^*$-algebras $\aalg$ and $\balg$ (meaning that $\Gamma$ acts on them) the abelian group $KK^\Gamma(A, B)$. The group $KK^\Gamma(A, B)$ contains, among other things, elements $[\varphi]$ induced from equivariant {\shom}s $\varphi \colon A \to B$. It is contravariant in $A$ and covariant in $B$, both with respect to equivariant {\shom}s. It is equivariantly homotopy-invariant, stably invariant, preserves equivariant split exact sequences, and satisfies \emph{Bott periodicity}, i.e., there are natural isomorphisms
\[
	KK^\Gamma(A, B) \cong KK^\Gamma(\Sigma^2 A, B) \cong KK^\Gamma(\Sigma A, \Sigma B)  \cong KK^\Gamma(A, \Sigma^2 B) 
\]
where $\Sigma^i A$ stands for $C_0(\rbbd^i,A)$ with $i\in \nbbd$ and $\Gamma$ acting trivially on $\rbbd$. These properties ensure that a short exact sequence $0 \to J \to E \to A \to 0$ of $\Gamma$-{\cstaralg}s and equivariant {\shom}s induces a \emph{six-term exact sequence} in the second variable, and with extra conditions such as that $E$ is a nuclear (in particular, commutative) proper $\Gamma$-{\cstaralg}, 
it also induces a six-term exact sequence in the first variable (see \cite[Appendix]{kasparovskandalis2} and \cite[Chapter~VI]{guentnerhigsontrout}), though this fails in general (\cite{Skandalis1991Le}). When one of the two variables is $\cbbd$, equivariant $KK$-theory recovers 
\begin{itemize}
	\item equivariant $K$-theory: $KK^\Gamma(\cbbd, B) \cong K^\Gamma_0(B)$;
	\item equivariant $K$-homology: $KK^\Gamma(A, \cbbd) \cong K_\Gamma^0(A)$. 
\end{itemize}

\begin{rmk}
	The definition of equivariant $KK$-theory is usually tailored to the theory of \emph{graded {\cstaralg}s}. However, in this paper, when we take the equivariant $KK$-groups of graded {\cstaralg}s, we \emph{disregard their gradings} and treat them as trivially graded. 
\end{rmk}

The most striking feature that gives equivariant $KK$-theory its power is the \emph{Kasparov product}, which gives a group homomorphism
\[
	KK^\Gamma(A, B) \otimes_\zbbd KK^\Gamma(B, C) \to KK^\Gamma(A, C) 
\]
for any three separable $\Gamma$-$C^*$-algebras $A$, $B$, and $C$. The Kasparov product of two elements $x \in KK^\Gamma(A, B)$ and $y \in KK^\Gamma(B, C)$ is often denoted by $x \otimes_B y$. The Kasparov product is associative. Moreover, for an equivariant {\shom} $\varphi \colon A \to B$, the group homomorphisms 
\[
	[\varphi] \otimes_B \colon KK^\Gamma(B, C) \to KK^\Gamma(A, C)  \text{ and } \otimes_A [\varphi] \colon KK^\Gamma(D, A) \to KK^\Gamma(D, B)
\]
induced by taking Kasparov products with $[\varphi] \in KK^\Gamma(A, B)$ 
coincide with the homomorphisms given by the functorial properties of equivariant $KK$-theory. 

When the acting group $\Gamma$ is the trivial group, we simply write $KK(A,B)$ for $KK^\Gamma(A, B)$ and drop the word ``equivariant'' everywhere. There is a forgetful functor from $KK^\Gamma$ to $KK$.

\begin{rmk} \label{rmk:KK-facts-de-equivariantize} 
	In some important special cases, we can turn an equivariant $KK$-group $KK^\Gamma(A, B)$ into a related non-equivariant $KK$-group, which is often much easier to study. 
	\begin{enumerate}
		\item\label{rmk:KK-facts-de-equivariantize:trivial-B} When $\Gamma$ is a countable discrete group and its action on $B$ is trivial, it is immediate from the definition that there is a natural isomorphism $KK^\Gamma(A, B) \cong KK(A \rtimes \Gamma, B)$ where $A \rtimes \Gamma$ is the maximal crossed product. In particular, if $A = C_0(X)$ for a locally compact second countable space $X$ and $\Gamma$ acts freely and properly on $X$, then since $C_0(X) \rtimes \Gamma$ is stably isomorphic to $C_0(X / \Gamma)$, we have a natural isomorphism $KK^\Gamma(C_0(X), B) \cong KK(C_0(X / \Gamma), B)$. 
		\item\label{rmk:KK-facts-de-equivariantize:translation-A} When $\Gamma$ is a countable discrete group and $A = C_0(\Gamma, D)$ with an action of $\Gamma$ by translation on the domain $\Gamma$, there is a natural isomorphism $KK^\Gamma(C_0(\Gamma, D), B) \overset{\cong}{\longrightarrow} KK(D, B)$ given by first applying the forgetful functor and then composing with the embedding $D \cong C(\{1_{\Gamma} \}, D) \hookrightarrow C_0(\Gamma, D)$. 
	\end{enumerate}
	
\end{rmk}

In this paper, we will focus on the case when the first variable $A$ in $KK^\Gamma(A,B)$ is commutative and view the theory as a homological theory on the spectrum of $A$. In fact, we will need a variant of it that may be thought of as homology with $\Gamma$-compact support. Recall that a subset of a topological space $X$, on which $\Gamma$ acts, is called \emph{$\Gamma$-compact} if it is contained in $\{g \cdot x \colon g \in \Gamma, ~ x \in K \}$ for some compact subset $K$ in $X$.

\begin{defn}\label{defn:KK-Gam-compact}
	Given a countable discrete group $\Gamma$, a Hausdorff space $X$ with a $\Gamma$-action, a $\Gamma$-$C^*$-algebra $B$, and $i \in \nbbd$, we write $KK^\Gamma_i(X, B)$ for the inductive limit of the equivariant $KK$-groups $KK^\Gamma \left(C_0(Z ), C_0(\rbbd^i, A) \right)$, where $Z$ ranges over $\Gamma$-invariant and $\Gamma$-compact subsets of $X$ and $A$ ranges over $\Gamma$-invariant separable $C^*$-subalgebras of $B$, both directed by inclusion. 
	
	We write $K^\Gamma_i(X)$ for $KK^\Gamma_i(X, \cbbd)$ and call it the \emph{$\Gamma$-equivariant $K$-homology of $X$ with $\Gamma$-compact supports}. 
\end{defn}

It is clear from Bott periodicity that there is a natural isomorphism $KK^\Gamma_i(X, B) \cong KK^\Gamma_{i+2}(X, B)$. Thus we can view the index $i$ as an element of $\zbbd / 2 \zbbd$. Also note that this construction is covariant both in $X$ with respect to continuous maps and in $B$ with respect to equivariant {\shom}s. Partially generalizing the functoriality in the second variable, the Kasparov product gives us a natural product $KK^\Gamma_i(X, B) \otimes_\zbbd KK^\Gamma(B,C) \to KK^\Gamma_i(X, C)$ for any separable $\Gamma$-{\cstaralg}s $B$ and $C$ (the separability condition can be dropped by extending the definition of $KK^\Gamma(B,C)$ through taking limits). 

We may think of $KK^\Gamma_i(-, B)$ as an extraordinary homology theory in the sense of Eilenberg-Steenrod. In the non-equivariant case, the coefficient algebra $B$ plays a rather minor role in this picture. 

\begin{lem}\label{lem:KK-separate-variables}
	For any CW-complex $X$, any $C^*$-algebra $B$, and any $i \in \zbbd / 2 \zbbd$, there is a natural isomorphism
	\[
		KK_i(X, B) \otimes_{\zbbd} \qbbd \cong \bigoplus_{j \in \zbbd / 2 \zbbd} K_j(X) \otimes_\zbbd K_{i-j}(B)  \otimes_{\zbbd} \qbbd
	\]
\end{lem}

This follows from a version of the K\"{u}nneth Theorem \cite{RosenbergSchochet1987}. For the sake of completeness, we provide a brief proof. 

\begin{proof}
	
	For $j \in  \zbbd / 2 \zbbd$, there are natural homomorphisms 
	\[
		K_j(X) \otimes_\zbbd K_{i-j}(B) \cong KK_j(X, \cbbd) \otimes_\zbbd KK_{i-j}(\cbbd, B) \to KK_i(X, B) 
	\]
	given by the Kasparov product. We claim that taking the direct sum of these two homomorphisms gives us the desired isomorphism after rationalization (i.e., tensoring by $\qbbd$). This is clear when $X$ is a point, since $K_0(X) \cong \zbbd$ and $K_1(X) \cong 0$. The homotopy invariance of $KK$-theory thus generalizes the isomorphism to the case when $X$ is contractible. A standard cutting-and-pasting argument using Mayer-Vietoris sequences and the five lemma then generalizes it to the case when $X$ is a finite CW-complex. Here it becomes clear that the summand with $j=1$ is needed for dimension shifts and rationalization is needed to preserve exactness after taking tensor products with $K_{*}(B)$. The general case follows by taking a direct limit. 
\end{proof}

Given a countable discrete group $\Gamma$, following \cite{baumconneshigson}, we use the term \emph{proper $\Gamma$-space} for a metrizable space $X$ with a proper $\Gamma$-action such that the quotient space is again metrizable. It is a \emph{free and proper $\Gamma$-space} if the action is, in addition, free. Proper $\Gamma$-spaces satisfy the so-called slice theorem. We shall only need the following special case. 

\begin{lem}\label{lem:slice}
	Let $X$ be a free and proper $\Gamma$-space. Then every point $x \in X$ has a neighborhood $U$ such that $\{ g \cdot U \colon g \in \Gamma \}$ is a family of disjoint sets. 
\end{lem}

Let $\univspfree\Gamma$ denote a \emph{universal space} for free and proper $\Gamma$-actions, that is, $\univspfree\Gamma$ is a free and proper $\Gamma$-space such that any free and proper $\Gamma$-space $X$ admits a $\Gamma$-equivariant continuous map into $\univspfree\Gamma$ that is unique up to $\Gamma$-equivariant homotopy. 
Let $B \Gamma$ be the quotient of $\univspfree\Gamma$ by $\Gamma$. 
Similarly, $\univspproper\Gamma$ denotes a \emph{universal space} for proper $\Gamma$-actions. 
These constructions are unique up to ($\Gamma$-equivariant) homotopy equivalence, and thus there is no ambiguity in writing $KK^\Gamma_i(\univspfree\Gamma, B)$, $KK_i(B\Gamma, B)$ and $KK^\Gamma_i(\univspproper\Gamma, B)$ for a $\Gamma$-{\cstaralg} $B$. By definition, there is a $\Gamma$-equivariant continuous map $\univspfree\Gamma \to \univspproper\Gamma$, regardless of the choice of models. 

The \emph{reduced Baum-Connes assembly map} for a countable discrete group $\Gamma$ and a $\Gamma$-{\cstaralg} $B$ is a group homomorphism 
\[
	\mu \colon KK^\Gamma_i(\univspproper\Gamma, B) \to K_i(B \rtimes_{\operatorname{r}} \Gamma) \; .
\]
It is natural in $B$ with respect to $\Gamma$-equivariant {\shom}s or more generally with respect to taking Kasparov products, in the sense that any element $\delta \in KK^\Gamma(B,C)$ induces a commuting diagram
\begin{equation}\label{eq:BC-assembly-natural}
	\xymatrix{
			KK^\Gamma_i(\univspproper\Gamma, B) \ar[r]^\mu \ar[d]^{\delta} & K_i(B \rtimes_{\operatorname{r}} \Gamma) \ar[d]^{\delta \rtimes_{\operatorname{r}} \Gamma} \\
			KK^\Gamma_i(\univspproper\Gamma, C) \ar[r]^\mu & K_i(C \rtimes_{\operatorname{r}} \Gamma)
		}
\end{equation}
for an induced group homomorphism $\delta \rtimes_{\operatorname{r}} \Gamma$.

The case when $B = \cbbd$ is of special interest. The \emph{rational strong Novikov conjecture} asserts that the composition
\[
	K^\Gamma_i(\univspfree\Gamma) \to K^\Gamma_i(\univspproper\Gamma) \overset{\mu}{\to} K_i(C^*_{\operatorname{r}} \Gamma) 
\] 
is injective after tensoring each term by $\qbbd$. It implies the Novikov conjecture,
the Gromov-Lawson conjecture on the nonexistence of positive scalar curvature for aspherical manifolds (cf.\,\cite{rosenberg1983c}) and Gromov's zero-in-the-spectrum conjecture.

On the other hand, it has proven extremely useful to have the flexibility of a general $\Gamma$-algebra $B$ in the picture, largely due to the following key observation, which is based on a theorem of Green \cite{green1982} and Julg \cite{julg1981} and an equivariant cutting-and-pasting argument on $B$. 

\begin{thm}[{cf.\,\cite[Theorem~13.1]{guentnerhigsontrout}}]\label{thm:proper-GHT}
	For any countable discrete group $\Gamma$, and a $\Gamma$-{\cstaralg} $B$, if $B$ is a proper $\Gamma$-$X$-{\cstaralg} for some locally compact Hausdorff space $X$, then the reduced Baum-Connes assembly map 
	\[
		\mu \colon KK^\Gamma_i(\univspproper\Gamma, B) \to K_i(B \rtimes_{\operatorname{r}} \Gamma) \; .
	\]
	is a bijection. \qed
\end{thm}

This is the basis of the \emph{Dirac-dual-Dirac} method (cf.\,\cite{kasparov1,kasparov95}; also see \cite[Chapter~9]{Valette2002}), which was applied very successfully to the study of the Baum-Connes assembly map. It is based on the construction of a proper $\Gamma$-$X$-{\cstaralg} $B$ together with $KK$-elements $\alpha \in KK^\Gamma(B, \cbbd)$ and $\beta \in KK^\Gamma(\cbbd, B)$ such that $\beta \otimes_B \alpha$ is equal to the identity element in $KK^\Gamma(\cbbd, \cbbd)$. When this is possible, Theorem~\ref{thm:proper-GHT} allows us to conclude that the Baum-Connes assembly map for $\Gamma$ is an isomorpism and the rational strong Novikov conjecture for $\Gamma$ follows. 
Although we do not directly apply this method to prove Theorem~\ref{thm:main}, our strategy still calls for a proper $\Gamma$-$X$-{\cstaralg} $B$ and a $KK$-element $\beta \in KK^\Gamma(\cbbd, B)$. 

It is not hard to see that whenever $\Gamma$ is infinite and $B$ is a proper $\Gamma$-$X$-{\cstaralg}, there is no $\Gamma$-equivariant {\shom} from $\cbbd$ to $B$. Thus one must look beyond $\Gamma$-equivariant {\shom}s in order to construct a suitable element $\beta \in KK^\Gamma(\cbbd, B)$. Many of such elements come from \emph{$\Gamma$-equivariant asymptotic morphisms} (cf.\,\cite{ConnesHigson1990, guentnerhigsontrout}). We will only make use of a special type of such morphisms, given below. 

\begin{constr}	\label{constr:KK-facts-asymptotic} 	
	Let $B$ be a $\Gamma$-$C^*$-algebra and let $\varphi_t \colon \salg \to B$ be a family of $*$\-/homomorphisms indexed by $t \in [1, \infty)$ that is
	\begin{enumerate}
		\item \emph{pointwise continuous}, i.e., $t \mapsto \varphi_t(f)$ is continuous for any $f \in \salg$, and 
		\item \emph{asymptotically invariant}, i.e., $\lim_{t \to \infty} \left\| g \cdot \left(\varphi_t(f)\right) - \varphi_t(f) \right\| = 0$ for any $f \in \salg$ and any $g \in \Gamma$. 
	\end{enumerate}
	Then by \cite[Definition~7.4]{higsonkasparov}, there is an element 
	\[
	\left[ (\varphi_t) \right] \in KK^\Gamma_0 \left(C_0(\rbbd), B \right) \cong  KK^\Gamma_1(\cbbd, B)
	\]
	whose image under the forgetful map 
	\[
	KK^\Gamma_0(\cbbd, C_0(\rbbd, B)) \to KK(\cbbd, C_0(\rbbd, B)) \cong KK(\salg, B)
	\]
	is equal to the element $[\varphi_t]$ induced by the homomorphism $\varphi_t$, for any $t \in [0, \infty)$. 
\end{constr}

To conclude our preparation of equivariant $KK$-theory, we recall the construction of equivariant $KK$-theory with real coefficients, recently introduced by Antonini, Azzali and Skandalis. 

\begin{constr}[cf.\,\cite{antoniniazzaliskandalis2016}] \label{constr:KKR}
	The \emph{equivariant $KK$-theory with real coefficients} is a bivariant theory that associates, to each pair $(A,B)$ of $\Gamma$-$C^*$-algebras, the groups  
	\[
	KK^\Gam_{\rbbd} (A, B) = \varinjlim_{N} KK^\Gam (A, B \otimes N)
	\]
	where $\otimes$ stands for the minimal tensor product and the inductive limit is taken over all II$_1$-factors $N$ with unital $*$\-/homomorphisms as connecting maps. This theory is contravariant in the first variable and covariant in the second, and there is a natural map from $KK^\Gam(A, B) \otimes_{\zbbd} \rbbd$ to $KK^\Gam_{\rbbd} (A, B)$ since $K_0(N) \cong \rbbd$ for any II$_1$-factor $N$. This map is an isomorphism when $\Gamma$ is trivial, $A = \cbbd$ and $B$ is in the bootstrap class (i.e., the class $\mathcal{N}$ in \cite{RosenbergSchochet1987}). Moreover, the Kasparov product extends to this theory. 
	
	Given a discrete group $\Gamma$, a Hausdorff space $X$ with a $\Gamma$-action, and a $C^*$-algebra $B$ with a $\Gamma$-action, we define $KK^\Gamma_{\rbbd,*}(X, B)$ in the same way as in Construction~\ref{defn:KK-Gam-compact}. Then the universal coefficient theorem allows us to identify $KK_{\rbbd, *} (X, \cbbd)$ with $K_*(X) \otimes_{\zbbd} \rbbd$ in a natural way. 
\end{constr}

The key reason we consider $KK$-theory with real coefficients is the following convenient fact. 

\begin{lem} \label{lem:KKR-EGam-inj}
	For any discrete group $\Gam$ and $\Gam$-$C^*$-algebra $A$, the homomorphism 
	\[
	\pi_* \colon KK^\Gam_{\rbbd, *} (\univspfree\Gam, A) \to KK^\Gam_{\rbbd, *} (\univspproper\Gam, A) \; ,
	\]
	which is induced by the natural $\Gam$-equivariant continuous map $\pi \colon \univspfree\Gam \to \univspproper\Gam$, is injective. 
\end{lem}
\begin{proof}
	It follows from \cite[Section~5]{antoniniazzaliskandalis} that the above homomorphism gives rise to an isomorphism between $KK^\Gam_{\rbbd, *} (\univspfree\Gam, A)$ and $KK^\Gam_{\rbbd, *} (\univspproper\Gam, A) _\tau$, which is a subgroup of $KK^\Gam_{\rbbd, *} (\univspproper\Gam, A)$ called its $\tau$-part. 
\end{proof}

\subsection{Metric geometry}

Let us give some background on metric geometry, in particular concerning 
the tangent cone of a metric space and properties of CAT(0) spaces. 
This is a summary of some of the material in \cite{burago}, most notably, Sections~3.6.2, 3.6.5, 3.6.6, and~9.1.8, as well as \cite[Part II]{BridsonHaefliger1999Metric}. 

Let $(X,d)$ be a metric space. The \emph{length} of a continuous path $\alpha: [a,b] \to X$, denoted by $|\alpha|$, is the supremum of 
\[
	\sum_{i=1}^{n} d(\alpha(t_1), \alpha(t_n))
\]
where $(t_1, \ldots, t_n)$ ranges over all finite tuples in $[a,b]$ with $t_1 \leq t_2 \leq \ldots t_n$. Such a path is called a \emph{geodesic segment} if $|\alpha| = d(\alpha(a), \alpha(b))$. The metric space $(X,d)$ is called a \emph{geodesic space} if any two points are connected by a geodesic segment. It is called a \emph{uniquely geodesic space} if any two points are connected by a unique geodesic segment, up to monotonous reparametrization of the domain. 
For any two points $x$ and $y$ in a unique geodesic space $X$, we write 
\begin{equation}\label{eq:notation-geodesic}
	[x,y] \colon [0,1] \to X
\end{equation} 
for the unique affinely parametrized geodesic segment connecting $x$ to $y$, i.e., $[x,y](0) = x$ and $[x,y](1) = y$, and $d(x, [x,y](t)) = t d(x,y)$. If there is no confusion, we also use $[x,y]$ to denote the image of this geodesic segment. 


A \emph{geodesic triangle} in a metric space consists of three geodesic segments, every two of which share a common endpoint. Given a geodesic triangle, an \emph{(Euclidean) comparison triangle} is a triangle in $\rbbd^2$ whose three sides have the same lengths as the three geodesic segments. Up to congruence, the comparison triangle is uniquely defined and depends only on the distances between the three endpoints on the original geodesic triangle. Thus we often write $\widetilde{pqr}$ for the comparison triangle of a geodesic triangle whose three endpoints are $p$, $q$, and $r$. 

A geodesic metric space $(X,d)$ is said to be \emph{CAT(0)} if for any points $x$ and $y$ on a geodesic triangle $\bigtriangleup$ in $X$, if $\widetilde{\bigtriangleup}$ is its Euclidean comparison triangle and $\widetilde{x}$ and $\widetilde{y}$ are the points on $\widetilde{\bigtriangleup}$ corresponding to $x$ and $y$, respectively (i.e., the distances from $x$ to its two adjacent endpoints are the same as those from $\widetilde{x}$ to its two adjacent vertices on the comparison triangle, and the same for $y$), then we have 
\[
	d\left( \widetilde{x}, \widetilde{y} \right) \leq d(x,y) \; .
\]
Intuitively, this says ``every geodesic triangle is thinner than its Euclidean comparison triangle''. A less intuitive but very useful equivalent definition is the following. 

\begin{rmk}\label{rmk:CAT0-equiv-defn-Bruhat-Tits}
	A metric space $(X,d)$ is {CAT(0)} if and only if for any $p,q,r,m \in X$ satisfying $d(q,m) = d(r,m) = \frac{1}{2}d(q,r)$, we have 
	\[
		d(p,q)^2 + d(p,r)^2 \geq 2 d(m,p)^2 + \frac{1}{2} d(q,r)^2 \; .
	\]
	This is the \emph{CN inequality} of Bruhat and Tits \cite{BruhatTits1972Groupes}, also called the \emph{semi parallelogram law} (cf.\,\cite[XI, {\S}3]{lang}).
\end{rmk}

\begin{rmk}\label{rmk:CAT0-facts}
	Here are some facts about CAT(0) spaces. Let $X$ be a CAT(0) space. 
	\begin{enumerate}
		\item\label{rmk:CAT0-facts-unique-geodesic} The metric space $X$ is uniquely geodesic. 
		\item\label{rmk:CAT0-facts-bicombing} The map  
		\[
		X \times X \times [0,1] \to X \, , \qquad (x, y, t) \mapsto [x, y](t )
		\]
		is continuous and is referred to as the \emph{geodesic bicombing}. It follows that $X$ is contractible. 
		\item\label{rmk:CAT0-facts-Lipschitz} For any $x,y,x',y'$ in $X$, we have
		\[
		d \big( [x, y](t ) , [x', y'](t ) \big) \leq \max\{ d(x,y), d(x',y')\} \; .
		\] 
	\end{enumerate}
\end{rmk}

Examples of CAT(0) spaces include Hilbert spaces, trees, and the so-called \emph{Hadamard manifolds}, i.e., complete connected and simply connected Riemannian manifolds with non-positive sectional curvature, e.g., hyperbolic spaces $\mathbb{H}^n$ and Riemannian symmetric spaces of noncompact type. This terminology comes from the following fundamental theorem. 

\begin{thm}[{Cartan-Hadamard Theorem; cf., e.g., \cite[XI, \S 3]{lang}}]\label{thm:Cartan-Hadamard}
	Given a Hadamard manifold $M$ and a point $x_0 \in M$, the exponential map $\exp_{x_0}$ from the tangent space $T_{x_0} M$ to $M$ is a diffeomorphism and is metric semi\-/increasing, that is, $d(\exp_{x_0} (v), \exp_{x_0} (w)) \geq \| v - w \|$ for any $v, w \in T_{x_0} M$. 
\end{thm}

\vspace{.2cm}

Next we review the notions of angle and tangent cone. Let $(X,d)$ be a geodesic metric space. For three distinct points $x,y,z \in X$, we define the comparison angle $\widetilde{\angle} xyz$ to be the angle at $\widetilde{y}$ of the Euclidean comparison triangle $\widetilde{xyz}$. More explicitly, we have
\[\widetilde{\angle} xyz = \arccos\left(\frac{d(x,y)^2 + d(y,z)^2 - d(x,z)^2}{2d(x,y)d(y,z)}\right).\]

Given two nontrivial geodesic segments $\alpha$ and $\beta$ emanating from a point $p$ in $X$, meaning that $\alpha(0) = \beta(0) = p$, we define the angle between them, $\angle(\alpha,\beta)$, to be 
\[\angle(\alpha,\beta) = \lim_{s,t \to 0} \widetilde{\angle} (\alpha(s),p,\beta(t)) \; ,\]
provided that the limit exists. 
By \cite[Theorem~3.6.34]{burago}, angles satisfy the triangle inequality. 

\begin{rmk}\label{rmk:CAT0-equiv-defn-thin-triangle}
	Using the notion of comparison angles, we get another equivalent definition for CAT(0) spaces. Namely, a metric space $(X,d)$ is \emph{CAT(0)} if it is geodesic and for any points $p,q,r, x,y$ in $X$ with $x$ on a geodesic segment connecting $p$ and $q$ and $y$ on a geodesic segment connecting $p$ and $r$, we have
	\[
		\widetilde{\angle} xpy \leq \widetilde{\angle} qpr \; . 
	\]
	It follows from this definition that in a CAT(0) space, the angle between any two nontrivial geodesic segments emanating at the same point exists. 
\end{rmk}

Now suppose the geodesic metric space $(X,d)$ satisfies that the angle between any two nontrivial geodesic segments emanating at the same point exists. 
For a point $p \in X$, let $\Sigma_p'$ denote the metric space consisting of all equivalence classes of geodesic segments emanating from $p$, where two geodesic segments are identified if they have zero angle and the distance $d([\alpha],[\beta])$ between two classes of geodesic segments is the angle $\angle(\alpha,\beta)$. 
Note, in particular, from our definition of angles, that $d([\alpha],[\beta]) \leq \pi$ for any geodesic segments $\alpha$ and $\beta$ emanating from $p$.
Let $\Sigma_p$ denote the completion of $\Sigma_p'$. 

The \textit{tangent cone} $T_p$ at a point $p$ in $X$ is then defined to be a metric space which is, as a topological space, the cone of $\Sigma_p$, that is,
\[\Sigma_p \times [0,\infty)/\Sigma_p \times \{0\} \;. \]
The metric on it is given as follows. For two points $p,q \in T_p$ we can express them as $p=[(\alpha,t)]$ and $q = [(\beta,s)]$. Then the metric is given by
\[d(p,q) = \sqrt{t^2+s^2-2st\cos(d([\alpha],[\beta]))} \; . \]
In other words, it is what the distance would be if we went along straight lines in a Euclidean plane with the same angle between them as the angle between the corresponding directions in $X$. A key motivation for this definition is that when $X$ is a Riemannian manifold, this construction of the tangent cone at a point recovers the tangent space equipped with the metric induced by the inner product. 

We remark that some authors do not take the completion when talking about the space of directions. This does not affect our main definition (\ref{defn:hhs}). 

Observe that if $\varphi \colon X \to Y$ is an isometry between two such geodesic metric spaces $X$ and $Y$, then for any $x \in X$, it induces an isometry 
\[
	D_x \varphi \colon T_x \to T_{\varphi(x)} \, , \qquad [(\alpha, t)] \mapsto [( \varphi\circ \alpha, t )] \; ,
\]
which we can think of as the derivative of $\varphi$ at $x$. 
This association is functorial, i.e., it is compatible with composition of isometries. 

\vspace{.2cm}


We conclude our preparation in metric geometry with a discussion of isometric embeddings into Hilbert spaces. 

\begin{constr}\label{constr:Hilbert-space-span}
	It is well known from the work of Schoenberg \cite{schoenberg38} that a metric space $(X,d)$ embeds isometrically into a Hilbert space if and only if the bivariant function $(x_1, x_2) \mapsto \left(d(x_1, x_2)\right)^2$ is a conditionally negative-type kernel. Given such a metric space $(X,d)$ and a fixed base point $x_0 \in X$, there is a canonical way to construct the smallest Hilbert space that contains it with $x_0$ being the origin. See, for example, \cite[Proposition~3.1]{higsonguentner}. More precisely, we define $\hhil_{X, d, x_0}$, the \emph{Hilbert space spanned by $(X,d)$ centered at $x_0$}, to be the completion of the real vector space $\rbbd_0 [X]$, which consists of formal finite linear combinations of elements in $X$ whose coefficients sum up to zero, under the pseudometric induced from the positive semidefinite blinear form 
	\[
	\left\langle \sum_{x \in X} a_x x, \sum_{y \in X} b_y y \right\rangle = -\frac{1}{2} \sum_{x, y \in X} a_x b_y \left(d(x, y)\right)^2 \; .
	\]
	Here a completion under a pseudometric is meant to also identify elements of zero distance to each other. There is a canonical isometric embedding from $(X,d)$ into $\hhil_{X, x_0, d}$ that maps each $x \in X$ to the linear combination $x - x_0$. Given an isometric embedding from $(X, d)$ to another metric space $(Y,d')$ that maps $x_0$ to $y_0$, there is a unique isometric linear embedding from $\hhil_{X, d, x_0}$ to $\hhil_{Y, d', y_0}$ that intertwines the canonical embeddings. It is straightforward to see that these assignments form a functor from the category of pointed metric spaces and isometric base-point-fixing embeddings to the category of Hilbert spaces and linear isometric embeddings. 
\end{constr}

\section{{\hhs}s and deformation of isometric actions}\label{sec:hhs}

In this section, we introduce a class of metric spaces that we call \emph{{\hhs}s}, and we prove that any isometric action on a {\hhs} can be deformed into a trivial action on a ``bigger'' {\hhs}. 

The concept of {\hhs}s is inspired by \cite[Page~2]{fishersilberman}.
Roughly speaking, this is a class of (possibly infinite\-/dimensional) non-positively curved spaces. The deformation result for isometric actions plays an essential role in the proof of our main theorem. The ``bigger'' {\hhs} in this deformation result is obtained by a general construction 
called the continuum product (cf.\,\cite[Page~3]{fishersilberman}). We explain in detail how continuum products provide new examples of typically infinite\-/dimensional {\hhs}s.

\begin{defn}\label{defn:hhs} 
	A \emph{\hhs} is a complete geodesic CAT(0) metric space (i.e., a Hadamard space) all of whose tangent cones are isometrically embeddable into Hilbert spaces. 
	
	For any point $x$ in a {\hhs} $X$, we define the \emph{tangent Hilbert space} $\hhil_x M$ to be the Hilbert space $\hhil_{T_x M}$ spanned by the tangent cone $T_x M$ such that the origin is at the tip of the cone, following Construction~\ref{constr:Hilbert-space-span}.
\end{defn}

We mostly focus on \emph{separable} {\hhs}s, i.e., those that contain countable dense subsets.

\begin{eg}\label{ex:Riemannian-Hilbertian}
	A Riemannian manifold without boundary is a {\hhs} if and only if it is complete, connected, and simply connected, and has non-positive sectional curvature. The same statement holds for \emph{Riemannian-Hilbertian manifolds} (cf.\,\cite{lang}), which are a kind of infinite\-/dimensional generalizations of Riemannian manifolds defined using charts which are open subsets in Hilbert spaces, instead of Euclidean spaces, in a way that a large part of differential geometry, including sectional curvatures, still makes sense. To see why the statement holds, observe that in this case, every tangent cone is itself a Hilbert space, and the equivalence between the CAT(0) condition and being connected, simply connected and non-positively curved follows from \cite[XI, Proposition~3.4 and Theorem~3.5]{lang}.  
\end{eg}

\begin{constr}\label{constr:log-map}
	A CAT(0) space $X$ is always uniquely geodesic. For any $x_0 \in X$, using the notation in Equation~\eqref{eq:notation-geodesic}, we define the \emph{logarithm function at $x_0$} by
	\[
		\log_{x_0} \colon X \to T_{x_0} X \; , \quad x \mapsto [([x_0, x], d(x_0, x))] \; .
	\]
	The CAT(0) condition (e.g., Remark~\ref{rmk:CAT0-equiv-defn-thin-triangle}) implies $\log_{x_0}$ is \emph{non-expansive} (also called \emph{weakly contractive} or \emph{short} by some authors), 
	i.e., 
	\[
		d\left(  \log_{x_0} (x) ,   \log_{x_0} (x') \right) \leq d(x, x')
	\]
	for any $x, x' \in X$ and, in particular, continuous. Moreover, it preserves the metric on each geodesic emanating from $x_0$, that is, 
	\[
		d\left(  \log_{x_0} (x_0) ,   \log_{x_0} (x) \right) = d(x_0, x)
	\]
	for any $x \in X$. 
\end{constr}

\begin{eg}\label{eg:logarithm-map-Riemannian}
	When $M$ is a complete, connected and simply connected Riemannian manifold with non-positive sectional curvature, the logarithm map as defined above is the inverse to the \emph{exponential map} 
	\[
		\exp_{\xi_0} \colon T_{x_0} X \to X 
	\]
	in Riemannian geometry. In this case, both maps are also diffeomorphisms by the Cartan-Hadamard theorem (cf.~\ref{thm:Cartan-Hadamard}).
\end{eg}

Recall that a subset of a geodesic metric space is called \emph{convex} if it is again a geodesic metric space when equipped with the restricted metric. We observe that a closed convex subset of a {\hhs} is itself a {\hhs}.

\begin{defn}\label{defn:hhs-admissible}
	A separable {\hhs} $M$ is called \emph{admissible} if there is an increasing sequence of closed convex subsets isometric to finite\-/dimensional Riemannian manifolds, whose union is dense in $M$. 
\end{defn}

\begin{rmk}\label{rmk:hhs-admissible}
	In Definition~\ref{defn:hhs-admissible}, observe that by Example~\ref{ex:Riemannian-Hilbertian}, each closed convex subset in this sequence is isometric to a Hadamard manifold, i.e., a complete, connected, and simply connected manifold with non-positive sectional curvature. By the classical Cartan-Hadamard theorem, the logarithm map at any point provides a diffeomorphism between the manifold and the corresponding tangent space. 
\end{rmk}

\begin{eg}
	Apart from finite\-/dimensional Hadamard manifolds, examples of admissible {\hhs}s include separable Hilbert spaces. More examples can be obtained by the continuum product construction we are about to discuss. 
\end{eg}

The notion of {\hhs}s is more general than Example~\ref{ex:Riemannian-Hilbertian}, due to the following construction. 

\begin{constr}[cf.\,\cite{fishersilberman}] \label{constr:continuum-product}
	Let $X$ be a metric space. Let $(Y, \mu)$ be a measure space with $\mu(Y) < \infty$. The \emph{($L^2$-)continuum product} of $X$ over $(Y, \mu)$ is the space $L^2(Y,\mu,X)$ of equivalence classes of measurable maps $\xi$ from $Y$ to $X$ satisfying 
	\[
		\int_Y d_X(\xi(y),x_0)^2 \, d\mu(y) < \infty \; ,
	\] 
	where $x_0$ is a fixed point in $X$ and two functions are identified if they differ only on a measure-zero subset of $Y$. It follows from the triangle inequality that the above condition does not depend on the choice of $x_0$. Moreover, the Minkowski inequality implies that the formula 
	\[
		d(\xi, \eta) = \left( \int_Y d_X(\xi(y),\eta(y))^2 \, d\mu(y) \right)^{\frac{1}{2}}
	\]
	defines a metric on $L^2(Y,\mu,X)$. 
\end{constr}

\begin{rmk}\label{rmk:continuum-product-functoriality}
	The continuum product construction is functorial in the following sense: given an isometric embedding $X_1 \to X_2$ of metric spaces and a measurable map $(Y_2, \mu_2) \to (Y_1, \mu_1)$ that is also measure-preserving (i.e., the push-forward of $\mu_2$ is equal to $\mu_1$), composition of maps induces an isometric embedding $L^2(Y_1,\mu_1,X_1) \to L^2(Y_2,\mu_2,X_2)$. 
\end{rmk}

\begin{eg}\label{eg:continuum-product-manifolds}
	As a first example, we consider the case when $Y$ is a finite set $\{y_1, \ldots, y_n \}$ with $\mu$ defined on every subset of $Y$. Notice that $\mu$ is determined by the weights $\mu(\{y_i\})$ for all $y_i \in Y$. Let $\operatorname{supp} \mu$ be the support of $\mu$, that is, the set of $y_i$ such that $\mu(\{y_i\}) > 0$. Then the continuum product $L^2(Y,\mu,X)$ is nothing but the Cartesian product $X^{\operatorname{supp} \mu}$, equipped with a weighted $\ell^2$-metric. 
	
	In particular, when $X$ is a Riemannian manifold, then so is $L^2(Y,\mu,X)$, where for any point $\xi \in L^2(Y,\mu,X)$, the tangent space $T_{\xi} L^2(Y,\mu,X)$ is canonically identified with $\displaystyle \bigoplus_{y_i \in \operatorname{supp} \mu} T_{\xi(y_i)} X$ as vector spaces and the Riemannian metric on $L^2(Y,\mu,X)$ at $\xi$ is given, under the above identification, by the weighted sum
	\[
		T_{\xi} L^2(Y,\mu,X) \times T_{\xi} L^2(Y,\mu,X) \ni (v, w) \mapsto \sum_{y_i \in \operatorname{supp} \mu} \mu(\{y_i\}) \cdot g_{\xi(y_i)} (v(y_i), w(y_i) ) \; ,
	\] 
	where $g_{\xi(y_i)} (-,-)$ is the Riemannian metric on $X$ at $\xi(y_i)$. 
\end{eg}

\begin{prop}\label{prop:continuum-product-CAT0}
	For any CAT(0) space $X$ and measure space $(Y, \mu)$, the continuum product $L^2(Y,\mu,X)$ is again a CAT(0) space. 
\end{prop}

\begin{proof}
	We are going to check the CN inequality of Bruhat and Tits (cf.~Remark~\ref{rmk:CAT0-equiv-defn-Bruhat-Tits}). Given any $\xi, \eta, \theta, \lambda \in L^2(Y,\mu,X)$ satisfying $d(\eta, \lambda) = d(\theta, \lambda) = \frac{1}{2}d(\eta, \theta)$, we first observe that 	these equalities imply that the Minkowski inequality
	\[
		d(\eta, \lambda) + d(\theta, \lambda) \geq \left( \int_Y \left( d_X(\eta(y),\lambda(y)) + d_X(\theta(y),\lambda(y)) \right)^2 \, d\mu(y) \right)^{\frac{1}{2}} 
	\]
	must reach equality, which happens if and only if there is a nonzero vector $(u,v) \in \rbbd^2$ such that $u \, d_X(\eta(y),\lambda(y)) = v \, d_X(\theta(y),\lambda(y))$ for almost every $y$ in $Y$. It follows that 
	\[
		d_X(\eta(y), \lambda(y)) = d_X(\theta(y), \lambda(y)) = \frac{1}{2}d_X(\eta(y), \theta(y)) 
	\]
	for almost every $y$ in $Y$. Hence the CN inequality for $X$ states that 
	\[
		d_X(\xi(y), \eta(y))^2 + d_X(\xi(y), \theta(y))^2 \geq 2 d_X(\lambda(y), \xi(y) )^2 + \frac{1}{2} d_X(\eta(y), \theta(y))^2 \; .
	\] 
	for almost every $y$ in $Y$. Integrating this over $Y$ yields the CN inequality for $L^2(Y,\mu,X)$. 
\end{proof}

\begin{rmk}\label{rmk:continuum-product-geodesics}
	For any CAT(0) space $X$ and measure space $(Y, \mu)$, an argument similar to that in the first half of the proof of Proposition~\ref{prop:continuum-product-CAT0} shows that for any distinct $\xi$ and $\eta$ in $L^2(Y,\mu,X)$, the unique geodesic segment $[\xi, \eta]$ is given by 
	\[
		[\xi, \eta](t)(y) = [\xi(y), \eta(y)] \left( t \right)   
	\]
	for almost every $y$ in $Y$.  
\end{rmk}

As hinted above, the class of {\hhs}s is also closed under taking continuum products. To streamline our presentation, we place the somewhat technical proofs in the appendix and merely summarize the main results here. 

Recall that a measure space $(Y,\mu)$ is called \emph{separable} if there is a countable family $\{A_n \colon n \in \nbbd \}$ of measurable subsets such that for any $\varepsilon >0$ and any measurable subset $A$ in $Y$, we have $\mu(A \bigtriangleup A_n) < \varepsilon$ for some $n$. For example, it is easy to see that any outer regular finite measure on a separable metric space is separable. This includes, in particular, any measure induced from a density on a closed smooth manifold.

\begin{prop}\label{prop:continuum-product-hhs-summary}
	Let $M$ be a {\hhs} and $(Y, \mu)$ be a finite measure space. Then
	\begin{enumerate}
		\item the continuum product $L^2(Y,\mu,M)$ is again a {\hhs}; 
		\item if $(Y, \mu)$ is separable and $M$ is admissible (respectively, separable), then $L^2(Y,\mu,M)$ is also admissible (respectively, separable).  
	\end{enumerate}
\end{prop}

\begin{proof}
	See Propositions~\ref{prop:continuum-product-hhs}, \ref{prop:continuum-product-afdhhm} and~\ref{prop:continuum-product-separable-separable}. 
\end{proof}

We conclude this section with a discussion of the group of isometries of a {\hhs}.

\begin{defn}\label{defn:isom-M}
	Let $M$ be a {\hhs}. We denote by $\isomgrp(M)$ the group of all isometries of $M$, equipped with the topology of pointwise convergence, namely, the weakest topology such that the orbit maps 
	\[
		\isomgrp(M)	 \to M \; , \quad \varphi \mapsto \varphi \cdot x \; ,
	\]
	for $x \in M$, are continuous.  
\end{defn}

\begin{rmk}\label{rmk:length-function-orbit}
	A helpful tool in the study of $\isomgrp(M)$ are the length functions
	\[
		\ell_x \colon \isomgrp(M) \to [0, \infty) \quad , \qquad \varphi \mapsto d_M(x , \varphi  (x) )
	\] 
	for $x \in M$. Their usefulness is reflected in the following properties: 
	\begin{enumerate}
		\item\label{rmk:length-function-orbit:topology} The topology on $\isomgrp(M)$ is the weakest one that makes every $\ell_x$ continuous. 
		\item\label{rmk:length-function-orbit:base-point} Triangle inequality implies that $| \ell_x (\varphi) - \ell_{x'} (\varphi) | \leq 2 d_M(x , x')$. Thus for any dense subset $D$ in $M$, the topology on $\isomgrp(M)$ is the weakest one that makes $\ell_x$ continuous for all $x$ in $D$. 
		\item\label{rmk:length-function-orbit:proper} A subgroup of $\isomgrp(M)$, when viewed as a discrete group, acts on $M$ metrically properly, i.e.,  $d(x, g \cdot x) \to \infty$ as $g \to \infty$ for some (equivalently, all) $x$ in $X$, if and only if for some (equivalently, for all) $x \in M$, the restriction of $\ell_x$ on the subgroup is a proper function, i.e., the preimage of any compact set is again compact (or rather finite in the case of a discrete group). 
	\end{enumerate}
\end{rmk}

Finally, we discuss a construction that serves as the base for a deformation technique we will use in Section~\ref{sec:proof}. 

\begin{constr}\label{constr:embedding-01}
	Let $M$ be a {\hhs} and let $(Y,\mu)$ be a finite measure space. Then $\isomgrp(M)$ embeds into $\isomgrp(L^2(Y,\mu,M))$ canonically by composition of maps, i.e., for any $\varphi \in \isomgrp(M)$, we define $\varphi^{(Y, \mu)} \in \isomgrp(L^2(Y,\mu,M))$ by 
	\[
		\varphi^{(Y, \mu)} ( \xi ) (y) = \varphi ( \xi (y) ) 
	\] 
	for any $\xi \in L^2(Y,\mu,M)$ and $ y \in Y$. 
	
	We will focus on the case when $(Y, \mu)$ is given by the unit interval $[0,1]$ equipped with the Lebesgue measure. In this case, we write $M^{[0,1]}$ for $L^2([0,1],m,M)$ and $\varphi^{[0,1]}$ for $\varphi^{(Y, \mu)}$. 
\end{constr} 

\begin{lem}\label{lem:embedding-01-length}
	Let $M$ be a {\hhs} and let $(Y,\mu)$ be a measure space. Then for any $\varphi \in \isomgrp(M)$ and $\xi \in L^2(Y,\mu,M)$, we have 
	\[
		\ell_\xi \left( \varphi^{(Y, \mu)} \right) = \left( \int_{Y} \left( \ell_{\xi(y)} (\varphi) \right)^2 d \mu(y) \right)^{\frac{1}{2}} \; .
	\]
\end{lem}

\begin{proof}
	This is an immediate consequence of the definition of the length function in Remark~\ref{rmk:length-function-orbit} and the definition of the metric of the continuum product in Construction~\ref{constr:continuum-product}. 
\end{proof}

The following two propositions tell us that the canonical embedding 
	\[
	\isomgrp(M) \hookrightarrow \isomgrp(M^{[0,1]})
	\]
makes the topological aspect of $\isomgrp(M)$ more tractable while keeping the large-scale behavior intact. 

\begin{prop}\label{prop:isom-01-nilhomotopic}
	Let $M$ be a {\hhs}. Then there is a homotopy of group homomorphisms connecting the canonical embedding
	\[
		\isomgrp(M) \hookrightarrow \isomgrp(M^{[0,1]})
	\]
	to the trivial homomorphism, that is, there is a continuous map 
	\[
		H \colon \isomgrp(M) \times [0,1] \to \isomgrp(M^{[0,1]})
	\]
	such that
	\begin{enumerate}
		\item $\varphi \mapsto H(\varphi, t)$ is a group homomorphism for each $t \in [0,1]$, 
		\item $H(\varphi, 0) = \idmap$ for any $\varphi \in \isomgrp(M) $, and
		\item $H(\varphi, 1) (\xi) (s) = \varphi (\xi(s))$ for any  $\varphi \in \isomgrp(M) $, $\xi \in M^{[0,1]}$ and $s \in [0,1]$. 
	\end{enumerate}

\end{prop}

\begin{proof}
	Define 
	\[
		H \colon \isomgrp(M) \times [0,1] \to \isomgrp(M^{[0,1]})
	\]
	by 
	\[
		H(\varphi, t) (\xi) (s) = 
		\begin{cases}
			\varphi(\xi(s))  \; , & s \in [0, t] \\
			\xi(s) \; , & s \in (t, 1]
		\end{cases} 
	\]
	for any $\varphi \in \isomgrp(M) $, $t \in [0,1]$, $\xi \in M^{[0,1]}$ and $s \in [0,1]$. It is straightforward to check the three conditions above. Moreover, by Lemma~\ref{lem:embedding-01-length}, we have
	\[
		\ell_\xi \left( H(\varphi, t) \right) = \left( \int_0^t \left( \ell_{\xi(s)} (\varphi) \right)^2 d s \right)^{\frac{1}{2}} 
	\]
	for any $\varphi \in \isomgrp(M)$, $t \in [0,1]$ and $\xi \in M^{[0,1]}$. Remark~\ref{rmk:length-function-orbit}\eqref{rmk:length-function-orbit:base-point} guarantees the use of the dominated convergence theorem to this integral, which implies that $\ell_\xi \left( H(\varphi, t) \right)$ is continuous in $\varphi$ and in $t$. Thus continuity of $H$ follows from Remark~\ref{rmk:length-function-orbit}\eqref{rmk:length-function-orbit:topology}. 
\end{proof}

\begin{prop}\label{prop:isom-01-proper}
	Let $\gamgrp$ be a discrete group, let $\alfmap \colon \gamgrp \to \isomgrp(\mnf)$ be an isometric, metrically proper action on $\mnf$, and let $(Y,\mu)$ be a nontrivial measure space. Then composing this homomorphism with the canonical embedding
	\[
		\isomgrp(M) \hookrightarrow \isomgrp(L^2(Y,\mu,M))
	\]
	gives an isometric, metrically proper action of $\gamgrp$ on $L^2(Y,\mu,M)$.
\end{prop}

\begin{proof}
	This follows from Remark~\ref{rmk:length-function-orbit}\eqref{rmk:length-function-orbit:proper} and Lemma~\ref{lem:embedding-01-length}. 
\end{proof}

\section{The space of $L^2$-Riemannian metrics}\label{sec:L2metrics}

In this section, we focus on a prominent example of {\admhhs}s\textemdash the space of {$L^2$-Riemannian metrics} on a closed smooth manifold with a fixed density. This example makes use of the general construction of continuum products described in Section~\ref{sec:hhs} and may be considered as an infinite\-/dimensional symmetric space. It is also pivotal in obtaining Theorem~\ref{thm:diffeo} from Theorem~\ref{thm:main}. 

Throughout the section, we let $N$ be an $n$-dimensional closed smooth manifold. We regard a density $\omega$ as a measure on $N$ which is, in each smooth chart, equivalent to the Lebesgue measure with a smooth Radon-Nikodym derivative. A Riemannian metric $g$ on $N$ naturally induces a density: in local coordinates, it can be expressed as
\[
	d \omega_g = \left( \left| \operatorname{det} (g_{ij}) \right| \right) ^{\frac{1}{2}} d m \; ,
\]
where $(g_{ij})$ is the positive definite symmetric matrix corresponding to the Riemannian metric in the local coordinates, $\operatorname{det} (g_{ij})$ is its determinant, and $m$ is the Lebesgue measure. If $\omega_g = \omega$, we say $g$ \emph{induces} $\omega$. 

\begin{constr}
	Consider the symmetric space $P(n)$ of positive definite symmetric real matrices in $M_n(\mathbb{R})$ with determinant $1$, which can be identified, through the congruence action of $\operatorname{SL}(n,\mathbb{R})$ (or more generally, the group $\widetilde{\operatorname{SL}}(n,\mathbb{R}) = \{ \varphi \in \operatorname{GL}(n,\mathbb{R}) \colon \operatorname{det} \varphi = \pm 1 \}$) on $P(n)$, with the quotient space $\operatorname{SL}(n,\mathbb{R}) / \operatorname{SO}(n)$ (and also $\widetilde{\operatorname{SL}}(n,\mathbb{R}) / \operatorname{O}(n)$), with a base point chosen to be the identity matrix $I_n$, which is identified with the class $[e]$ of the identity element in $\operatorname{SL}(n,\mathbb{R})$. As an irreducible Riemannian symmetric space of noncompact type, $P(n)$ is a complete simply connected Riemannian manifold with non-positive curvature, and in particular, also an {\admhhs}. 
\end{constr}

\begin{defn}\label{defn:L2-Riemannian}
	The continuum product $L^2(N, \omega, P(n))$ is called the \emph{space of {$L^2$-Riemannian metrics}} on $N$ with the density $\omega$. It is an {\admhhs} by Proposition~\ref{prop:continuum-product-hhs-summary}. 
\end{defn}

We point out that the space of $L^2$-Riemannian metrics is not a Riemannian-Hilbertian manifold (see Remark~\ref{rmk:not-Hilbert-manifold}). 

\begin{rmk}
	The rationale behind this terminology is the following: The set of all Riemannian metrics on $N$ that induce $\omega$ is identified with sections on a $P(n)$-bundle over $N$. Fixing a Riemannian metric $g$ on $N$ that induces $\omega$ and a Borel trivialization of the tangent bundle $T N$ such that the inner product $g_x$ on $T_x N$ corresponds to $I_n$ for all $x \in N$, we see that the set of all Riemannian metrics embeds densely into $L^2(N, \omega, P(n))$, since the closedness of $N$ implies $d(g', g) < \infty$ for any Riemannian metric $g'$ on $N$ inducing $\omega$. 
\end{rmk}

\begin{rmk}\label{rmk:Pn-metric}
	We make a few remarks on the metric of $P(n)$. The tangent space $T_{I_n} P(n)$ is canonically identified with the linear space of all symmetric real matrices in $M_n(\mathbb{R})$ with trace $0$, on which the Riemannian metric is given by $\langle A, B \rangle = \operatorname{Tr} (AB)$ for $A, B \in T_{I_n} P(n)$ and the Riemannian-geometric exponential map agrees with the matrix exponential map. Thus for any $D \in P(n)$, the distance $d_{P(n)} (D, I_n)$ is given by $\| \log D \|_{\operatorname{HS}}$, the Hilbert-Schmidt norm of the logarithm of the positive definite matrix $D$. Equivalently, for any $T \in \widetilde{\operatorname{SL}}(n,\mathbb{R})$, we have $d_{P(n)} ([T], [e]) = \| \log (T^*T) \|_{\operatorname{HS}}$. Since $\widetilde{\operatorname{SL}}(n,\mathbb{R})$ acts isometrically on $P(n)$, the assignment $T \mapsto d_{P(n)} ([T], [e])$ defines a length function on $\widetilde{\operatorname{SL}}(n,\mathbb{R})$, which is bilipschitz to the length function 
	\[
		T \mapsto \max \left\{ \log (\|T\|), \log (\|T^{-1}\|) \right\} \; ,
	\]
	where $\| \cdot \|$ denotes the operator norm, because we have 
	\[
		2 \log (\|T\|) = \log (\|T^*T\|) = \| \log (T^*T) \|
	\]
	and diagonalizing $T^*T$ yields
	\begin{equation} \label{eq:bilipschitz}
	\| \log (T^*T) \| \leq \| \log (T^*T) \|_{\operatorname{HS}} \leq \sqrt{n} \| \log (T^*T) \| 
	\end{equation}
	for any $T \in \widetilde{\operatorname{SL}}(n,\mathbb{R})$. We also note that 
	\begin{equation} \label{eq:bilipschitz-inverse}
		\log (\|T\|) \leq ({n-1}) \log (\|T^{-1}\|)
	\end{equation}
	and vice versa. 
\end{rmk}

The main reason we consider $L^2(N, \omega, P(n))$ is that any diffeomorphism $\varphi$ of $N$ preserving $\omega$ induces an isometry on $L^2(N, \omega, P(n))$. 

\begin{constr}\label{constr:diff-isom}
	Fix a measurable trivialization of the $P(n)$-bundle over $N$.
	For any diffeomorphism $\varphi$ of $N$ preserving $\omega$. Define an isometry $\varphi_*$ of $L^2(N, \omega, P(n))$ by 
	\[
	\varphi_*(f)(x) =  (D_{\varphi^{-1} (x)} \varphi )\cdot f( \varphi^{-1} (x) ) \; ,
	\]
	where $D_{\varphi^{-1} (x)} \varphi \colon T_{\varphi^{-1} (x)} N \to T_x N$ is the derivative of $\varphi$ at $\varphi^{-1} (x)$, which, under the trivialization we chose, is identified with an element of $\widetilde{\operatorname{SL}}(n, \mathbb{R})$ and thus acts on $P(n)$. The prescription $\varphi \mapsto \varphi_*$ extends the formula for the pushing forward Riemannian metrics under diffeomorphisms and is easily verified to be functorial and implements an embedding of the group $\operatorname{Diff}(N, \omega)$ of volume preserving diffeomorphisms of $N$ into the group $\operatorname{Isom}(L^2(N, \omega, P(n)))$ of isometries of $L^2(N, \omega, P(n))$. 
\end{constr}

There are two special types of isometries of $\operatorname{Isom}(L^2(N, \omega, P(n)))$ that are of particular interest: those that ``dial'' the individual fibers $P(n)$, and those that ``rotate'' the base space. We study them separately. 

An isometry of the first type can be viewed as a measurable function from $N$ to $\operatorname{Isom}(P(n))$, which is identified with $\widetilde{\operatorname{SL}}(n,\mathbb{R})$, satisfying an integrability condition. We make this precise. 

\begin{lem}
	Let $M(N, \omega, \widetilde{\operatorname{SL}}(n, \mathbb{R}))$ be the group of (equivalence classes of) all measurable functions from $N$ to $\widetilde{\operatorname{SL}}(n, \mathbb{R})$, where the product is defined pointwise and two functions are identified if they only differ on a measure-zero set. Then the function 
	\begin{equation}\label{eq:length-lambda-plus}
		\lambda_+ \colon M(N, \omega, \widetilde{\operatorname{SL}}(n, \mathbb{R})) \to [0,\infty] \;, \quad f \mapsto \left(\int_{N} \log  ^2(\|f(x)\|) \, d\omega(x) \right)^{\frac{1}{2}}
	\end{equation}
	satisfies 
	\begin{equation}\label{eq:length-lambda-product}
		\lambda_+ (f f') \leq \lambda_+(f) + \lambda_+(f') 
	\end{equation}
	and 
	\begin{equation}\label{eq:length-lambda-inverse}
		\lambda_+ (f^{-1})   \leq ({n-1}) \,  \lambda_+ (f) 
	\end{equation}
	for any $f$ and $f'$ in $M(N, \omega, \widetilde{\operatorname{SL}}(n, \mathbb{R}))$. In particular, the subset 
	\begin{equation}\label{eq:L2-SL}
		\left\{ f \in M(N, \omega, \widetilde{\operatorname{SL}}(n, \mathbb{R})) \mid \lambda_+(f) < \infty \right\}
	\end{equation}
	is a subgroup. 
\end{lem}

\begin{proof}
	For any $f, f' \in M(N, \omega, \operatorname{SL}(n, \mathbb{R}))$, we compute
	\begin{align*}
	\lambda_+ (f f') & = \left(\int_{N} \log  ^2(\|(f f') (x)\|) \, d\omega(x) \right)^{\frac{1}{2}} \\
	& \leq \left(\int_{N} \log  ^2(\|f (x)\| \| f' (x)\|) \, d\omega(x) \right)^{\frac{1}{2}} \\ 
	& \leq \left(\int_{N} \left(\log(\|f (x)\|) + \log (\| f' (x)\|)\right)  ^2 \, d\omega(x) \right)^{\frac{1}{2}} \\ 
	& \leq \left(\int_{N} \log  ^2(\|f(x)\|) \, d\omega(x) \right)^{\frac{1}{2}} +  \left(\int_{N} \log  ^2(\|f'(x)\|) \, d\omega(x) \right)^{\frac{1}{2}} \\
	& \leq \lambda_+(f) + \lambda_+(f') \; .
	\end{align*}
	Similarly, we have $\lambda_+ (f^{-1})   \leq ({n-1}) \, \lambda_+ (f)$. 
\end{proof}

\begin{constr}\label{constr:L2-SL}
	We define the group $L^2(N, \omega, \widetilde{\operatorname{SL}}(n,\mathbb{R}) )$ to be the subgroup of $M(N, \omega, \widetilde{\operatorname{SL}}(n, \mathbb{R}))$ given in Equation~\eqref{eq:L2-SL}. 
	
	We embed $L^2(N, \omega, \widetilde{\operatorname{SL}}(n,\mathbb{R}) )$ into $\operatorname{Isom}(L^2(N, \omega, P(n)))$ by applying pointwise the congruence action of $\widetilde{\operatorname{SL}}(n,\mathbb{R})$ on $P(n)$, which is well-defined thanks to Equation~\eqref{eq:bilipschitz}. 
\end{constr}

\begin{rmk}
	We make the following observations:
	\begin{enumerate}
		\item The action of $L^2(N, \omega, \widetilde{\operatorname{SL}}(n,\mathbb{R}) )$ on $L^2(N, \omega, P(n))$ is transitive, which can be seen by making measurable lifts of functions in $L^2(N, \omega, P(n))$. 
		\item When we endow $L^2(N, \omega, \widetilde{\operatorname{SL}}(n,\mathbb{R}) )$ with the pointwise convergence topology, i.e., the topology as a subgroup of $\isomgrp(M)$, it is not hard to see that $L^2(N, \omega, P(n))$ is homeomorphic to the quotient
		\[
			L^2(N, \omega, \widetilde{\operatorname{SL}}(n,\mathbb{R}) ) / L(N, \omega, {\operatorname{O}}(n) )
		\]
		where $L(N, \omega, {\operatorname{O}}(n) )$ is the group of all measurable functions from $N$ to the compact group $\operatorname{O}(n)$, identified up to measure 0. Here $L(N, \omega, {\operatorname{O}}(n) )$ is identified with the stabilizer group of the constant function $I_n$ in $L^2(N, \omega, P(n))$. Moreover, $L(N, \omega, {\operatorname{O}}(n) )$ is the fixed point set of the involutive automorphism of $L^2(N, \omega, \widetilde{\operatorname{SL}}(n,\mathbb{R}) )$ given by taking the transpose inverse on $\widetilde{\operatorname{SL}}(n,\mathbb{R}) $. In view of this, we regard $L^2(N, \omega, P(n))$ as an \emph{infinite\-/dimensional symmetric space}. 
		\item The function 
		\begin{equation}\label{eq:length-function-L}
			\lambda \colon L^2(N, \omega, \widetilde{\operatorname{SL}}(n,\mathbb{R}) ) \to \mathbb{R}^{\geq 0} \; , \quad f \mapsto \max \left( \lambda_+(f) , \lambda_+(f^{-1})  \right)
		\end{equation}
		is a length function. 
	\end{enumerate}
\end{rmk}

\begin{constr}\label{constr:rotation-subgroup}
	The group $T(N, \omega)$ of $\omega$-preserving measurable transformations of $N$, which embeds into $\operatorname{Isom}(L^2(N, \omega, P(n)))$ by composition of maps. 
	
	We observe that $T(N, \omega)$ normalizes $L^2(N, \omega, \widetilde{\operatorname{SL}}(n,\mathbb{R}) )$ and thus the subgroup of $\operatorname{Isom}(L^2(N, \omega, P(n)))$ they generate is isomorphic to a semidirect product $L^2(N, \omega, \widetilde{\operatorname{SL}}(n,\mathbb{R}) ) \rtimes T(N, \omega)$. 
\end{constr}

\begin{rmk}
	We make the following observations:
	\begin{enumerate}
		\item It is clear from Construction~\ref{constr:diff-isom} that the image of the map 
		\[
			\operatorname{Diff} (N, \omega) \to \operatorname{Isom}(L^2(N, \omega, P(n)))	
		\]
		is contained in $L^2(N, \omega, \widetilde{\operatorname{SL}}(n,\mathbb{R}) ) \rtimes T(N, \omega)$. 
		\item The length function in Equation~\eqref{eq:length-function-L} is invariant under the conjugation action of $T(N, \omega)$ and thus extends to a length function on $L^2(N, \omega, \widetilde{\operatorname{SL}}(n,\mathbb{R}) ) \rtimes T(N, \omega)$ by taking the value $0$ on $T(N, \omega)$. 
	\end{enumerate}
\end{rmk}

This length function has appeared in Definition~\ref{defn:geometrically-discrete} in the introduction. It is useful in the discussion of proper actions because the following lemma. 

\begin{lem}\label{lem:length-bilipschitz}
	For any $\xi \in L^2(N, \omega, P(n))$, the length functions $\ell_\xi$ (as defined in Remark~\ref{rmk:length-function-orbit}) and $\lambda$ on the group $L^2(N, \omega, \widetilde{\operatorname{SL}}(n,\mathbb{R}) ) \rtimes T(N, \omega)$ are large-scale bilipschitz to each other. 
\end{lem}

\begin{proof}
	Thanks to Remark~\ref{rmk:length-function-orbit}\eqref{rmk:length-function-orbit:base-point}, we know that any two length functions $\ell_{\xi_1}$ and  $\ell_{\xi_2}$ are large-scale bilipschitz. Thus without loss of generality, we may choose $\xi$ to be the constant function $I_n$. Then it is clear from Construction~\ref{constr:continuum-product} that for any $f \in L^2(N, \omega, \widetilde{\operatorname{SL}}(n,\mathbb{R}) )$ and $\tau \in T(N, \omega)$, we have
	\[
		\ell_{I_n} \left( f \tau \right) = d(f \tau \cdot I_n, I_n) = d(f \cdot I_n, I_n) =  \left( \int_{N} \left( d_{P(n)} ([f(x)], I_n) \right)^2 d \omega(x) \right)^{\frac{1}{2}} \; .
	\]
	By Remark~\ref{rmk:Pn-metric}, we see that the right-hand side is equal to 
	\[
		\left( \int_{N} \left( \left\| \log \left( f(x)^{\operatorname{T}} f(x) \right) \right\|_{\operatorname{HS}} \right)^2 d \omega(x) \right)^{\frac{1}{2}} \; .
	\]
	Hence again by Remark~\ref{rmk:Pn-metric}, we see that $\ell_{I_n}$ is bilipschitz to the function
	\[
		f \tau \mapsto \lambda_+ (f) = \left( \int_{N} \left( \log \left\|  f(x) \right\| \right)^2 d \omega(x) \right)^{\frac{1}{2}}
	\]
	and thus, by Equation~\eqref{eq:length-lambda-inverse}, it is also bilipschitz to $\lambda$. 
\end{proof}

\begin{prop}\label{prop:SLSO-proper-length}
	Let $\Gamma$ be a subgroup of $\operatorname{Diff} (N, \omega)$, or more generally a subgroup of $L^2(N, \omega, \widetilde{\operatorname{SL}}(n,\mathbb{R}) ) \rtimes T(N, \omega)$. 
	It inherits an action on $L^2(N, \omega, P(n))$. When we consider $\Gamma$ as a discrete group, the following are equivalent: 
	\begin{enumerate}
		\item the action on $L^2(N, \omega, P(n))$ by $\Gamma$ is metrically proper; 
		\item the length function $\lambda$ is proper on $\Gamma$; 
		\item the metric induced by $\lambda$ is a proper metric.
	\end{enumerate}
\end{prop}

\begin{proof}
	This follows from Lemma~\ref{lem:length-bilipschitz} and Remark~\ref{rmk:length-function-orbit}\eqref{rmk:length-function-orbit:proper}. 
\end{proof}

\begin{rmk}\label{rmk:SLSO-proper-L2-vs-uniform}
	It follows from Equation~\eqref{eq:length-lambda-inverse} that for a subgroup $\Gamma$ of $\operatorname{Diff} (N, \omega)$, the length function $\lambda$ is proper on $\Gamma$ if and only if $\lambda_+$ is proper on $\Gamma$, which is equivalent to  
	\begin{equation}\label{eq:geometrically-discrete-L2}
		\left(\int_N (\log(\|D_x\varphi\|))^2 d\omega(x)\right)^{1/2} \to \infty \quad \text{ as } \gamma \to \infty \text{ in }  \Gamma \; .
	\end{equation}
	Observe that this is weaker than requiring 
	\begin{equation}\label{eq:geometrically-discrete-uniform}
		\inf_{x \in N} \|D_x\varphi\|\to \infty \quad \text{ as } \gamma \to \infty \text{ in }  \Gamma \; .
	\end{equation}
	This latter condition would imply that there is a coarse embedding $\iota$ of $\Gamma$ into $P(n)$, that is, $\iota$ satisfies the property that there are unbounded increasing functions $f_+, f_- \colon [0, \infty) \to [0, \infty)$ such that 
	\[
		f_-(d_{\Gamma}(\gamma, \gamma')) \leq d_{P_n}(\iota(\gamma), \iota(\gamma')) \leq f_+(d_{\Gamma}(\gamma, \gamma')) 
	\]
	for any $\gamma, \gamma' \in \Gamma$, where $d_{\Gamma}$ is a fixed right-invariant proper metric on $\Gamma$. The best possible such functions $f_+$ and $f_-$ are called the \emph{dilation} and the \emph{compression}, respectively, of $\iota$. Such an embedding $\iota$ can be constructed by fixing a family of trivializations for the tangent spaces of $N$ and realizing the derivative $D_{x_0}\varphi$ at a fixed point $x_0$ as an element in $SL(n, \rbbd)$. By a Gram-Schmidt procedure, we see that $P(n)$ is coarsely equivalent to the group of upper triangular matrices, which is amenable and thus coarsely embeds into a Hilbert space. It follows that any $\Gamma$ that satisfies \eqref{eq:geometrically-discrete-uniform} coarsely embeds into a Hilbert space and thus satisfies the strong Novikov conjecture by \cite{yu3}.  
	
	However, it is not clear whether any subgroup $\Gamma$ that is geometrically discrete, i.e., satisfies \eqref{eq:geometrically-discrete-L2}, also embeds into a Hilbert space. A naive attempt would be to embed $L^2(N, \omega, P(n))$ into a Hilbert space by ``integrating'' over $(N, \omega)$ a coarse embedding $\varepsilon$ of $P(n)$ into a Hilbert space $\hil$, that is, defining a map
	\[
		\varepsilon^{(N,\omega)} \colon L^2(N, \omega, P(n)) \to L^2(N, \omega, \hil) \, , \quad \xi \mapsto \varepsilon \circ \xi \; .
	\] 
	To see why this naive attempt is not enough, we let $f_+, f_- \colon [0, \infty) \to [0, \infty)$ be, respectively, the dilation and compression of $\varepsilon$. 
	While since $P(n)$ is a geodesic space, the dilation of $\varepsilon$, and thus also that of $\varepsilon^{(N,\omega)}$, can be controlled by a linear function, yet there is an issue with the compression of $\varepsilon^{(N,\omega)}$: as long as $f_-$ is not bounded from below by a strictly increasing linear function, we can show that the compression of $\varepsilon^{(N,\omega)}$ fails to be unbounded, that is, there exist elements $\xi_n, \xi'_n \in L^2(N, \omega, P(n))$, for $n = 1,2,\ldots$, such that $\dist(\xi_n, \xi'_n) \to \infty$ as $n \to \infty$ but $\dist(\varepsilon^{(N,\omega)}(\xi_n), \varepsilon^{(N,\omega)}(\xi'_n))$ is bounded. Indeed, by our assumption, we can find points $x_n, x'_n \in P(n)$ for $n = 1,2,\ldots$ such that $\dist(\varepsilon(x_n), \varepsilon(x'_n)) \leq \frac{1}{n} \dist(x_n, x'_n)$. Since the measure space $(N,\omega)$ is atomless, we can choose measurable subsets $Y_n \subset N$ with $\omega(Y_n) = (\dist(\varepsilon(x_n), \varepsilon(x'_n)))^{-2}$. Then we define 
	\[
		\xi_n = 
		\begin{cases}
			x_n \, , & y \in Y_n \\
			x_1 \, , & y \in N \setminus Y_n
		\end{cases}
		\qquad \text{and} \qquad 
		\xi'_n = 
		\begin{cases}
		x'_n \, , & y \in Y_n \\
		x_1 \, , & y \in N \setminus Y_n
		\end{cases}
		\; .
	\]
	It follows that for any positive integer $n$, we have $\dist(\varepsilon^{(N,\omega)}(\xi_n), \varepsilon^{(N,\omega)}(\xi'_n)) \leq 1$ but $\dist(\xi_n, \xi'_n) \geq n$, as desired. This means $\varepsilon^{(N,\omega)}$ is not a coarse embedding. 
\end{rmk}

\newcommand{\piunbdd}{\Pi}
\newcommand{\pialg}{\Pi_{\operatorname{b}}}

\section{A $C^*$-algebra $\AofM$ associated to a {\hhs}~$M$}
\label{sec:AofM}

In this section, we define a $C^*$-algebra $\AofM$ associated to a Hilbert-Hadamard space~$M$. This $C^*$-algebra and its $K$-theory will play a key role in our proof.

Throughout this section, we let $M$ be a {\hhs} as in Definition~\ref{defn:hhs}. Recall that for any point $x \in M$, we use $\hhil_x M$ to denote the tangent Hilbert space at $x$.

\begin{defn}\label{defn:Pi-M}
	Given a {\hhs}, we define the $*$\-/algebra
	\[
		\piunbdd (M) = \prod_{(x,t) \in M \times [0,\infty) } \cliffc (\hhil_x M \oplus t \rbbd ) \; ,
	\]
	where 
	\[
		t \rbbd = 
		\begin{cases}
			\rbbd \, , & t > 0 \\
			\{0\} \, , & t = 0 
		\end{cases}
		\; 
	\]
	and $\rbbd$ carries the canonical inner product (independent of $t$). 
	We also define the $C^*$-algebra  
	\[
		\pialg(M) = \left\{ \sigma \in \piunbdd (M) \colon \sup_{(x,t) \in M \times [0,\infty)} \| \sigma({x,t}) \| < \infty \right\}
	\] 
	equipped with pointwise algebraic operations and the uniform norm. 
\end{defn}

Intuitively, these $*$\-/algebras can be viewed as built out of (typically discontinuous) tangent vector fields over $M \times [0, \infty)$. 
Although $\pialg(M)$ is too large a $C^*$-algebra to be of much use, it will contain our key object in this section, $\AofM$, as a $C^*$-subalgebra. A key ingredient is a following version of outward pointing Euler vector fields. 

\begin{defn}\label{defn_Cliffordmultiplier}
	Let $\mnf$ be a {\hhs}. For any point $\xpt_0 \in \mnf$, we define the \emph{Clifford operator} $\cliffmult \in \piunbdd (M)$ by
	\[
	\cliffmult (x , t) = \left(-\log_{\xpt} (\xpt_0) , t \right) \in \tanbndl_{\xpt} \mnf \times t \rbbd \subset \cliffc({\hhil_{\xpt}{\mnf} \oplus t \rbbd})
	\]
	for any $\xpt \in \mnf$ and $t \in [0,\infty)$. We also write $\cliffmult^{M}$ when we need to emphasize the {\hhs} $M$. 
\end{defn}

Note that $\cliffmult$ is unbounded, i.e., $\displaystyle \sup_{(x,t) \in M \times [0,\infty)} \| \cliffmult({x,t}) \| = \infty$, whenever $M$ is unbounded. Also observe that each $\cliffmult(x,t)$ is self-adjoint. 

\begin{eg}\label{eg_CliffordmultiplieronaHilbertspace}
	When $\mnf$ is a (real) Hilbert \emph{space}, upon identifying $\mnf$ with its own tangent cones $\tanbndl_{\xpt} \mnf$ in the canonical way, we have
	\[
	\cliffmult( \xpt, t )=  (\xpt-\xpt_0, t )\in  \tanbndl_{\xpt} \mnf \oplus t \rbbd \; .
	\] 
	Thus in this case, $\cliffmult$ is a restriction of the classical \emph{Euler vector field} on $M \times \rbbd$ centered at $(\xpt_0, 0)$, given by
	\[
		( \xpt, t ) \mapsto (\xpt-\xpt_0, t ) \; .
	\]
\end{eg}

\begin{rmk}\label{rmk:Clifford-even-odd-functional-calculus}
	We point out a standard fact regarding the functional calculus of vectors inside the Clifford algebra. Let $\hil$ be a Hilbert space and let $\xi$ be a vector in $\hil$, viewed as a self-adjoint element of the Clifford algebra $\cliffc \hil$. Let $f$ and $g$ be bounded continuous functions on $\rbbd$ with $f$ even and $g$ odd. Applying functional calculus to $\xi$, we have
	\begin{enumerate}
		\item the element $f(\xi)$ in $\cliffc \hil$ is equal to the scalar $f(\|\xi\|)$, and
		\item the element $g(\xi)$ in $\cliffc \hil$ is equal to the vector in $\hil$ with norm equal to $| g(\|\xi\|) |$, and pointing in the same direction as $\xi$ if $g(\|\xi\|) \geq 0$ and otherwise in the opposite direction. Note that when $\xi = 0$, we have $g(\|\xi\|) = g(0) = 0 $, so there is no ambiguity. 
	\end{enumerate} 
\end{rmk}

\begin{prop}\label{prop_welldefinednessofBotthomomorphism}
	For any point $\xpt_0\in\mnf$, the map
	\[
	\botthom \colon \salg \to \piunbdd (M)
	\]
	is defined by functional calculus such that 
	\[
	\botthom (\ffunc) (x,t) = \ffunc(\cliffmult(\xpt, t))
	\]
	for all $\xpt \in \mnf$, $t \in [0,\infty)$, and $\ffunc \in \salg$. Then $\botthom $ is a graded $*$\-/homomorphism from $\salg$ to the $*$\-/subalgebra $\pialg(M)$.  
\end{prop}
\begin{proof}
	This follows from the fact that $\cliffmult(\xpt, t)$ is an odd bounded self-adjoint operator for any $\xpt \in \mnf$ and $t \in [0,\infty)$. 
\end{proof}

\begin{defn}\label{defn_Botthomomorphism}
	For any point $\xpt_0\in\mnf$, the graded $*$\-/homomorphism 
	$$\botthom: \salg \to \pialg(M) $$
	is called the \emph{Bott homomorphism} centered at $\xpt_0$.
	We also write $\botthom^{M}$ when we need to emphasize the {\hhs} $M$. 
\end{defn}

We discuss some important features of $\botthom$. Let us denote the even part of $ \salg $ by $ \salg_\text{ev} $, which consists of all even functions.

\begin{prop}\label{prop_Botthomomorphismandevenpart}
	For any $ \xpt_0 \in \mnf $, the map $ \botthom$ takes $ \salg_\text{ev} $ into the subalgebra $ \ell^\infty (M \times [0,\infty))$ of the center of $\pialg(M) $. 
\end{prop}
\begin{proof}
	For any $ (\xpt, t ) \in \mnf \times [0, \infty)$ and $\ffunc \in 
	\salg_\text{ev}$, we can write $ \ffunc(s)=\gfunc(s^2) $ for some $\gfunc\in \cz([0,\infty))$ and thus
	\begin{equation}\label{eq:AevenofM-distance}
		\botthom(\ffunc)(\xpt, t) = \gfunc \left( \cliffmult^2(\xpt, t) \right) =  \gfunc \left(  \dist(\xpt_0,\xpt)^2 + t^2 \right) \; ,
	\end{equation}
	which is a scalar. 
\end{proof}

Now we discuss the dependence of the Bott homomorphisms on the base points. 

\begin{lem}\label{lem_Cliffordmultiplierandchangeofbasepoint}
	For any $\xpt_0, \xpt_1\in\mnf$, we have $\cliffmult-\cliffmult[\xpt_1] \in \pialg(M)$ and in fact 
	\[
	\|\cliffmult-\cliffmult[\xpt_1]\| \le \dist(\xpt_0,\xpt_1) \; .
	\]
\end{lem}
\begin{proof}
	By Construction~\ref{constr:log-map}, the logarithm map $\log_x$ is non-expansive for any $x$. It follows that 
	\[
	\|\cliffmult(\xpt, t)-\cliffmult[\xpt_1](\xpt, t)\| = \|- \log_x(\xpt_0) + \log_x(\xpt_1)\|_{\hhil_x} \le \dist(\xpt_0,\xpt_1) 
	\] 
	for any $(x, t) \in M \times [0, \infty)$, whence the claims follow. 
\end{proof}

\newcommand{\meansym}[1]{\Theta_{#1}}

\begin{defn}\label{defn:Omega-Theta}
	For any $\ffunc\in\salg$ and $r \geq 0$, let us define its \emph{$\rdist$-oscillation} by 
	\begin{equation}\label{eq:oscillation}
		\oscill{\rdist}\ffunc = \sup \left\{ |\ffunc(\tvar) - \ffunc(\tvar')| : \ \tvar,\tvar'\in\rbbd, |\tvar-\tvar'|\le \rdist \right\} \; .
	\end{equation}
	For $r>0$, we also define
	\begin{equation}\label{eq:mean}
		\meansym{\rdist} \ffunc = \rdist \cdot \sup \left\{ \frac{| f(t) - f(-t) |}{2t} \colon t \geq \rdist \right\} \; .
	\end{equation}
	It is clear that $\oscill{\rdist}\ffunc \leq 2 \|f\|$ and $\meansym{\rdist} \ffunc \leq \|f\|$.  
\end{defn}

\begin{lem}\label{lem:Omega-Theta-lim-r}
	For any $\ffunc \in \salg$, we have
	\[
		\lim_{r \to 0} \oscill{\rdist}\ffunc = 0 = \lim_{r \to 0} \meansym{\rdist}\ffunc  \; .
	\]
\end{lem}
\begin{proof}
	The first equation follows from the fact that every function in $\salg$, being a uniform limit of compactly supported functions, is uniformly continuous. The second equation follows from the observation that for any continuous function $g$ from $(0,\infty)$ to $[0, \infty)$ with $\displaystyle \lim_{r \to 0} t g(t) = 0$, we have
	\[
		\lim_{r \to 0} \rdist \sup \left\{ g(t) \colon t \geq \rdist \right\} = 0 \; .
	\]
	We apply this to the function $g(t) = \displaystyle \frac{| f(t) - f(-t) |}{2t}$. 
\end{proof}

Using these notations, we have the following estimate. 

\begin{prop}\label{prop_Botthomandchangeofbasepoint}
	For any $\xpt_0,\xpt_1 \in \mnf$ and any $\ffunc\in\salg$, writing $\rdist$ for the distance $ \dist(\xpt_0,\xpt_1)$, we have
	\[
		\|\botthom(\ffunc)- \botthom[\xpt_1](\ffunc)\| \le 2 \; \oscill{\rdist}\ffunc + \max\left\{  \oscill{2\rdist}\ffunc , \, \meansym{\rdist} \ffunc \right\} \; . 
	\]
\end{prop}

\begin{proof}
	Observe that for any $(x, t) \in M \times [0,\infty)$, since we have $\|\cliffmult[x_i](\xpt, t)\| = \sqrt{d(x_i , x)^2 + t^2}$ for $i = 0,1$, it follows from the triangle inequality that 
	\begin{align}
		\label{eq:prop_Botthomandchangeofbasepoint::leq} \left| \|\cliffmult[x_0](\xpt, t)\| - \|\cliffmult[x_1](\xpt, t)\| \right| & \leq |d(x_0 , x) - d(x_1 , x) | \leq r \; .
	\end{align}

	Now let $f_0$ and $f_1$ be the even and odd parts of $f$, that is, 
	\[
		f_0 (t) = \frac{f(t) + f(-t)}{2} \quad \text{ and } \quad f_1 (t) = \frac{f(t) - f(-t)}{2} 
	\]
	for any $t \in \rbbd$. Then for $ i = 0, 1$ and any $\sdist > 0$, we have $f_i \in \salg$, $\|f_i \| \leq \|f\|$, $\oscill{\sdist} (f_i) \leq \oscill{\sdist} (f) $, and 
	\begin{equation} \label{eq:prop_Botthomandchangeofbasepoint::meansym} 
		\meansym{\sdist}\ffunc = \meansym{\sdist}\ffunc_1 = \sdist \cdot \sup \left\{ \frac{| f_1(t) |}{t} \colon t \geq \sdist \right\} \; .
	\end{equation}
	
	By Remark~\ref{rmk:Clifford-even-odd-functional-calculus}, we have $\botthom[x_i](\ffunc_0)(\xpt, t)= \ffunc_0 \left(\|\cliffmult[x_i](\xpt, t) \|\right)$ for $i=0,1$,
	and thus by Equation~\eqref{eq:prop_Botthomandchangeofbasepoint::leq}, we have
	\begin{align*}
		\left\| \big( \botthom(\ffunc_0)-\botthom[\xpt_1](\ffunc_0) \big)(\xpt, t) \right\| & = \big| \ffunc_0 \left( \|\cliffmult[x_0](\xpt, t) \| \right) - \ffunc_0 \left( \|\cliffmult[x_1](\xpt, t) \| \right) \big| \\
		& \leq \oscill{\rdist}\ffunc_0 \leq \oscill{\rdist}\ffunc \; .
	\end{align*}
	
	On the other hand, to estimate $\left\| \big(\botthom(\ffunc_1)-\botthom[\xpt_1](\ffunc_1) \big)(\xpt, t) \right\|$, we discuss two complementary cases:
	\begin{enumerate}
		\item When one of $\|\cliffmult[x_0](\xpt, t) \|$ and $ \|\cliffmult[x_1](\xpt, t) \|$ is less than $\rdist$, then by Equation~\eqref{eq:prop_Botthomandchangeofbasepoint::leq}, the other one is less than $2 \rdist$, whence, observing that $f_1(0)=0$, we have
		\begin{align*}
			& \left\| \big(\botthom(\ffunc_1)-\botthom[\xpt_1](\ffunc_1) \big)(\xpt, t) \right\| \\
			=\ & \left\|  \ffunc_1\left(\cliffmult(\xpt, t)\right) - \ffunc_1\left(\cliffmult[x_1](\xpt, t)\right)   \right\| \\
			\leq \ & \left\|  \ffunc_1\left(\cliffmult(\xpt, t)\right)   \right\| + \left\|  \ffunc_1\left(\cliffmult[x_1](\xpt, t)\right)   \right\| \\ 
			\leq \ & \oscill{\rdist}\ffunc_1  + \oscill{2\rdist}\ffunc_1 \\ 
			\leq \ & \oscill{\rdist}\ffunc  + \oscill{2\rdist}\ffunc  \; .
		\end{align*}
		\item When $\min \left\{ \|\cliffmult[x_0](\xpt, t) \| , \|\cliffmult[x_1](\xpt, t) \| \right\} \geq \rdist$, then we denote the angle between the vectors $\cliffmult (\xpt, t) $ and $\cliffmult[x_1] (\xpt, t)$ in $\hil_\xpt \oplus t \rbbd$ by $\theta$. By Euclidean geometry, we have
		\[
		\cos \theta = \frac{ \left\| \cliffmult(\xpt, t)  \right\|^2 + \left\|  \cliffmult[x_1](\xpt, t)   \right\|^2 - \left\| \cliffmult(\xpt, t) - \cliffmult[x_1](\xpt, t)   \right\|^2 }{ 2 \, \left\| \cliffmult(\xpt, t)  \right\| \, \left\|  \cliffmult[x_1](\xpt, t)   \right\|}
		\]
		and thus by Lemma~\ref{lem_Cliffordmultiplierandchangeofbasepoint}, we have
		\begin{align*}
			1 - \cos \theta & = \frac{ \left\| \cliffmult(\xpt, t) - \cliffmult[x_1](\xpt, t)   \right\|^2  - \left( \|\cliffmult[x_0](\xpt, t) \| - \|\cliffmult[x_1](\xpt, t) \| \right)^2 }{ 2 \, \|\cliffmult[x_0](\xpt, t) \| \, \|\cliffmult[x_1](\xpt, t) \|} \\
			& \leq \frac{ d(x_0 , x_1)^2  }{ 2 \, \|\cliffmult[x_0](\xpt, t) \| \, \|\cliffmult[x_1](\xpt, t) \|} \; .
		\end{align*}
		Hence, writing $s_i = \|\cliffmult[x_i](\xpt, t) \|$ for $i = 0,1$ and using Remark~\ref{rmk:Clifford-even-odd-functional-calculus} and Equations~\eqref{eq:prop_Botthomandchangeofbasepoint::leq} and~\eqref{eq:prop_Botthomandchangeofbasepoint::meansym} , we compute
		\begin{align*}
		& \left\| \big(\botthom(\ffunc_1)-\botthom[\xpt_1](\ffunc_1) \big)(\xpt, t) \right\|^2 \\
		=\ & \left\|  \ffunc_1\left(\cliffmult(\xpt, t)\right) - \ffunc_1\left(\cliffmult[x_1](\xpt, t)\right)   \right\|^2 \\
		=\ & \ffunc_1\left(s_0\right) ^2 + \ffunc_1\left(s_1\right) ^2  - 2 \, \ffunc_1\left(s_0\right) \, \ffunc_1\left(s_1\right) \, \cos\theta \\
		=\ & \big( \ffunc_1\left(s_0\right) - \ffunc_1\left(s_1\right) \big)^2  + 2 \,  \ffunc_1\left(s_0\right) \, \ffunc_1\left(s_1\right) \, (1 - \cos\theta) \\
		\leq \ & \left(\oscill{\rdist}\ffunc_1 \right)^2 + \frac{\ffunc_1\left(s_0\right)}{s_0} \, \frac{\ffunc_1\left(s_1\right)}{s_1} \cdot d(x_0 , x_1)^2 \\
		\leq \ & \left(\oscill{\rdist}\ffunc \right)^2 + \left( \meansym{\rdist} \ffunc \right)^2 \\
		\leq \ & \left(\oscill{\rdist}\ffunc + \meansym{\rdist} \ffunc  \right)^2 \; .
		\end{align*}
	\end{enumerate}

	Since $\ffunc = \ffunc_0 + \ffunc_1$, combining the two estimates gives us the final result. 
\end{proof}

\begin{cor}\label{cor_botthomiscontwrtbasepoint}
	If a sequence $ \{ \xpt_\nind \in \mnf \}_{\nind \in \nbbd} $ converges to $\xpt_0 \in \mnf$, then for any $\ffunc\in\salg$, we have 
	\[
	\lim_{\nind \to \infty} \botthom[\xpt_\nind](\ffunc) = \botthom(\ffunc) 
	\]
	in $\pialg(\mnf)$. 
\end{cor}
\begin{proof}
	This follows from Proposition~\ref{prop_Botthomandchangeofbasepoint} and Lemma~\ref{lem:Omega-Theta-lim-r}. 
\end{proof}

The following remark is not essential in the proofs of our main theorems. We will only use it in Remark~\ref{rmk:AofM-complemented}. 

\begin{rmk}\label{rmk:botthom-alternative}
	We sketch an alternative, geometric description of the Bott homomorphisms that does not require functional calculus. 
	Let $\{\varhexstar\}$ denote a one-point space. Then we have an embedding 
	\begin{align*}
	\salg &~\hookrightarrow \Pi_{\operatorname{b}} (\{\varhexstar\}) & = & \prod_{t \in [0,\infty) } \cliffc(t \rbbd) \\
	&& \cong &  \left\{ \sigma \in \prod_{t \in [0,\infty) } (\cbbd \oplus \cbbd) \colon \sigma(0) \in \cbbd \cdot 1_{\cbbd \oplus \cbbd} \right\} \\ 
	f &~ \mapsto && \Big( f(t), f(-t) \Big)_{t \in [0,\infty) }
	\end{align*}
	where the identification $\cliffc( \rbbd) \cong \cbbd \oplus \cbbd$ is so that the vector $1 \in \rbbd$ corresponds to $(1,-1)$. Note that it follows from Remark~\ref{rmk:Clifford-even-odd-functional-calculus} that 
	\[
	\Big( f(t), f(-t) \Big)_{t \in [0,\infty) } = \Big( f( \cliffmult[\varhexstar](t) ) \Big)_{t \in [0,\infty) } \; .
	\]
	On the other hand, fixing a base point $x_0 \in M$ as before, we define a continuous map
	\[
	\dist_{x_0} \colon M \times [0, \infty) \to [0, \infty) \, , \quad (x,t) \mapsto \sqrt{ d(x, x_0)^2 + t^2 } \; .
	\]
	Observe that $\dist_{x_0} (x , t) = \| \cliffmult (x , t) \| $. For each $(x,t) \in M \times [0, \infty)$, we thus let 
	\[
		\tau_{x_0}(x,t) \colon \dist_{x_0} (x , t) \cdot \rbbd \to \hhil_{x} M \oplus t \rbbd
	\]
	be the linear isometric embedding that maps $\dist_{x_0} (x , t)$ to $\cliffmult (x , t)$. Notice that $\tau_{x_0}(x,t)$ is the embedding of the trivial vector space when $\dist_{x_0} (x , t) = 0$. 
	Together they make up a $*$\-/homomorphism
	\[
	B_{x_0} \colon \Pi (\{\varhexstar\}) \to \Pi (M) \, , \quad \sigma \mapsto \Big( \left(\cliffc \big( \tau_{x_0} (x,t) \big) \circ \sigma \circ \dist_{x_0} \right) (x,t)  \Big)_{(x,t) \in M \times [0, \infty)}
	\]
	where 
	\[
	\cliffc \big( \tau_{x_0} (x,t) \big) \colon \cliffc (\dist_{x_0}(x,t) \cdot \rbbd) \to \cliffc \left( \hhil_{x} M \oplus t \rbbd \right)
	\]
	is the induced $*$\-/homomorphism between the Clifford algebras. Notice that when $\dist_{x_0} (x , t) = 0$, then $t=0$ and the map $\cliffc \big( \tau_{x_0} (x,t) \big)$ can be identified with the embedding of $\cbbd$ into $\cliffc \left( \hhil_{x} M \oplus t \rbbd \right) = \cliffc \left( \hhil_{x} M \right)$ as the scalars. This $*$\-/homomorphism obviously maps $\Pi_{\operatorname{b}} (\{\varhexstar\})$ to $\Pi_{\operatorname{b}} (M)$ and is characterized by the equation
	\[
		B_{x_0} (\cliffmult[\varhexstar]) = \cliffmult \; .
	\]
	Unpacking the definitions, we see that the Bott homomorphism $\botthom$ is just the restriction of $B_{x_0}$ to $\salg \subset \Pi_{\operatorname{b}} (\{\varhexstar\})$. 
	
	Incidentally, the linear map $\tau_{x_0}(x,t)$ may be seen as the ``adjoint of the derivative'' of $\dist_{x_0}$ at $(x,t)$, because it can be characterized by the equations
	\[
	\left\langle \tau_{x_0}(x,t) \left( s \right) , \log_{(x,t)} (x', t')  \right\rangle_{\hhil_x \oplus \rbbd} = \lim_{r \to 0} \frac{1}{r}  \left\langle  s  , \dist_{x_0} (x'_r,t'_r) - \dist_{x_0} (x,t)   \right\rangle_{\rbbd}
	\]
	for all $(x',t') \in M \times [0, \infty)$, 
	where we write 
	\[
	\log_{(x,t)} (x', t') = \left( \log_{x} (x') , t'-t \right) \in \hhil_{x} M \oplus \rbbd \; ,
	\]
	\[
	x'_r = [x, x'](r) \quad \text{and} \quad t'_r = t + r(t'-t) \; ,
	\]
	using the notation introduced in Construction~\ref{constr:log-map}. This equation can be proved by checking the case when $s = \dist_{x_0} (x,t)$ and observing that the angle between the geodesic segments $[x, x_0]$ and $[x, x']$ is the limit of the comparison angle $\widetilde{\angle}(x_0, x, x'_r)$. 
\end{rmk}

We are now ready to introduce the main definition of this section. 

\begin{defn}\label{defn_AofM}
	Let $\mnf$ be a {\hhs}. The algebra $\aalg(\mnf)$ is the $\cstar$-subalgebra of $\pialg(M)$ generated by 
	\[
	\{ \botthom(\ffunc) \colon \xpt_0 \in \mnf,\ \ffunc \in \salg \ \} \; .
	\]
	We also define $ \AevenofM $ to be the $\cstar$-subalgebra of $\aalg(\mnf)$ generated by 
	\[
	\{ \botthom(\ffunc) \colon \xpt_0 \in \mnf,\ \ffunc \in \salg_\text{ev} \ \} \; . 
	\]
\end{defn}

There are two main reasons why we construct the $C^*$-algebra $\AofM$ this way: 
\begin{enumerate}
	\item It carries a natural and well-behaved action by $\isomgrp(M)$. 
	\item Its $K$-theory can be at least partially computed. 
\end{enumerate}
These topics will be discussed in the following sections. For the rest of this section, we establish some basic facts about $\AofM$.

\begin{prop}\label{prop_AofMisseparable}
	$\aalg(\mnf)$ is separable whenever $M$ is separable.
\end{prop}
\begin{proof}
	By the separability of $\mnf$ and $\salg$, there are countable dense subsets $\xsp$ and $\fset$ of $\mnf$ and $\salg$, respectively. It follows then from Lemma~\ref{cor_botthomiscontwrtbasepoint} and functional calculus that $\{ \botthom(\ffunc) \ |\ \xpt_0 \in \xsp,\ \ffunc \in \fset \}$ is a countable dense subset of $\AofM$. 
\end{proof}

We observe that by Proposition~\ref{prop_Botthomomorphismandevenpart}, $\AevenofM$ consists of scalar functions and is thus in the center of $\AofM$. In fact, we have the following proposition.  

\begin{prop}\label{prop_AofMisAevenofMalgebra}
	Let $\widehat{\AevenofM}$ be the spectrum of the central subalgebra $\AevenofM$. Then $\AofM$ is an $\widehat{\AevenofM}$-{\cstaralg}. 
\end{prop}
\begin{proof}
	Since we already know that $\AevenofM$ is a sub-$\cstar$-algebra of the center $Z(\AofM)$, to get a continuous map from $\widehat{Z(\AofM)}$ to $\widehat{\AevenofM}$, it suffices to show that $\AevenofM \cdot \AofM$ is dense in $\AofM$, which follows from the definition of $\AofM$ and the fact that every $\ffunc\in\salg$ can be written as a product $\ffunc = \ffunc_1 \ffunc_2$ where $\ffunc_1 \in \salg_\text{ev}$ and $ \ffunc_2 \in \salg$.
\end{proof}

In the special case when the {\hhs} $M$ is a finite\-/dimensional Riemannian manifold (cf.~Example~\ref{ex:Riemannian-Hilbertian}), we have rather concrete identifications for $\AevenofM$ and $\AofM$. 

\begin{lem}\label{lem:AofMeven-fin-dim}
	Let $\mnf$ be a complete, connected and simply connected finite\-/dimensional Riemannian manifold with non-positive sectional curvature. Then $\AevenofM$ coincides with $C_0(M \times [0,\infty))$ as $C^*$-subalgebras of $\Pi_{\operatorname{b}} (M)$, where $C_0(M \times [0,\infty))$ embeds into $\Pi_{\operatorname{b}} (M)$ as scalar functions on $M \times [0,\infty)$. 
\end{lem}

\begin{proof}
	It follows from Proposition~\ref{prop_Botthomomorphismandevenpart} that $\left(\cliffmult\right) ^2$ is a continuous and proper function from $M \times [0,\infty)$ to $[0,\infty) \subset\cbbd$ for any $x_0 \in M$; thus $\botthom(f)$ falls in $C_0(M \times [0,\infty))$ for any even function $f \in \salg$. Therefore we have $\AevenofM \subset C_0(M \times [0,\infty))$. 
	
	To prove the other direction, we first claim that the generating set
	\[
		\left\{ \beta_{x_0} (f) \colon f \in \salg_\text{ev}, x_0 \in M \right\}
	\]
	of $\AevenofM$, when viewed as a subset of $C_0(M \times [0,\infty))$, separates points. Indeed, for any two different points $(x_1, t_1)$ and $(x_2 , t_2)$ in $M \times [0,\infty)$ with $t_1 \leq t_2$, we have
	\[
	\dist(\xpt_1,\xpt_1)^2 + t_1^2 = t_1^2 <  \dist(\xpt_1,\xpt_2)^2 + t_2^2
	\]
	and thus by Equation~\eqref{eq:AevenofM-distance}, we can choose $f \in \salg_\text{ev}$ so that 
	\[
	\beta_{x_1} (f) (x_1, t_1) \not= \beta_{x_1} (f) (x_2, t_2) \; .
	\]
	
	Hence the Stone-Weierstrass theorem implies $\AevenofM = C_0(M \times [0,\infty))$. 
\end{proof}

\begin{prop}\label{prop_AofM-Euclidean}
	Let $V$ be a finite\-/dimensional Euclidean space, viewed as a {\hhs}, with each tangent cone identified canonically with $V$ itself. Then inside $\Pi_{\operatorname{b}} (V)$, the $C^*$-subalgebra $\AofM[V]$ coincides with 
	\begin{equation}\label{eq:prop_AofM-Euclidean}
		\left\{ f \in C_0 \big( V \times [0, \infty) , \cliffc (V \oplus \rbbd ) \big) \colon f( x , 0 ) \in \cliffc (V ) \text{ for all } x \in V \right\} \; ,
	\end{equation}
	where the embedding $\cliffc (V ) \hookrightarrow \cliffc (V \oplus \rbbd )$ is induced from embedding $V$ into the first factor of $V \oplus \rbbd $.  
\end{prop}

\begin{proof}
	Let us denote the $C^*$-algebra in Equation~\eqref{eq:prop_AofM-Euclidean} by $\balg$. By Lemma~\ref{lem:AofMeven-fin-dim}, the $C^*$-subalgebra $\AevenofM[V]$ is contained in both $\AofM[V]$ and $\balg$. It also makes both into $( V \times [0, \infty) )$-$C^*$-algebras (cf.\,Proposition~\ref{prop_AofMisAevenofMalgebra}). 
	
	Now for $(x, t) \in V \times [0,\infty)$, we see that  
	\[
		\left\{ C_{x_0} (x,t) \colon x_0 \in V \right\} = V \times \{ t \} \subset V \oplus t \rbbd \subset \cliffc (V \oplus t \rbbd ) \; .
	\]
	It follows that as a $( V \times [0, \infty) )$-$C^*$-algebra, the fiber of $\AofM[V]$ at $(x,t)$ is $\cliffc (V \oplus t \rbbd )$, which is the same as that of $\balg$. 
	Therefore $\AofM[V] = \balg$ as $C^*$-subalgebras of $\pialg(V)$. 
\end{proof}

\begin{cor}\label{cor:AofM-fin-dim}
	We write $\mtrxalg{n}$ for the algebra of $n \times n$-matrices. Let $k$ be a natural number. Then we have the isomorphisms
	\begin{align*}
		\AofM[\rbbd^{2k}] &~\cong \cz(\rbbd^{2k}\times\rbbd, \mtrxalg{2^k}) \; , \\
		\AofM[\rbbd^{2k+1}] &~\cong \{ \ffunc \in \cz(\rbbd^{2k+1} \times [0,\infty), \mtrxalg{2^{k+1}} ) \colon \ffunc(\rbbd^{2k+1} \times \{0\}) \subset \mtrxalg{2^{k}} \oplus \mtrxalg{2^{k}} \ \} \; ,
	\end{align*}
	where $\mtrxalg{2^{k}} \oplus \mtrxalg{2^{k}}$ embeds into $\mtrxalg{2^{k+1}}$ diagonally. 
\end{cor}

\begin{proof}
	These follow from Proposition~\ref{prop_AofM-Euclidean} using the identifications $\cliffc (\rbbd^{2k} ) \cong \mtrxalg{2^k}$ and $\cliffc (\rbbd^{2k+1} ) \cong \mtrxalg{2^k} \oplus \mtrxalg{2^k} $. 
\end{proof}

\begin{rmk} \label{rmk:AofM-fin-dim-graded}
	For a finite\-/dimensional Euclidean space $V$, the $C^*$-algebra $\AofM[V]$ is isomorphic to $\mathscr{SC}(V)$ used in \cite[\S 2, 3]{higsonkasparovtrout}, which is defined as the graded tensor product $\salg \widehat{\otimes} C_0 \big( V  , \cliffc V   \big)$, where $C_0 \big( V , \cliffc (V  ) \big)$ inherits the grading from $\cliffc V$ and $\salg$ is graded by even and odd functions. This follows from the facts that $\cliffc \rbbd \widehat{\otimes} \cliffc V \cong \cliffc (V \oplus \rbbd)$ and  
	\[
		\salg \cong \left\{ f\in C_0 \left([0,\infty) , \cliffc \rbbd \right) \colon f(0) \in \cbbd \right\} \; .
	\]
	Moreover, under this identification, it is clear that our Bott homomorphism coincides with the $*$\-/homomorphism $\beta \colon \mathscr{SC}(\{0\}) \to \mathscr{SC}(V)$ constructed in \cite[\S 2, 5]{higsonkasparovtrout}.
\end{rmk}

\begin{prop}\label{prop_AofM-fin-dim}
	Let $M$ be a complete, connected and simply connected finite\-/dimensional Riemannian manifold with non-positive sectional curvature and let $V$ denote its tangent space at a point $x_0$, viewed as an Euclidean vector space. Then there is an isomorphism 
	\[
		\Phi \colon \AofM[V] \to \AofM
	\]
	which intertwines the Bott homomorphisms in the sense that $\Phi \circ \botthom[0]^{V} = \botthom^{M}$. 
\end{prop}

\begin{proof}
	By the Cartan-Hadamard theorem (cf.\ Example~\ref{eg:logarithm-map-Riemannian}), the logarithm map $\log_{x_0}$ and the exponential map $\exp_{x_0}$ are mutually inverse diffeomorphisms between $M$ and $V$. 
	Thus there is a vector bundle isomorphism $D \exp_{x_0}$ from the trivial $V$-bundle on $V$ to the tangent bundle $T M$; namely, for any $(v,w) \in V \times V$, we have $D \exp_{x_0} (v, w) = \left(\exp_{x_0}(v), D_{v} \exp_{x_0} (w) \right)$, where $D_{v} \exp_{x_0} \colon V \to T_{\exp_{x_0}(v)} M = \hhil_{\exp_{x_0}(v)} M$ is the derivative. Its inverse is the similarly defined $D \log_{x_0}$. 
	
	Now the Riemannian metric on $M$ induces inner products on the fibers of both bundles. Applying polar decomposition, we obtain, for each $v \in V$, 
	\[
		D_{v} \exp_{x_0} = \varphi_v \, \lambda_v \; ,
	\]
	where $\lambda_v = \left( \left( D_{v} \exp_{x_0} \right)^* \, \left(D_{v} \exp_{x_0}\right) \right)^{\frac{1}{2}}$ and $\varphi_v \colon V \to T_{\exp_{x_0}(v)} M$ is an isometric linear isomorphism. Thus we obtain a Riemannian vector bundle isomorphism $\varphi = (\varphi_v)_{v\in V}$ from the trivial bundle $V \times V$ to $T M$. 
	
	Observe that $\exp_{x_0}$ maps each line through $0$ to a geodesic in $M$ passing through $x_0$ and is an isometry when restricted to each such line. Thus for any $v \in V$, we have
	\[
		D \exp_{x_0} (v, v) = \left(\exp_{x_0}(v), - D_{v} \exp_{x_0} (-v) \right) =  \left(\exp_{x_0}(v), - \log_{\exp_{x_0} (v) } x_0 \right)  \; , 
	\] 
	that is, $D \exp_{x_0}$ intertwines the Euler vector field $v \mapsto v$ on $V$ and the vector field $x \mapsto - \log_{x} (x_0)$ on $M$ defining the Clifford operator in Definition~\ref{defn_Cliffordmultiplier}. In particular, $D_{v} \exp_{x_0}$ is isometric on $v$. 
	Since $\exp_{x_0}$ is metric semi\-/increasing (cf.\,Construction~\ref{constr:log-map}), so are $D_{v} \exp_{x_0}$ and thus $\lambda_v$. Hence $v$ is in the eigenspace of $1$, the smallest eigenvalue of $\lambda_v$. It follows that $\varphi_v (v) = D_{v} \exp_{x_0} (v) = - \log_{\exp_{x_0} (v) } x_0$, that is, $\varphi$ also intertwines the two aforementioned vector fields. 
	
	The Riemannian vector bundle isomorphism $\varphi$ induces a $*$\-/isomorphism
	\begin{align*}
		\widetilde{\Phi}_{x_0} \colon \Pi (V) &\to \Pi (M) \\
		\sigma &\mapsto \left( \cliffc \left( \varphi_{  \log_{x_0} (x) } \times \idmap_{t \rbbd} \right) \left( \sigma \left(\log_{x_0} (x) , t\right)\right) \right)_{(x,t) \in M \times [0,\infty)}
	\end{align*}
	that restricts to a $*$\-/isomorphism between $\Pi_{\operatorname{b}} (V)$ and $\Pi_{\operatorname{b}} (M)$ and intertwines the Clifford operators: 
	\[
		\widetilde{\Phi}_{x_0} \left( \cliffmult[0]^{V} \right) = \cliffmult^{M} \; .
	\]
	It follows that $\widetilde{\Phi}_{x_0} \circ \botthom[0]^{V} = \botthom^{M}$. Hence it only remains to show that 
	\[
		\widetilde{\Phi}_{x_0} \left( \AofM[V] \right) = \AofM \text{, or equivalently, } \AofM[V]  = \left(\widetilde{\Phi}_{x_0}\right)^{-1} \left(\AofM \right)  \; ,
	\]
	as then we may just define ${\Phi}_{x_0}$ to be the restriction of $\widetilde{\Phi}_{x_0}$ on $\AofM[V]$. 
	
	To this end, we observe that for any $x_1 \in M$ and any $f \in \salg$, we have 
	\begin{align*}
		&~ \widetilde{\Phi}_{x_0}^{-1} \left(\beta_{x_1} (f)\right) \\
		=&~ \left( f \left( \widetilde{\Phi}_{x_0}^{-1} \left( \cliffmult[x_1]^{M} \right) (v,t)  \right) \right)_{(v,t) \in V \times [0,\infty)} \\
		= &~ \left( f \left(   - \varphi_v^{-1} \circ \log_{\exp_{x_0}(v)} \left( x_1 \right) \oplus t \right)   \right)_{(v,t) \in V \times [0,\infty)} \; ,
	\end{align*}
	which falls in $\AofM[V]$ because of Proposition~\ref{prop_AofM-Euclidean} and the fact that the function 
	\[
		V \times [0,\infty) \ni (v,t) \mapsto \left(   - \varphi_v^{-1} \circ \log_{\exp_{x_0}(v)} \left( x_1 \right) \oplus t \right) \in V \oplus \rbbd
	\] 
	is continuous and proper. 
	This implies $\AofM[V]  \supset \left(\widetilde{\Phi}_{x_0}\right)^{-1} \left(\AofM \right)$. 
	
	To prove the other direction, we observe from Lemma~\ref{lem:AofMeven-fin-dim} that $\widetilde{\Phi}_{x_0}$ restricts to a $*$\-/isomorphism between $\AevenofM[V] = C_0(V \times [0,\infty))$ and $\AevenofM[M] = C_0(M \times [0,\infty))$ such that the corresponding maps between $V \times [0,\infty)$ and $M \times [0,\infty)$ are induced by $\exp_{x_0}$ and $\log_{x_0}$. Since $\AofM$ is an $(M \times [0,\infty))$-$C^*$-algebra, $\widetilde{\Phi}_{x_0}^{-1} \left( \AofM \right)$ is a $(V \times [0,\infty))$-$C^*$-algebra, just as $\AofM[V]$. For any $(v, t) \in V \times [0,\infty)$, since 
	\[
		\left\{ C_{x_1} (\exp_{x_0} (v),t) \colon x_1 \in M \right\} = (- T_{\exp_{x_0} (v)} M) \times \{ t \} \; ,
	\]
	which is a generating subset of $\cliffc (\hhil_{\exp_{x_0} (v)} M \oplus t \rbbd )$, it follows that the fiber of $\AofM$ at $(\exp_{x_0} (v),t)$ is $\cliffc (\hhil_{\exp_{x_0} (v)}  M\oplus t \rbbd )$, and thus the fiber of $\widetilde{\Phi}_{x_0}^{-1} \left( \AofM \right)$ at $\left(v, t \right)$ is given by $\left(\cliffc \left(\varphi_{v} \oplus \idmap_{t\rbbd} \right) \right)^{-1} \left(\cliffc (\hhil_{\exp_{x_0} (v)} M \oplus t \rbbd )\right)$, which is the same as the corresponding fiber of $\AofM[V]$, namely $\cliffc (V  \oplus t \rbbd )$. 
	Therefore $\AofM[V]  = \left(\widetilde{\Phi}_{x_0}\right)^{-1} \left(\AofM \right)$ as $C^*$-subalgebras of $\pialg(V)$, as desired. 
\end{proof}

\section{Isometries of $M$ and automorphisms of $\AofM$} \label{sec:automorphisms}

In this section, we discuss how isometries of a \hhs{} $M$ act on $\AofM$ by automorphisms and study continuity and properness of this action.

\begin{constr}\label{constru:Pi-M-isometry}
	Given an isometry $\varphimap$ of $\mnf$, we may construct a graded $*$\-/automorphism $\varphimap_*$ of $\piunbdd(M)$ by 
	\begin{equation}\label{eq:action-Clifford-algebra}
	\varphimap_* (\sigfunc) (\xpt, t) = \cliffc \left( \hhil \deriv_{\varphimap \inv (\xpt)} \varphimap  \oplus \idmap_{t \rbbd} \right) \; \left( \sigfunc \left(\varphimap \inv (\xpt), t \right) \right) 
	\end{equation}
	for any $\xpt \in \mnf$, $t \in [0,\infty)$ and $\sigfunc \in \piunbdd(M)$, where
	\begin{itemize}
		\item $\deriv_{\varphimap \inv (\xpt)} \varphimap \colon \tanbndl_{\varphimap \inv (\xpt)} \mnf \to \tanbndl_{\xpt} \mnf$ is the derivative of $\varphimap$ at $\xpt$, which is an isometric bijection fixing the base point, 
		\item $\hhil\deriv_{\varphimap \inv (\xpt)} \varphimap \colon \hhil_{\varphimap \inv (\xpt)} \mnf \to \hhil_{\xpt} \mnf$ is the induced isometric isomorphism, and
		\item $\cliffc \left( \hhil \deriv_{\varphimap \inv (\xpt)} \varphimap  \oplus \idmap_{t \rbbd}  \right)$ is the induced graded $*$\-/isomorphism between the corresponding Clifford algebras. 
	\end{itemize}  
	
	It is clear that the $*$\-/subalgebra $\pialg (M)$ is invariant under this action.
	
	The assignment $\varphimap \mapsto \varphimap_*$ give rise to group homomorphisms from the isometry group $\isomgrp(\mnf)$ to the group $\autgrp(\piunbdd(M))$ of $*$\-/automorphisms of $\piunbdd(M)$, as well as $\autgrp(\pialg(M))$. 
	
\end{constr}

Next we study the relation between Bott homomorphisms and isometries on $\mnf$. 

\begin{lem}\label{lem_Botthomandgroupaction}
	For any $\varphimap \in \isomgrp(\mnf)$ and any $\xpt_0 \in \mnf$, we have
	\[
	\varphimap_* \circ \botthom[\xpt_0] = \botthom[\varphimap (\xpt_0)] \; ,
	\]
	where $\varphimap_*$ is the induced $*$\-/automorphism of $ \pialg(M)$ defined in \eqref{eq:action-Clifford-algebra}.
\end{lem}

\begin{proof}
	For any $ \ffunc \in \salg $, $\xpt\in\mnf$ and $t \in \rbbd$, we have
	\begin{align*}
	\varphimap_* \left( \botthom[\xpt_0] (\ffunc) \right) (\xpt, t) =\ & \cliffc \left( \hhil \deriv_{\varphimap \inv (\xpt)} \varphimap  \oplus \idmap_{t \rbbd} \right) \; \left( \botthom (\ffunc) \left(\varphimap \inv (\xpt), t \right) \right)  \\
	=\ & \cliffc \left( \hhil \deriv_{\varphimap \inv (\xpt)} \varphimap  \oplus \idmap_{t \rbbd} \right) \; \left( \ffunc \left( \cliffmult \left(\varphimap \inv (\xpt), t \right) \right)  \right)  \\
	=\ & \ffunc \Big( \cliffc \left( \hhil \deriv_{\varphimap \inv (\xpt)} \varphimap  \oplus \idmap_{t \rbbd} \right) \;  \left( \cliffmult \left(\varphimap \inv (\xpt), t  \right) \right) \Big) \\
	=\ & \ffunc \left( \left( \deriv_{{\varphimap \inv (\xpt)}} \varphimap  \right) \; \left(  -\log_{\varphimap \inv (\xpt)}  (\xpt_0)\right), t \right) \\
	=\ & \ffunc \left( - \log_\xpt (\varphimap (\xpt_0)) , t \right) \\
	=\ & \ffunc \left( \cliffmult[\varphimap (\xpt_0)](\xpt, t) \right) \\
	=\ & \botthom[\varphimap (\xpt_0)](\ffunc)(\xpt, t).
	\end{align*}
	Here we used the fact that functional calculus commutes with automorphisms of $\cstar$-algebras, and that the isometry $ \varphimap \colon \mnf \to \mnf $ maps the geodesic segment $ [ \varphimap \inv (\xpt), \xpt_0 ] $ to $ [ \xpt, \varphimap (\xpt_0) ] $. 
\end{proof}

\begin{rmk}\label{rmk_alternativeproof_lem_Botthomandgroupaction}
	Somewhat underlying Lemma~\ref{lem_Botthomandgroupaction} is the fact that 
	$$ \varphimap_* \left( \cliffmult \right) = \cliffmult[\varphimap (\xpt_0)] $$
	for any $ \xpt_0 \in \mnf $ and $ \gingam $.  
\end{rmk}

\begin{prop}\label{prop_AofMisinvariant}
	For any isometry $\varphimap$ of $\mnf$, the induced $*$\-/automorphism $\varphimap_*$ of $ \pialg(M) $ given in Equation~\eqref{eq:action-Clifford-algebra} preserves $\aalg(\mnf)$ and $ \AevenofM $.
\end{prop}
\begin{proof}
	This follows directly from Lemma~\ref{lem_Botthomandgroupaction}.  
\end{proof}

Recall that $\widehat{\AevenofM}$ is the spectrum of the central $C^*$-subalgebra ${\AevenofM}$. The action of $\isomgrp(\mnf)$ on $\AofM$ induces actions on ${\AevenofM}$ by $*$\-/automorphisms and on $\widehat{\AevenofM}$ by homeomorphisms. 

\begin{cor}\label{cor:AofM-isom-AevenofM-algebra}
	The {\cstaralg} $\AofM$ is an $\isomgrp(\mnf)$-$\widehat{\AevenofM}$-$\cstar$-algebra.
\end{cor}
\begin{proof}
	This combines Proposition~\ref{prop_AofMisAevenofMalgebra} and Proposition~\ref{prop_AofMisinvariant}. 
\end{proof}

One of the key properties of $\AofM$ is that any isometric, metrically proper action on $\mnf$ by a discrete group $\gamgrp$ induces by means of Proposition~\ref{prop_AofMisinvariant} a proper action on $\AofM$. Here $ \AevenofM $ will play an important role, as we will show the action of $\gamgrp$ on the spectrum of $\AevenofM$ is proper.

\begin{lem}\label{lem_Cstarcharacterizationofproperaction}
	Let $\xsp$ be a locally compact Hausdorff space and let $\gamgrp$ be a discrete group. Let $\alfmap \colon \gamgrp \curvearrowright \xsp$ be an action by homeomorphisms and let $\alfmap_*$ be the induced action on $\cz(\xsp)$. Then this action is (topologically) proper if and only if for any $\ffunc\in\cz(\xsp)$,
	$$\lim_{\gamelem\to\infty} \| \big( (\alfmap_*)_\gamelem (\ffunc) \big) \cdot \ffunc \|\to 0 \; ,$$
	i.e., for any $\varepsilon>0$, there is a finite subset $\fset\subset\gamgrp$ such that for any $\gamelem \in \gamgrp \setminus \fset$,
	$$\| \big( (\alfmap_*)_\gamelem (\ffunc) \big) \cdot \ffunc \| < \varepsilon . $$
\end{lem}
\begin{proof} 
	If the action $\gamgrp \curvearrowright \xsp$ is proper, then for any $\ffunc\in\cc(\xsp)$, there is a finite subset $\fset\subset\Gam$ such that for any $\gamelem \in \gamgrp \setminus \fset$, $\|(\alfmap_*)_{\gamelem} (\ffunc) \cdot \ffunc \| =0$. The statement for a general $\ffunc\in\cz(\xsp)$ follows by approximation.
	
	On the other hand, if every element of $\cz(\xsp)$ satisfies the condition in the statement, then for any compact subset $\kset\subset\xsp$, picking a positive function $\ffunc\in\cz(\xsp)$ such that $\ffunc(\xpt)\ge1$ for $\xpt\in\kset$, we can find, according to the condition, a finite $\fset\subset\gamgrp$ such that for any $\gam\not\in\fset$, $\| (\alfmap_*)_{\gamelem} (\ffunc) \cdot \ffunc\| < \frac{1}{2}$, which implies that $\alfmap_{\gamelem}(\kset) \cap \kset = \varnothing$. 
\end{proof}

\begin{prop}\label{prop_AofMisproper}
	Let $\gamgrp$ be a discrete group and $\alfmap \colon \gamgrp \to \isomgrp(\mnf)$ an isometric, metrically proper action on $\mnf$. Then the induced action on $\AofM$ (also denoted by $\alfmap$) given in Proposition~\ref{prop_AofMisinvariant} makes $\AofM$ into a proper $\Gamma$-$\widehat{\AevenofM}$-\cstaralg.
\end{prop}

\begin{proof}
	Observe that for any even and compactly supported function $\ffunc \in \salg$, the element $\botthom (\ffunc) \in \AevenofM \subset \pialg(M)$ is, as a function over $M \times [0, \infty)$, supported in a bounded ball around $(x_0, 0)$. Thus because of the metric properness of the action $\Gam\curvearrowright\mnf$, all but finitely many elements $\gamelem$ of $\gamgrp$ satisfy $(\alfmap_{\gamelem} (\botthom (\ffunc))) \cdot \botthom (\ffunc) = 0$. Since any even function $\ffunc \in \salg$ is approximated by the compactly supported even ones, every element $\sigma$ of $\AevenofM$ satisfies 
	\[
	\lim_{\gamelem\to\infty} \|(\alfmap_{\gamelem} (\sigma)) \cdot \sigma \| = 0 \; .
	\]
	This ensures the action of $\gamgrp$ on the spectrum of $\AevenofM$ is (topologically) proper by Lemma~\ref{lem_Cstarcharacterizationofproperaction}, i.e., $\AevenofM$ is a commutative proper $\Gamma$-\cstaralg. It follows that $\AofM$ is a proper $\Gamma$-$\widehat{\AevenofM}$-\cstaralg.
\end{proof}

Finally, we discuss the topological aspect of the action of $\isomgrp(M)$ on $\AofM$ by automorphisms. This will be crucial for our deformation technique in Section~\ref{sec:proof}. 

\begin{constr}
	Recall from Definition~\ref{defn:isom-M} that $\isomgrp(M)$ is equipped with the topology of pointwise convergence. Similarly, let us endow $\autgrp(\AofM)$, the group of $*$\-/automorphisms of $\AofM$, with the \emph{topology of pointwise (norm) convergence}, so that a net $\{\varphimap_\iind\}_{\iiniind}$ converges to the identity if and only if $\lim_{\iiniind} \varphimap_\iind (\aelem) = \aelem$ in norm for any $\aelem \in \AofM$. Note that it suffices to check the latter condition for any $\aelem$ in a generating set of $\AofM$, e.g., for all $\aelem$ of the form $\botthom(\ffunc)$ for $\xpt_0 \in \mnf$ and $\ffunc \in \salg$. 
\end{constr}

\begin{prop}\label{prop:isom-M-aut-AofM}
	When both $\isomgrp(\mnf)$ and $\autgrp(\AofM)$ carry the {topology of pointwise convergence}, the canonical homomorphism $\isomgrp(\mnf) \to \autgrp(\AofM)$ defined in Proposition~\ref{prop_AofMisinvariant} is continuous. 
\end{prop}

\begin{proof}
	It suffices to show that for any net $\{\varphimap_\iind\}_{\iiniind}$ in $\isomgrp(\mnf)$ that converges to the identity, the induced net $\{ (\varphimap_\iind)_* \}_{\iiniind}$ in $\autgrp(\AofM)$ also converges to the identity. Since $\AofM$ is generated by $\botthom(\ffunc)$ for $\xpt_0 \in \mnf$ and $\ffunc \in \salg$, it suffices to check 
	\[
	\lim_{\iiniind} (\varphimap_\iind)_* (\botthom(\ffunc)) = \botthom(\ffunc)
	\]
	for any $\xpt_0 \in \mnf$ and $\ffunc \in \salg$. By Lemma~\ref{lem_Botthomandgroupaction}, the left-hand side is equal to $\lim_{\iiniind} \botthom[\varphimap_\iind (\xpt_0)](\ffunc)$, which by Lemma~\ref{cor_botthomiscontwrtbasepoint} is equal to the right-hand side, as, by assumption, $\lim_{\iiniind} \varphimap_\iind (\xpt_0) = \xpt_0$. 
\end{proof}

\section{The $K$-theory of $\aalg(\mnf)$} \label{sec:K-theory}

In this section, we discuss the computation of the $K$-theory of $\AofM$. We shall see that when $M$ is finite\-/dimensional, the Bott homomorphism $\botthom \colon \salg \to \AofM$ induces an isomorphism on $K$-theory, which can be seen as a version of Bott periodicity. This statement remains true when $M$ is a separable Hilbert space (cf.\,\cite{higsonkasparov}). For general Hilbert-Hadamard spaces, however, the problem of computing the $K$-theory of $\AofM$ remains open. We provide a partial solution to this problem by using an approximation technique to show that the Bott homomorphism induces an injection on $K$-theory when the \hhs{} $M$ is admissible. This is the only place we need the admissibility condition. It is an open question whether the injectivity of the Bott homomorphism on $K$-theory remains true without this condition.

\begin{defn}\label{defn_AofM_relative}
	Let $\nmnf \subset \mnf$ be a subset. We define $\AofMrel{\nmnf} $ to be the $C^*$-subalgebra of $\AofM$ generated by 
	\[
	\{ \botthom[\xpt_0](\ffunc) \colon \xpt_0 \in \nmnf \subset \mnf,\ \ffunc \in \salg \} .
	\]
	Likewise, we define $ \AevenofMrel{\nmnf}$ to be the $C^*$-subalgebra of $\AofMrel{\nmnf} $ generated by 
	\[
	\{ \botthom[\xpt_0](\ffunc) \colon \xpt_0 \in \nmnf \subset \mnf,\ \ffunc \in \salg_\text{ev} \} .
	\]
\end{defn}

We list some immediate consequences of the definition. 

\begin{lem}\label{lem:AofMrel-basic}
	Let $\nmnf_1, \nmnf_2, \ldots$ be subsets of $\mnf$.
	\begin{enumerate}
		\item $\AofMrel{\mnf} = \AofM$. 
		\item If $\nmnf_1 \subset \nmnf_2$ then $\AofMrel{\nmnf_1} \subset \AofMrel{\nmnf_2}$. 
		\item If $\overline{\nmnf}$ is the closure of $\nmnf$, then $\AofMrel{\overline{\nmnf}} = \AofMrel{\nmnf}$. 
		\item If $\nmnf_1 \subset \nmnf_2 \subset \ldots$, then $ \AofMrel{\overline{ \bigcup_{\kind=1}^\infty \nmnf_\kind}}$ is the direct limit of the sequence $\AofMrel{\nmnf_1} \subset \AofMrel{\nmnf_2} \subset \ldots$ of $C^*$-subalgebras. 
	\end{enumerate}
\end{lem}

\begin{proof}
	The first and second claims are immediate from the definition. The third claim is a consequence of Corollary~\ref{cor_botthomiscontwrtbasepoint}. The last claim follows from the second and the third. 
\end{proof}

The construction of $\AofMrel{\nmnf}$ is particularly interesting when $\nmnf$ is a closed convex subset of $\mnf$. In this case, $\nmnf$ is again a {\hhs}; thus we can compare the algebras $\AofMrel{\nmnf}$ and $\AofM[\nmnf]$. Consider the $C^*$-algebra
\[
	\pialg(N; M) = \left\{ \sigma \in \prod_{(x,t) \in N \times [0,\infty) } \cliffc (\hhil_x M \oplus t \rbbd ) \colon \sup_{(x,t) \in N \times [0,\infty)} \| \sigma({x,t}) \| < \infty \right\}
\]
together with the natural quotient map 
\[
	\pimap_{\mnf, \nmnf} \colon \pialg(\mnf) \to \pialg(N; M)
\]
and the natural embedding
\[
	\iotmap_{\mnf, \nmnf} \colon \pialg(N) \to \pialg(N; M)
\]
induced from the embeddings $\hhil_x N \hookrightarrow \hhil_x M$ for all $x \in N$. 

\begin{lem}\label{lem:AofMrel-agree}
	Let $\nmnf$ be a closed convex subset of $\mnf$. Then for any $\xpt_0 \in \nmnf$, we have
	\[
	\pimap_{\mnf, \nmnf} \circ \botthom[\xpt_0]^\mnf = \iotmap_{\mnf, \nmnf} \circ \botthom[\xpt_0]^\nmnf \; ,
	\]
	where $\botthom[\xpt_0]^\mnf$ is the Bott homomorphism into $\AofMrel{\nmnf}$, which in turn is contained in $\AofM$ and $\pialg(M)$, and $\botthom[\xpt_0]^\nmnf$ is the Bott homomorphism into $\AofM[\nmnf]$, viewed as a subalgebra of $\pialg(N)$. In particular,
	\[
	\pimap_{\mnf, \nmnf} (\AofMrel{\nmnf}) = \iotmap_{\mnf, \nmnf} (\AofM[\nmnf]) \; .
	\]
\end{lem}

\begin{proof}
	We check that for any $(\xpt, \tvar) \in \nmnf \times [0,\infty)$ and any $\ffunc \in \salg$, we have, inside $\cliffc (\hhil_x M \oplus \rbbd )$,
	\begin{align*}
		\pimap_{\mnf, \nmnf} \circ \botthom^\mnf  (\ffunc) (\xpt, \tvar) = \pimap_{\mnf, \nmnf} \left( \ffunc(  \cliffmult^\mnf (\xpt, \tvar) ) \right) = \ffunc \left(  \cliffmult^\mnf (\xpt, \tvar) \right)
	\end{align*}
	and 
	\begin{align*}
		\iotmap_{\mnf, \nmnf} \circ \botthom^\nmnf  (\ffunc) (\xpt, \tvar) = \iotmap_{\mnf, \nmnf} \left( \ffunc(  \cliffmult^\nmnf (\xpt, \tvar) ) \right) = \ffunc \left(  \cliffmult^\nmnf (\xpt, \tvar) \right) \;.
	\end{align*}
	They give the same element because the geodesic segment in $\mnf$ connecting $\xpt_0$ to $\xpt$ coincides with the geodesic segment in the convex subset $\nmnf$ connecting the same two points. 
\end{proof}

\begin{lem}\label{lem:fin-dim-$K$-theory}
	If $\mnf$ is a complete, connected and simply connected (finite\-/dimensional) Riemannian manifold with non-positive sectional curvature, then for any $\xpt_0 \in \mnf$, the Bott homomorphism $\botthom  \colon \salg \to \AofM$ induces isomorphisms on $K$-theory as well as on $K$-theory with real coefficients. 
\end{lem}

\begin{proof}
	Write $V$ for the tangent space of $M$ at $x_0$. By Proposition~\ref{prop_AofM-fin-dim}, it suffices to show the homomorphism 
	\[
		\botthom[0]^{V} \colon \salg \to \AofM[V]
	\]
	induces isomorphisms on $K$-theory and $K$-theory with real coefficients. In view of Remark~\ref{rmk:AofM-fin-dim-graded}, the isomorphism on $K$-theory follows from the Bott Periodicity Theorem \cite[\S 2, 6]{higsonkasparovtrout}. The case for $K$-theory with real coefficients then follows since both algebras are type I \textemdash\ all of their irreducible representations factor through finite-dimensional Clifford algebras\textemdash\ and thus in the bootstrap class, and hence the natural map in Construction~\ref{constr:KKR} from $K$-theory tensored with $\rbbd$ to $K$-theory with real coefficients is an isomorphism. 
\end{proof}

\begin{lem}\label{lem:botthom-homotopy}
	For any two points $\xpt_0, \xpt_1 \in \mnf$, the Bott homomorphisms
	\[
	\botthom[\xpt_0], \botthom[\xpt_1] \colon \salg \to \AofM[\mnf]
	\]
	are homotopic to each other. 
\end{lem}

\begin{proof}
	Let $(\xpt_\svar )_{\svar \in [0,1]}$ be a path in $\mnf$ connecting $\xpt_0$ and $\xpt_1$ (e.g., the geodesic segment between the two points). By Corollary~\ref{cor_botthomiscontwrtbasepoint}, the family $\big( \botthom[\xpt_\svar] \big)_{\svar \in [0,1]}$ constitutes a homotopy between $\botthom[\xpt_0]$ and $\botthom[\xpt_1]$.
\end{proof}

Recall that $\mnf$ is said to be \emph{admissible} if there is a sequence $\mnf_1 \subset \mnf_2 \subset \ldots$ of closed convex subsets isometric to finite\-/dimensional Riemannian manifolds, whose union is dense in $M$.

\begin{prop} \label{prop:botthom-K-inj}
	Suppose that $\mnf$ is admissible. Then for any $\xpt_0 \in \mnf$, the Bott homomorphism 
	\[
	\botthom  \colon \salg \to \AofM
	\]
	induces injections on $K$-theory as well as $K$-theory with real coefficients (Construction~\ref{constr:KKR}), that is, the induced homomorphisms
	\[
		K_i (\salg) \overset{(\botthom)_*}{\longrightarrow} K_i (\AofM) \quad \text{and} \quad K_{\rbbd, i} (\salg) \overset{(\botthom)_*}{\longrightarrow} K_{\rbbd, i} (\AofM)
	\] 
	are injective. 
\end{prop}

\begin{proof}
	Let $\mnf_1 \subset \mnf_2 \subset \ldots$ be a sequence of closed convex subsets isometric to finite\-/dimensional Riemannian manifolds such that $\displaystyle \mnf = \overline{\bigcup_{\kind=1}^\infty \mnf_\kind}$. By Lemma~\ref{lem:botthom-homotopy}, the Bott homomorphisms associated to any two base points agree on $K$-theory. Hence we may assume without loss of generality that $\xpt_0 \in \tanbndl\mnf_1$. By Lemma~\ref{lem:AofMrel-basic}, we see that $\AofM$ is the direct limit of the increasing sequence of subalgebras $\AofMrel{\mnf_1} \subset \AofMrel{\mnf_2} \subset \ldots$. Since the image of $\botthom$ is contained in $\AofMrel{\mnf_\kind}$ for any $\kind \in \znum{>0}$, by the continuity of the $K$-theory functor with regard to direct limits, it suffices to show that 
	\[
	\botthom  \colon \salg \to \AofMrel{\mnf_\kind}
	\]
	induces an injection on $K$-theory and $K$-theory with real coefficients for every $\kind \in \znum{>0}$. To this end, we fix an arbitrary $\kind \in \znum{>0}$ and observe that Lemma~\ref{lem:AofMrel-agree} yields a commutative diagram 
	\[
		\xymatrix{
				\salg \ar[r]^{\botthom} \ar[d]_{\botthom^{\mnf_\kind}} & \AofMrel{\mnf_\kind} \ar[d]^{\pimap_{\mnf, \mnf_\kind}} \\
				\AofM[\mnf_\kind] \ar[r]_-{\iotmap_{\mnf, \mnf_\kind}} &	 \iotmap_{\mnf, \mnf_\kind} \left( \AofM[\mnf_\kind] \right) = \pimap_{\mnf, \mnf_\kind} \left( \AofMrel{\mnf_\kind} \right) \subseteq \pialg({\mnf_\kind; \mnf})
			}
	\]
	of homomorphisms of $C^*$-algebras. Its lower horizontal arrow is an isomorphism. By Lemma~\ref{lem:fin-dim-$K$-theory}, the left vertical arrow induces isomorphisms on $K$-theory as well as $K$-theory with real coefficients. It follows that $\botthom$ induces injections on $K$-theory as well as $K$-theory with real coefficients, as desired. 
\end{proof}

The following remark is not essential to the proofs of our main theorems, but it helps to connect our construction of $\AofM[\hhil]$ for a separable Hilbert space $\hhil$ with those of $\mathscr{SC}(\hhil)$ in \cite[\S 3, 3]{higsonkasparovtrout} and $\widetilde{\mathscr{A}}(\hhil)$ in \cite[4.3]{higsonkasparov}. 

\begin{rmk}\label{rmk:AofM-complemented}
	We say that a closed convex subset $N$ of $M$ is \emph{complemented} if there exists a closed convex subset $N'$ of $M$ such that $N \cap N' = \left\{ x_0 \right\}$ for some $x_0 \in M$ and there is an isometry $M \xrightarrow{\simeq} N \times N'$ (equipped with the $\ell^2$-product metric) sending $N \subseteq M$ to $N \times \{x_0\} \subseteq N \times N'$.
	
	We claim that when $N$ is complemented, the quotient map $\pimap_{\mnf, \nmnf}$ maps $\AofMrel{\nmnf}$ isomorphically onto $\iotmap_{\mnf, \nmnf} (\AofM[\nmnf])$; thus we have $\AofMrel{\nmnf} \cong \AofM[\nmnf]$. To see this, we follow Remark~\ref{rmk:botthom-alternative} and introduce, for any $(x',t) \in N' \times [0, \infty)$, the isometric linear embedding 
	\[
		\tau_{x_0}^{N'}(x',t) \colon \dist_{x_0} (x' , t) \cdot \rbbd \to \hhil_{x'} N' \oplus t \rbbd 
	\]
	sending $\dist_{x_0} (x' , t)$ to $\cliffmult^{N'} (x' , t)$, where $\dist_{x_0} (x' , t) = \sqrt{ d(x', x_0)^2 + t^2 }$. 
	Together they make up a $*$\-/homomorphism
	\begin{align*}
		B_{N}^{M} \colon \Pi(N) &\to \Pi(N \times N') \cong \Pi(M) \\
		\sigma &\mapsto \bigg( \cliffc \left( \idmap_{ \hhil_x N } \oplus \tau_{x_0}^{N'} (x',t) \right) \left( \sigma (x, \dist_{x_0}(x',t)) \right) \bigg)_{(x, x',t) \in N \times N' \times [0, \infty)} \; ,
	\end{align*}
	where 
	\[
		\cliffc \left( \idmap_{ \hhil_x N } \oplus \tau_{x_0}^{N'} (x',t) \right) \colon \cliffc ( \hhil_x N \oplus \dist_{x_0}(x,t) \rbbd) \to \cliffc \left( \hhil_x N \oplus \hhil_{x'} N' \oplus t \rbbd \right)
	\]
	is the induced $*$\-/homomorphism between the Clifford algebras. It follows from straightforward computations that 
	\[
		\pimap_{\mnf, \nmnf} \circ B_{N}^{M} = \iotmap_{\mnf, \nmnf} \colon \Pi_{\operatorname{b}}(N) \to \Pi_{\operatorname{b}}(N; M) 
	\]
	and $B_{N}^{M} (\cliffmult[x]^N) = \cliffmult[x]^M$ for any $x \in N$. The latter implies $B_{N}^{M} \circ \botthom[x]^{N} = \botthom[x]^{M}$ for any $x \in N$ and thus
	\[
		B_{N}^{M} (\AofM[N]) = \AofMrel{N} \; .
	\]
	These facts prove the above claim. 
	
	In the case when $M$ is a separable Hilbert space, every closed convex subset $N$ is an affine subspace and is clearly complemented. Thus choosing an increasing sequence $(M_k)_{k \in \nbbd}$ of finite\-/dimensional affine subspaces with $M = \overline{\bigcup_{k\in \nbbd} M_k}$, we have
	\[
		\AofM = \overline{\bigcup_{k\in \nbbd} B_{M_k}^{M} (\AofM[M_k])} \; .
	\]
	In view of the natural identification  
	\[
		\cliffc \left( \idmap_{ \hhil_x N } \oplus \tau_{x_0}^{N'} (x',t) \right) = \idmap_{ \cliffc (\hhil_x N) } \widehat{\otimes} \cliffc \left( \tau_{x_0}^{N'} (x',t) \right) \; , 
	\]
	it is not hard to see that the map $B_{N}^{M} \colon \AofM[N] \to \AofMrel{N} \subset \AofM$ agrees with the connecting map $\AofM[V'] \to \AofM[V]$ given in Definition~4.4 of \cite{higsonkasparov}. It follows that in this case, our construction of $\AofM$ agrees with the algebras $\mathscr{SC}(M)$ in \cite[\S 3, 3]{higsonkasparovtrout} and $\widetilde{\mathscr{A}}(M)$ in \cite[4.3]{higsonkasparov}. 
\end{rmk}

\section{The proofs of the main theorems} \label{sec:proof}
In this section, we prove our main results, which will make use of the various ingredients from the previous sections. More precisely, suppose that a discrete group $\Gamma$ acts on a {\hhs} $M$ properly and isometrically. Then, Proposition~\ref{prop_AofMisproper} gives one hope to apply the standard Dirac-dual-Dirac method (cf. ~\cite{kasparov1,kasparov95}; also see \cite[Chapter~9]{Valette2002}) to the $C^*$-algebra $\AofM$ in order to prove the injectivity of the assembly map $\mu \colon KK^\Gamma_i(\univspproper\Gamma) \to K_i(C^*_{\operatorname{r}} \Gamma)$ via Theorem~\ref{thm:proper-GHT} and the commutative diagram in \eqref{eq:BC-assembly-natural}. However, we are not able to directly apply this powerful method since we have not been able to compute the $K$-theory of $\AofM$ in general. We have only obtained some partial information from Proposition~\ref{prop:botthom-K-inj}. 

To circumvent this problem, we amplify the {\hhs} $M$ to a bigger and necessarily infinite\-/dimensional {\hhs} $M^{[0,1]}$, on which $\Gamma$ still acts properly and isometrically, and then employ a deformation technique to simplify the calculation of the equivariant $KK$-groups involving the $C^*$-algebra $\AofM[M^{[0,1]}]$. 
To formalize this deformation technique, we introduce the following $C^*$-algebra. 

\newcommand{\AIofM}[1][\mnf]{\mathcal{A}_{[0,1]}(#1)}

\begin{constr}
	Let us fix a {\hhs} $M$.  
	Recall from Proposition~\ref{prop:isom-01-nilhomotopic} that we write $M^{[0,1]}$ for the continuum product $L^2([0,1], m , M)$. We define 
	\[
		\AIofM = C \left( [0,1], \AofM[{M^{[0,1]}}] \right) \; ,
	\]
	that is, the {\cstaralg} of all continuous functions with values in $\AofM[{M^{[0,1]}}]$,
	and equip it with the action $\alpha_{[0,1]}$ of $\isomgrp(M)$: for any $\varphi \in \isomgrp(M)$ and any $f \in C \left( [0,1], \AofM[{M^{[0,1]}}] \right)$, we define $\varphi \cdot_{\alpha_{[0,1]}} f$ by
	\[
		\left(\varphi \cdot_{\alpha_{[0,1]}} f\right) (t) = H(\varphi, t)_* \, (f(t)) \quad \text{for any } t \in [0,1] \; ,
	\]
	where $H \colon \isomgrp(M) \times [0,1] \to \isomgrp\left(M^{[0,1]} \right)$ is the homotopy given in Proposition~\ref{prop:isom-01-nilhomotopic} and the lower $*$ denotes the induced element in $\autgrp(\AofM[{M^{[0,1]}}])$. Note that the continuity statement in Proposition~\ref{prop:isom-01-nilhomotopic}, together with Proposition~\ref{prop:isom-M-aut-AofM}, guarantees this action is well-defined and continuous. 
	
	For each $t \in [0,1]$, there is an evaluation map
	\[
		\operatorname{ev}_t \colon \AIofM \to \AofM[{M^{[0,1]}}] \;, \quad f \mapsto f(t) \; ,
	\]
	which clearly intertwines the actions $\alpha_{[0,1]}$ and 
	\[
		\alpha_t \colon \isomgrp(M) \curvearrowright \AofM[{M^{[0,1]}}] \; , \quad \varphi \cdot_{\alpha_{t}} a = H(\varphi, t)_* \, (a) \; . 
	\]
\end{constr}

\begin{rmk}\label{rmk:AIofM-ev}
	It is clear that for any $t \in [0,1]$, the evaluation map $\operatorname{ev}_t$ is a homotopy equivalence and thus induces an isomorphism on (non-equivariant) $K$-theory. 
\end{rmk}

\begin{lem}\label{lem:AIofM-ev-KK}
	Let $\Gamma$ be a subgroup of $\isomgrp(M)$ and let $X$ be a free and proper $\Gamma$-space. Then for any $t \in [0,1]$, the evaluation maps 
	\[
		\operatorname{ev}_t \colon \AIofM \to \AofM[{M^{[0,1]}}]
	\]
	induce isomorphisms
	\[
		(\operatorname{ev}_t)_* \colon KK^{\Gamma, \alpha_{[0,1]}}_i \left(X, \AIofM \right) \overset{\cong}{\longrightarrow} KK^{\Gamma, \alpha_t}_i \left(X, \AofM[{M^{[0,1]}}] \right)
	\]
	for $i=0,1$, where the superscripts $\alpha_{[0,1]}$ and $\alpha_t$ are inserted to specify the actions of $\Gamma$ on the $C^*$-algebras $\AIofM$ and $\AofM[{M^{[0,1]}}]$. 
\end{lem}

\begin{proof}
	By Definition~\ref{defn:KK-Gam-compact}, it suffices to prove the lemma in the case when $X$ is $\Gamma$-compact. Combining this additional condition with the slice lemma (Lemma~\ref{lem:slice}), we see that such an $X$ can be written as the union of finitely many $\Gamma$-invariant open subsets $W_1, \ldots, W_k$, with each $W_j$ being a disjoint union of $\Gamma$-translates of a single open subset $U_j$, that is, $W_j$ is equivariantly homeomorphic to $\Gamma \times U_j$, with $\Gamma$ acting by translation on the first factor of the Cartesian product. Thus each $W_j$ determines a $\Gamma$-invariant ideal $C_0 (W_j)$ in $C_0(X)$ that is equivariantly isomorphic to $C_0(\Gamma, C_0 (U))$. 
	Remark~\ref{rmk:KK-facts-de-equivariantize}\eqref{rmk:KK-facts-de-equivariantize:translation-A} then yields a commutative diagram
	\[
		\xymatrix{
				KK^{\Gamma, \alpha_{[0,1]}}_i \left( C_0(W_j), \AIofM \right) \ar[r]^{(\operatorname{ev}_t)_*} \ar[d]^{\cong} & KK^{\Gamma, \alpha_t}_i \left( C_0(W_j) , \AofM[{M^{[0,1]}}] \right) \ar[d]^{\cong} \\
				KK_i \left( C_0(U_j), \AIofM \right) \ar[r]^{(\operatorname{ev}_t)_*} & KK_i \left( C_0(U_j) , \AofM[{M^{[0,1]}}] \right)
			}
	\]
	Remark~\ref{rmk:AIofM-ev} implies that the bottom map is an isomorphism, and thus so is the top map. 
	Now, for any $j \in \{2, \ldots, k\}$, the two $\Gamma$-invariant ideals $C_0 (W_1 \cup \ldots \cup W_{j-1})$ and $C_0 (W_{j})$ in the commutative proper $\Gamma$-$C^*$-algebra $C_0 (W_1 \cup \ldots \cup W_{j})$ give rise to a commutative diagram consisting of two Mayer-Vietoris sequences with regard to the functors $KK^{\Gamma, \alpha_{[0,1]}}_i \left( -, \AIofM \right)$ and $KK^{\Gamma, \alpha_t}_i \left( - , \AofM[{M^{[0,1]}}] \right)$, together with various maps induced by $\operatorname{ev}_t$. 
	By a standard inductive argument using the five lemma (see, for example, \cite{guentnerhigsontrout}), we obtain the desired isomorphism. 
\end{proof}

\begin{constr}
	Let 
	\[
		\sigma \colon \rbbd_+^{\ast} \curvearrowright \salg 
	\]
	be the rescaling action given by 
	\[
		(s \cdot f) (t) = f( s^{-1} t )
	\]
	for any $s \in \rbbd_+^{\ast}$, $f \in \salg$, and $t \in \rbbd$. This action preserves the set of even (respectively, odd) functions. 
\end{constr}

\begin{lem}\label{lem:Omega-Theta-rescale}
	Following the notations of Definition~\ref{defn:Omega-Theta}, for any $\ffunc \in \salg$ and $r \in \rbbd_+^{\ast}$, we have 
	\[
		\oscill{\rdist} (\sigma_s (\ffunc)) = \oscill{s^{-1} \rdist}\ffunc \qquad \text{and} \qquad \meansym{\rdist} (\sigma_s (\ffunc)) = \meansym{s^{-1} \rdist}\ffunc
	\]
	for all $s \in \rbbd_+^{\ast}$ and thus
	\[
		\lim_{s \to \infty} \sup_{\rdist' \leq \rdist} \oscill{\rdist'} (\sigma_s (\ffunc)) = 0  = \lim_{s \to \infty} \sup_{\rdist' \leq \rdist} \meansym{\rdist'} (\sigma_s (\ffunc)) \; .
	\]
\end{lem}

\begin{proof}
	These follow immediately from Definition~\ref{defn:Omega-Theta} and Lemma~\ref{lem:Omega-Theta-lim-r}. 
\end{proof}

\begin{lem}\label{lem:botthom-rescaled-asymptotic-invariant-01}
	Consider $M$ as embedded in $M^{[0,1]}$ as constant functions. For any $x_0 \in M$, let
	\[
		\botthom^{[0,1]} \colon \salg \to \AIofM \; ,
	\]
	be the composition of $\botthom \colon \salg \to \AofM[M^{[0,1]}]$ and the embedding of $\AofM[M^{[0,1]}]$ into $\AIofM$ as constant functions. Then the family 
	\[
		\left\{ {\botthom^{[0,1]}} \circ \sigma_s \right\}_{s \in [1, \infty)}
	\]
	of {\shom}s from $\salg$ to $\AIofM$ is asymptotically invariant in the sense of Construction~\ref{constr:KK-facts-asymptotic} with regard to the action $\alpha_{[0,1]} \colon \isomgrp(M) \curvearrowright \AIofM$ and the trivial action on $\salg$. 
\end{lem}

\begin{proof}
	For any $t \in [0,1]$, $f \in \salg$, $x_0 \in M$ and $\varphi \in \isomgrp(M)$, it follows from Lemma~\ref{lem_Botthomandgroupaction} that 
	\[
		\alpha_{[0,1]}(\varphi) \left( \botthom^{[0,1]}(\sigma_s(f)) \right) (t) = \alpha_{t}(\varphi) \left( \botthom(\sigma_s(f)) \right) = \botthom[H(\varphi,t) \cdot x_0](\sigma_s(f)) \; ,
	\]
	where $H$ is the homotopy in Proposition~\ref{prop:isom-01-nilhomotopic}. Since the map $t \to H(\varphi,t) \cdot x_0$ is continuous, we can define $R>0$ to be the supremum of $d(x_0, H(\varphi,t) \cdot x_0)$ as $t$ ranges over $[0,1]$. Then by Proposition~\ref{prop_Botthomandchangeofbasepoint}, we have 
	\begin{align*}
		&~ \left\| \botthom^{[0,1]}(\sigma_s(f)) - \alpha_{[0,1]}(\varphi) \left( \botthom^{[0,1]}(\sigma_s(f)) \right) \right\| \\
		= &~ \sup_{t \in [0,1]} \left\| \botthom(\sigma_s(f)) - \botthom[H(\varphi,t) \cdot x_0](\sigma_s(f)) \right\| \\
		\leq &~ \sup_{\rdist \leq R} \left( 2 \; \oscill{\rdist} (\sigma_s(f)) + \max\left\{ \oscill{2\rdist}(\sigma_s(f))  , \, \meansym{\rdist} (\sigma_s(f)) \right\} \right)\; ,
	\end{align*}
	which converges to $0$ by Lemma~\ref{lem:Omega-Theta-rescale}.
\end{proof}

\begin{constr}\label{constr:bott-element}
	Thanks to Lemma~\ref{lem:botthom-rescaled-asymptotic-invariant-01}, we may define the \emph{Bott element}
	\[
		[\beta] \in \kkgam[1]( \cbbd, \AIofM ) 
	\]
	as the one induced by the family $\left\{ {\botthom^{[0,1]}} \circ \sigma_s \right\}_{s \in [1, \infty)}$ according to Construction~\ref{constr:KK-facts-asymptotic}. Thus the forgetful map 
	\[
		\kkgam[1]( \cbbd, \AIofM ) \to KK_1 ( \cbbd, \AIofM ) \cong KK_0(\salg, \AIofM ) \; , 
	\]
	maps the Bott element to the class of the Bott homomorphism $\botthom^{[0,1]} \colon \salg \to \AIofM$, for any $x_0 \in M$. 
\end{constr}

Recall from Construction~\ref{constr:KKR} that $ KK^{\Gamma}_{\rbbd,i} (X, A)$ stands for equivariant $KK$-theory with real coefficients and $\Gamma$-compact supports in the space $X$.  

\begin{prop}\label{prop:beta-EGamma-injective}
	The composition of the group homomorphisms 
	\[
		K_{i+1}^{\Gam}(\univspfree\Gamma)\otimes_{\zbbd}\qbbd \overset{[\beta]}{\longrightarrow} KK^{\Gamma}_{i}(\univspfree\Gamma, \AIofM)\vphantom{\otimes_{\zbbd}\qbbd} \otimes_{\zbbd}\qbbd \to KK^{\Gamma}_{\rbbd,i} (\univspfree\Gamma, \AIofM)\vphantom{\otimes_{\zbbd}\qbbd} 
	\]
	is injective, where the first map is given by taking Kasparov product with $[\beta]$ and the second is given by the natural map mentioned in Construction~\ref{constr:KKR}. 
\end{prop}

\begin{proof}
	After composing the homomorphism in question with 
	\[
		(\operatorname{ev}_0)_* \colon KK^{\Gamma, \alpha_{[0,1]}}_{\rbbd,*} \left(\univspfree\Gamma, \AIofM \right) \overset{}{\longrightarrow} KK^{\Gamma, \alpha_0}_{\rbbd,*} \left(\univspfree\Gamma, \AofM[{M^{[0,1]}}] \right) 
	\]
	and identifying the left-hand side with $KK_{*}^{\Gam}(\univspfree\Gamma, \salg)\otimes_{\zbbd}\qbbd$, 
	we obtain a map 
	\[
		KK_{*}^{\Gam}(\univspfree\Gamma, \salg)\otimes_{\zbbd}\qbbd \to KK^{\Gamma, \alpha_0}_{\rbbd,*} \left(\univspfree\Gamma, \AofM[{M^{[0,1]}}] \right) \; ,
	\]
	which is seen to be induced by the Bott homomorphism $\botthom \colon \salg \to \AofM[{M^{[0,1]}}]$ at an arbitrary base point $x_0$. It suffices to show this composition is injective. Since the action $\alpha_0$ is trivial, we obtain a commutative diagram
	\[
		\xymatrix{
				KK_{i}^{\Gam}(\univspfree\Gamma, \salg)\otimes_{\zbbd}\qbbd \ar[r] \ar[d]^{\cong} & KK^{\Gamma, \alpha_0}_{\rbbd,i} \left(\univspfree\Gamma, \AofM[{M^{[0,1]}}] \right) \ar[d]^{\cong} \\
				KK_{i}(B\Gamma, \salg)\otimes_{\zbbd}\qbbd \ar[r] & KK_{\rbbd,i} \left(B\Gamma, \AofM[{M^{[0,1]}}] \right) \\
				\displaystyle \bigoplus_{j \in \zbbd / 2 \zbbd} K_{i-j}(B\Gamma) \otimes_{\zbbd} K_{j}(\salg) \otimes_{\zbbd} \qbbd  \ar[r]  \ar[u]_{\cong} & \displaystyle \bigoplus_{j \in \zbbd / 2 \zbbd} K_{i-j}(B\Gamma) \otimes_{\zbbd} K_{\rbbd, j} \left( \AofM[{M^{[0,1]}}] \right)   \ar[u]_{\cong} 
			}
	\]
	where the upper vertical maps are the natural isomorphisms given by Remark~\ref{rmk:KK-facts-de-equivariantize}\eqref{rmk:KK-facts-de-equivariantize:trivial-B}, the lower vertical maps are the natural isomorphisms given by Lemma~\ref{lem:KK-separate-variables}, and the horizontal maps are induced from the Bott homomorphism $\botthom$ and the change-of-coefficient homomorphisms. It suffices to show the bottom horizontal map is injective. Since this is a homomorphism between $\qbbd$-vector spaces, it suffices to show the maps on the second tensor components, i.e., the compositions 
	\[
		K_j(\salg) \otimes_{\zbbd} \qbbd
		\to K_{j} \left(\AofM[{M^{[0,1]}}]\right)  \otimes_{\zbbd} \qbbd \to K_{\rbbd, j} \left(\AofM[{M^{[0,1]}}]\right) 
	\]
	for $j= 0,1$, are injective. This is clear for $j = 0$ since $K_0(\salg) \cong 0$. As for $j = 1$, we rewrite the composition as 
	\[
		\qbbd \cong K_1(\salg) \otimes_{\zbbd} \qbbd \hookrightarrow K_1(\salg) \otimes_{\zbbd} \rbbd \cong K_{\rbbd, 1}(\salg) \overset{(\botthom)_*}{\longrightarrow} K_{\rbbd, 1} \left(\AofM[{M^{[0,1]}}]\right) \; ,
	\]
	which is injective by Proposition~\ref{prop:botthom-K-inj}, as desired. 
\end{proof}

\newcommand{\rtimesred}{\rtimes_{\operatorname{r}}}

\begin{proof}[Proof of Theorem~\ref{thm:main}]
	Consider the commuting diagram
	\[
	\xymatrix{
		K_{*+1}^{\Gam}(\univspfree\Gamma)\otimes_{\zbbd}\qbbd \ar[r]^{\pi_*} \ar[d]^{[\bottmap]}  & K_{*+1}^{\Gam}(\univspproper\Gamma)\otimes_{\zbbd}\qbbd \ar[r]^\mu \ar[d]^{[\bottmap]}  & \kfunctr_{*+1}(C^*_{\operatorname{r}}\Gam)\otimes_{\zbbd}\qbbd \ar[d]^{[\bottmap] \rtimesred \Gam} \\
		\rkkgam[\rbbd,*](\univspfree\Gamma, \AIofM)\vphantom{\otimes_{\zbbd}\qbbd} \ar[r]^{\pi_*} \ar[d]^{(\operatorname{ev}_{1})_*}  & \rkkgam[\rbbd,*](\univspproper\Gamma, \AIofM)\vphantom{\otimes_{\zbbd}\qbbd} \ar[r]^\mu \ar[d]^{(\operatorname{ev}_{1})_*}  & \kfunctr_{\rbbd,*}(\AIofM \rtimesred \Gam)\vphantom{\otimes_{\zbbd}\qbbd} \ar[d]^{(\operatorname{ev}_{1})_* \rtimesred \Gam }  \\
		\rkkgam[\rbbd,*](\univspfree\Gamma, \AofM[M^{[0,1]}])\vphantom{\otimes_{\zbbd}\qbbd} \ar[r]^{\pi_*} & \rkkgam[\rbbd,*](\univspproper\Gamma, \AofM[M^{[0,1]}])\vphantom{\otimes_{\zbbd}\qbbd} \ar[r]^\mu & \kfunctr_{\rbbd,*}(\AofM[M^{[0,1]}] \rtimesred \Gam)\vphantom{\otimes_{\zbbd}\qbbd} 
	}
	\]
	where $\Gam$ acts on $\AIofM$ by $\alpha_{[0,1]}$ and on $\AofM[M^{[0,1]}]$ by $\alpha_1$. 
	Tracing along the leftmost column and then the bottom row, we see that the first vertical map is injective by Proposition~\ref{prop:beta-EGamma-injective}, the second vertical map is a bijection by Lemma~\ref{lem:AIofM-ev-KK}, the first horizontal map is injective by Lemma~\ref{lem:KKR-EGam-inj}, and the second horizontal map is bijective by Theorem~\ref{thm:proper-GHT} and the fact that $\AofM[M^{[0,1]}]$ is a proper $\Gamma$-$X$-{\cstaralg}, with $X$ being the spectrum of $\AevenofM[M^{[0,1]}]$, by Propositions~\ref{prop:isom-01-proper} and~\ref{prop_AofMisproper}. This implies the composition of the maps in the top row is injective, which is what we need. 
\end{proof}

\begin{proof}[Proof of Theorem~\ref{thm:diffeo}]
	This follows from Theorem~\ref{thm:main} and Proposition~\ref{prop:SLSO-proper-length}. 
\end{proof}

\section{Appendix: {\hhs}s and continuum products}\label{sec:appendix}

In this appendix, we prove some technical results regarding the permanence of {\hhs}s (Definition~\ref{defn:hhs}) under taking continuum products (Construction~\ref{constr:continuum-product}). These results are summarized in Proposition~\ref{prop:continuum-product-hhs-summary}. 

\begin{lem}\label{lem:angle-upper-semicontinuous}
	If $X$ is a CAT(0) space, then the map 
	\[
	(x, y , z)  \mapsto \angle([y,x],[y,z])
	\]
	is upper semi-continuous, that is, roughly speaking, small perturbations of $(x,y,z)$ do not increase $\angle([y,x],[y,z])$ by much. 
\end{lem}

\begin{proof}
	Given any $\varepsilon > 0$ and any distinct points $x_0,y_0,z_0 \in X$, we wish to find $r \in \left(0, \frac{1}{2} \min\{d(x_0,y_0), d(y_0,z_0), d(z_0,x_0)\} \right)$ such that for any $x \in B_r(x_0)$, $y \in B_r(y_0)$ and $z \in B_r(z_0)$, we have 
	\[
	\angle([y,x],[y,z]) \leq \angle([y_0,x_0],[y_0,z_0]) + \varepsilon \; .
	\]
	To this end, we find $u_0 \in [y_0,x_0] \setminus \{y_0\}$ and $v_0 \in [y_0, z_0] \setminus \{y_0\}$ such that 
	\[
	\widetilde{\angle}(u_0, y_0 , v_0) \leq \angle([y_0,x_0],[y_0,z_0]) + \frac{\varepsilon}{2} \; .
	\] 
	Since the function $(x,y,z) \mapsto \widetilde{\angle}(x,y,z)$ is continuous by its definition, we can find $r \in \left(0, \frac{1}{2} \min\{d(u_0,y_0), d(y_0,v_0), d(v_0,u_0)\} \right)$ such that for any $u \in B_r(u_0)$, $y \in B_r(y_0)$ and $v \in B_r(v_0)$, we have 
	\[
	\left| \widetilde{\angle}(u, y , v) - \widetilde{\angle}(u_0, y_0 , v_0) \right| \leq  \frac{\varepsilon}{2} \; .
	\]
	Now for any $x \in B_r(x_0)$, $y \in B_r(y_0)$ and $z \in B_r(z_0)$, Remark~\ref{rmk:CAT0-facts}\eqref{rmk:CAT0-facts-Lipschitz} implies that $[y,x] \cap B_r(u_0) \not= \varnothing$ and $[y,z] \cap B_r(v_0) \not= \varnothing$. Hence, fixing $u \in [y,x] \cap B_r(u_0)$ and $v \in [y,z] \cap B_r(v_0) $, we have
	\[
	\angle([y,x],[y,z]) \leq \widetilde{\angle}(u, y , v) \leq \widetilde{\angle}(u_0, y_0 , v_0) + \frac{\varepsilon}{2}  \leq \angle([y_0,x_0],[y_0,z_0]) + \varepsilon \; ,
	\]
	which is what we wanted to prove. 
\end{proof}

\begin{lem}\label{lem:log-inner-lower-semicontinuous}
	If $M$ is a {\hhs}, then with the notations of Constructions~\ref{constr:log-map} and~\ref{constr:Hilbert-space-span}, the map 
	\begin{align*}
	M \times M \times M & \to \rbbd \\
	(x,y,z) & \mapsto \left\langle \log_y (x) , \log_y (z) \right\rangle
	\end{align*}
	is lower semi-continuous. 
\end{lem}

\begin{proof}
	We observe that 
	\[
	\left\langle \log_y (x) , \log_y (z) \right\rangle = 
	\begin{cases}
	0 \; , & x = y \text{ or } y = z \\
	d(y,x) d(y,z) \cos \angle([y,x],[y,z]) \; , & \text{otherwise}
	\end{cases} \; .
	\]
	This is continuous at any point $(x,y,z)$ where $x=y$ or $y = z$ because the cosine function is bounded. At any other point, since $\angle([y,x],[y,z]) \in [0, \pi]$ and the cosine function is strictly decreasing on the interval $[0, \pi]$, it follows from Lemma~\ref{lem:angle-upper-semicontinuous} that the above function is lower semi-continuous. Combining the two cases gives the result. 
\end{proof}

For the sake of convenience in the proof of the next proposition, we write $\mathbb{L}^2(Y,\mu,M)$ for the space of all $L^2$-functions from $(Y,\mu)$ to $M$, \emph{without} identifying functions that are almost everywhere equal. Thus $L^2(Y,\mu,M)$ is a quotient of $\mathbb{L}^2(Y,\mu,M)$. 

\begin{prop}\label{prop:continuum-product-hhs}
	For any {\hhs} $M$ and finite measure space $(Y, \mu)$, the continuum product $L^2(Y,\mu,M)$ is again a {\hhs}. 
\end{prop}

\begin{proof}
	\newcommand{\la}{\left\langle}
	\newcommand{\ra}{\right\rangle}
	\newcommand{\LOG}{\operatorname{LOG}}
	\newcommand{\ol}{\overline}
	The fact that $L^2(Y,\mu,M)$ is a CAT(0) metric space follows from Proposition~\ref{prop:continuum-product-CAT0}. 
	
	Completeness of $L^2(Y,\mu,M)$ is proved in a similar way as that of classical $L^2$-spaces. More precisely, if $([\xi_n])_{n \in \nbbd}$ is a sequence in $\mathbb{L}^2(Y,\mu,M)$ that gives rise to a Cauchy sequence in $L^2(Y,\mu,M)$, then by passing to a subsequence, we may assume $d([\xi_n], [\xi_{n+1}]) \leq 2^{-2n}$. Define 
	\[
	Y_n = \left\{y \in Y \colon d_M(\xi_n(y), \xi_{n+1}(y) ) \geq {2^{-n}} \right\} \; .
	\]
	Thus $\mu(Y_n) \leq 2^{-2n}$ by the definition of the metric on $L^2(Y,\mu,M)$. Hence the set $\displaystyle \bigcap_{m=0}^\infty \bigcup_{n = m}^\infty Y_n$ has measure zero and for any $y$ in the complement of this set, the sequence $(\xi_n(y))_{n \in \nbbd}$ is Cauchy. It follows from the completeness of $M$ that $([\xi_n])_{n \in \nbbd}$ converges almost everywhere. A standard argument shows that the limit is in $\mathbb{L}^2(Y,\mu,M)$. 
	
	It remains to show that the tangent cone at any point of $L^2(Y,\mu,M)$ embeds isometrically into a Hilbert space. To this end, we fix any representative $\xi \in \mathbb{L}^2(Y,\mu,M)$ of the said point. 
	For each point $y \in Y$, by the assumption that $M$ is a {\hhs}, the tangent cone $T_{\xi(y)}M$ embeds isometrically into a Hilbert space $\hil_{\xi,y}$. Form the direct product vector space $\prod_{y} \hil_{\xi,y}$, whose generic element is written as $v = (v_{y})_{y \in Y}$, where $v_y \in \hil_{\xi,y}$. Consider the logarithm map $\log_{\xi(y)} \colon M \to T_{\xi(y)}M \subset \hil_{\xi,y}$ defined in Construction~\ref{constr:log-map}, which is a non-expansive map. These logarithm maps can be assembled into a map
	\begin{align*}
	\LOG_\xi \colon \mathbb{L}^2(Y,\mu,M) &\to \prod_{y} \hil_{\xi,y} \; , \\
	\varphi &\mapsto ( \log_{\xi(y)}(\varphi(y)) )_{y \in Y} \; .
	\end{align*}
	The image of this assembled logarithm map spans a vector subspace $V_\xi$.

	We wish to define a positive-semidefinite symmetric bilinear form on $V_\xi$ by integrating the inner products on $\hil_{\xi,y}$ over the measure $\mu$. To show that this is well defined, we need to show that for $v,w \in V_\xi$, the map
	\[
	y \mapsto \la  v_y, w_y \ra
	\]
	is a measurable function with a finite integral. To this end, it suffices to show the case when $v$ and $w$ are in the image of ${\LOG}_\xi$; thus we can consider $\varphi,\psi \in \mathbb{L}^2(Y,\mu,M)$ such that 
	\[
	(v_y) = ( \log_{\xi(y)} (\varphi(y)) ) \quad \text{and} \quad (w_y) = ( \log_{\xi(y)} (\psi(y)) ) \; ,
	\]
	so that the map $y \mapsto \la v_y, w_y \ra$ becomes
	\[
	y \mapsto \la \log_{\xi(y)} (\varphi(y)) ,\log_{\xi(y)} (\psi(y)) \ra \; , 
	\]
	which is a measurable function because $\xi$, $\varphi$ and $\psi$ are measurable and the function $(x_1,x_0,x_2) \mapsto \left\langle \log_{x_0} (x_1) , \log_{x_0} (x_2) \right\rangle$ is lower semi-continuous by Lemma~\ref{lem:log-inner-lower-semicontinuous}. On the other hand, by the Cauchy-Schwarz inequality and the metric properties of the logarithm map, we have 
	\begin{align*}
	&~ \int_{Y} \left| \la \log_{\xi(y)} (\varphi(y)),\log_{\xi(y)} (\psi(y)) \ra \right| d\mu (y) \\
	\leq &~ \int_{Y}  \left\| \log_{\xi(y)} (\varphi(y)) \right\| \left\| \log_{\xi(y)} (\psi(y)) \right\| d\mu (y) \\
	= &~ \int_{Y} d(\xi(y) , \varphi(y)) \cdot d(\xi(y) , \psi(y)) \, d\mu (y) \\
	\leq &~  d(\xi, \varphi) \cdot d(\xi, \psi) < \infty \;,
	\end{align*}
	which shows integrability. 
	
	It is straightforward to check that the formula
	\[
	\la v, w \ra = \int_Y \la v_y , w_y \ra \, d\mu (y)
	\]
	defines a positive-semidefinite symmetric bilinear form on $V_\xi$. Let $V_\xi^0 = \{ v \in V_\xi \colon \la v, w \ra = 0 \text{ for any } w \in V_\xi \}$ and let $\hil_\xi$ be the completion of $V_\xi / V_\xi^0$ by the inner product induced from $\la -, - \ra$.  
	
	It is easy to see that $T_{[\xi]} L^2(Y,\mu,M)$ embeds isometrically into $\hil_\xi$, by taking the fibre-wise embedding, because the tangent cone, which is constructed from geodesic segments, can be constructed fibrewise; after all, geodesic segments in $L^2(Y,\mu,M)$ are fibre-wise geodesic segments over each point $y \in Y$. This yields the desired isometric embedding of the tangent cone $T_{[\xi]} L^2(Y,\mu,M)$ into a Hilbert space.
\end{proof}

Next we discuss the admissibility of $L^2(Y,\mu,M)$ (cf., Definition~\ref{defn:hhs-admissible}). For this purpose, we need a few results which are reminiscent of the classical theory of $L^p$-spaces. 

\begin{lem}\label{lem:continuum-product-simple-functions}
	For any separable metric space $X$ and finite measure space $(Y, \mu)$, the set of \emph{simple functions} from $Y$ to $X$ (i.e., functions with finite ranges) is dense in $L^2(Y,\mu,X)$. 
\end{lem}

\begin{proof}
	Given any $\xi \in L^2(Y,\mu,X)$ and $\varepsilon > 0$, we are going to find a simple function $\eta \in L^2(Y,\mu,X)$ with $d(\xi, \eta) \leq \varepsilon \sqrt{\mu(Y) + 1}$. To this end, we pick a countable dense subset $\{ x_i \colon i \in \nbbd \}$ of $X$. Then there is a countable Borel partition $\{X_i \colon i \in \nbbd \}$ of $X$ such that each $X_i$ is contained in the $\varepsilon$-ball $B_{\varepsilon}(x_i)$ around $x_i$. Indeed, we may define
	\[
	X_i = B_{\varepsilon}(x_i) \setminus \left( \bigcup_{j=0}^{i-1} B_{\varepsilon}(x_j) \right) \; .
	\] 
	Write $Y_i = \xi^{-1} (X_i)$ (defined up to measure zero). Then $\{Y_i \colon i \in \nbbd \}$ is a countable measurable partition of $Y$. Hence if we write $x_0$ also for the constant function mapping $Y$ to $\{x_0\} \subset X$, then we have
	\[
	d(x_0, \xi)^2 =  \int_Y d_X( x_0, \xi(y) )^2 \, d\mu(y) = \sum_{i =0}^\infty \int_{Y_i} d_X( x_0, \xi(y) )^2 \, d\mu(y) \; .
	\]
	Since the right-hand side is a positive series that converges, there is $N  \in \nbbd$ such that 
	\[
	\sum_{i = N}^\infty \int_{Y_i} d_X( x_0, \xi(y) )^2 \, d\mu(y)  <  \varepsilon^2 \; .
	\]
	Now define the simple function
	\[
	\eta (y) = 
	\begin{cases}
	x_i \, , & y \in Y_i \text{ and } i < N \\
	x_0 \, , & y \in Y_i \text{ and } i \geq N
	\end{cases}
	\; .
	\]
	Then by our construction, we have 
	\begin{align*}
	d(\xi, \eta)^2 &= \sum_{i =0}^\infty \int_{Y_i} d_X( \eta(y), \xi(y) )^2 \, d\mu(y) \\
	&= \sum_{i =0}^{N-1} \int_{Y_i} d_X( x_i, \xi(y) )^2 \, d\mu(y) + \sum_{i = N}^\infty \int_{Y_i} d_X( x_0, \xi(y) )^2 \, d\mu(y)  \\
	&\leq \sum_{i =0}^{N-1} \int_{Y_i} \varepsilon^2 \, d\mu(y) + \varepsilon^2 \\
	&\leq \varepsilon^2 \mu(Y) + \varepsilon^2 \; . 
	\end{align*}
	Therefore we have $d(\xi, \eta) \leq \varepsilon \sqrt{\mu(Y) + 1}$. 
\end{proof}

For the following lemmas, we are going to view a \emph{finite measurable partition} of a measure space $(Y,\mu)$ as a measurable map from $Y$ to a finite space $I$, regarded as the \emph{index set}. Given two finite measurable partitions $\mathcal{P}_1, \mathcal{P}_2 \colon Y \to I$ with the same index set, we define
\[
\delta(\mathcal{P}_1, \mathcal{P}_2) = \sum_{i\in I} \mu \left( \mathcal{P}_1^{-1}(\{i\}) \bigtriangleup \mathcal{P}_2^{-1}(\{i\}) \right) \; .
\]
Given two finite measurable partitions $\mathcal{P}_j \colon Y \to I_j$ for $j = 1,2$, we say $\mathcal{P}_1$ \emph{refines} $\mathcal{P}_2$ if there is a map $f \colon I_1 \to I_2$ such that $\mathcal{P}_2 = f \circ \mathcal{P}_1$; more generally, for $\varepsilon > 0$, we say $\mathcal{P}_1$ \emph{$\varepsilon$-refines} $\mathcal{P}_2$ if there is a map $f \colon I_1 \to I_2$ such that $\delta(\mathcal{P}_2, f \circ \mathcal{P}_1) \leq \varepsilon$.

The following lemma can probably be found in the literature. We include a proof for completeness. 

\begin{lem}\label{lem:separable-measure-space-approx-partitions}
	Let $(Y,\mu)$ be a separable finite measure space. Then there is a sequence $(\mathcal{P}_n)_{n \in \nbbd}$ of finite measurable partitions of $(Y,\mu)$ with $\mathcal{P}_{n+1}$ refining $\mathcal{P}_{n}$ for all $n \in \nbbd$ and satisfying that for any $\varepsilon > 0$, any finite measurable partition of $(Y,\mu)$ is $\varepsilon$-refined by some $\mathcal{P}_{n}$.
\end{lem}

\begin{proof}
	Since $(Y,\mu)$ is separable, we can choose a countable family $\{A_n \colon n \in \nbbd \}$ of measurable subsets such that for any $\varepsilon >0$ and any measurable subset $A$ in $Y$, we have $\mu(A \bigtriangleup A_n) < \varepsilon$ for some $n$. For any $n \in \nbbd$, define $\mathcal{P}_n \colon Y \to \{0,1\}^{\{0,\ldots, n \}}$ by
	\[
	\mathcal{P}_n (y) (j) = 
	\begin{cases}
	0 , & y \notin A_j \\
	1 , & y \in A_j 
	\end{cases}
	\]
	for any $y \in Y$ and $j \in \{0,\ldots, n \}$. Then clearly $\mathcal{P}_{n+1}$ refines $\mathcal{P}_{n}$ for all $n \in \nbbd$. 
	
	To see that for any finite measurable partition $\mathcal{Q} \colon Y \to J$ and any $\varepsilon > 0$, we have that $\mathcal{Q}$ is $\varepsilon$-refined by some $\mathcal{P}_{n}$, we apply induction on the cardinality of $J$. The case when $J$ is a singleton is trivial. Now fix $k \in \nbbd$ and suppose for any $\delta > 0$, any finite measurable partition of $(Y,\mu)$ with the index set containing no more than $k$ elements is $\delta$-refined by some $\mathcal{P}_{n}$. Then given $\varepsilon > 0$ and a finite measurable partition $\mathcal{Q} \colon Y \to \{0, \ldots, k \}$, we form another finite measurable partition $\mathcal{R} \colon Y \to \{0, \ldots, k-1 \}$ by merging $\mathcal{Q}^{-1}(k-1)$ and $\mathcal{Q}^{-1}(k)$. By our choice of the family $\{ A_n \colon n \in \nbbd \}$, we can choose $A_n$ such that $\mu( \mathcal{Q}^{-1}(k) \bigtriangleup A_n) < \varepsilon/3$. On the other hand, by our inductive assumption, there is $m \in \nbbd$ such that $\mathcal{R}$ is $\delta$-refined by some $\mathcal{P}_{m}$, i.e., there is a map $f \colon \{0,1\}^{\{0,\ldots, m \}} \to \{0, \ldots, k-1 \}$ such that $\delta(\mathcal{R}, f \circ \mathcal{P}_m) \leq \varepsilon/3$. Without loss of generality, we may assume $m \geq n$. Define $g \colon \{0,1\}^{\{0,\ldots, m \}} \to \{0, \ldots, k \}$ such that 
	\[
	g \colon \{0,1\}^{\{0,\ldots, m \}} \to \{0, \ldots, k \} \, , \quad s = (s_j)_{j \in \{0,\ldots, m \}} \mapsto 
	\begin{cases}
	k \, , & s_n = 1 \\
	f(s) \, , & s_n = 0
	\end{cases}
	\; .
	\]
	Thus we have 
	\[
	\mu \left(\mathcal{Q}^{-1}(k) \bigtriangleup (g \circ \mathcal{P}_m)^{-1} (k)\right) = \mu \left( \mathcal{Q}^{-1}(k) \bigtriangleup A_n \right) < \varepsilon/3
	\]
	and for any $j \in \{0, \ldots, k-1 \}$, since $\mathcal{Q}^{-1}(j) = \mathcal{R}^{-1}(j) \setminus \mathcal{Q}^{-1}(k)$, we have
	\begin{align*}
	&~ \mu \left(\mathcal{Q}^{-1}(j) \bigtriangleup (g \circ \mathcal{P}_m)^{-1} (j)\right) \\
	=&~ \mu \left( \Big( \mathcal{R}^{-1}(j) \setminus \mathcal{Q}^{-1}(k) \Big) \bigtriangleup \Big( (f \circ \mathcal{P}_m)^{-1} (j) \setminus A_n \Big)\right) \\
	\leq &~ \mu \left(\mathcal{R}^{-1}(j) \bigtriangleup  (f \circ \mathcal{P}_m)^{-1} (j) \right) \\
	&+ \mu \left( \Big( \mathcal{R}^{-1}(j) \setminus \mathcal{Q}^{-1}(k) \Big) \cap A_n \right) + \mu \left( \mathcal{Q}^{-1}(k) \cap \Big( (f \circ \mathcal{P}_m)^{-1} (j) \setminus A_n \Big) \right) \; ,
	\end{align*}
	whence 
	\begin{align*}
	&~ \delta(\mathcal{Q}, g \circ \mathcal{P}_m) \\
	=&~ \sum_{j = 0}^k  \mu \left(\mathcal{Q}^{-1}(j) \bigtriangleup (g \circ \mathcal{P}_m)^{-1} (j)\right) \\
	\leq &~ \frac{\varepsilon}{3} +  \sum_{j = 0}^{k-1}   \mu \left(\mathcal{R}^{-1}(j) \bigtriangleup  (f \circ \mathcal{P}_m)^{-1} (j) \right)  \\
	& + \sum_{j = 0}^{k-1}  \left( \mu \left( \Big( \mathcal{R}^{-1}(j) \setminus \mathcal{Q}^{-1}(k) \Big) \cap A_n \right) + \mu \left( \mathcal{Q}^{-1}(k) \cap \Big( (f \circ \mathcal{P}_m)^{-1} (j) \setminus A_n \Big) \right) \right) \\
	=&~ \frac{\varepsilon}{3} +  \delta(\mathcal{R}, f \circ \mathcal{P}_m) + \mu \left( \Big( Y \setminus \mathcal{Q}^{-1}(k) \Big) \cap A_n \right) + \mu \left( \mathcal{Q}^{-1}(k) \cap \Big( Y \setminus A_n \Big) \right) \\
	\leq &~ \frac{\varepsilon}{3} +  \frac{\varepsilon}{3} + \mu \left(\mathcal{Q}^{-1}(k) \bigtriangleup A_n \right) \\
	\leq &~ \frac{\varepsilon}{3} +  \frac{\varepsilon}{3} + \frac{\varepsilon}{3} = \varepsilon \; .
	\end{align*}
	This shows $\mathcal{Q}$ is $\varepsilon$-refined by $\mathcal{P}_{m}$. The result thus follows by induction. 
\end{proof}

We remark that the condition in Lemma~\ref{lem:separable-measure-space-approx-partitions} is in fact equivalent to the separability of $(Y,\mu)$. 

\begin{prop}\label{prop:continuum-product-afdhhm}
	For any admissible {\hhs} $M$ and any separable finite measure space $(Y, \mu)$, the continuum product $L^2(Y,\mu,M)$ is again an admissible {\hhs}.
\end{prop}

\begin{proof}
	By our assumption, the {\hhs} $M$ contains an increasing sequence $M_0 \subset M_1 \subset \ldots$ of convex subsets isometric to finite\-/dimensional complete Riemannian manifolds, such that $M = \overline{\bigcup_{n \in \nbbd} M_n}$. We are going to produce such a sequence for $L^2(Y,\mu,M)$, too. 
	
	Since $(Y, \mu)$ is separable, we can find a sequence $(\mathcal{P}_n \colon Y \to I_n )_{n \in \nbbd}$ of finite measurable partitions of $(Y,\mu)$ satisfying the conditions in Lemma~\ref{lem:separable-measure-space-approx-partitions}. Let $\mu_n$ be the push-forward of the measure $\mu$ under $\mathcal{P}_n$. Then by Remark~\ref{rmk:continuum-product-functoriality}, there is a commutative diagram of isometric embeddings
	\[
	\xymatrix{
		L^2(I_0,\mu_0,M_0) \ar[r] \ar[rrrd]_{\iota_0}  & \ldots \ar[r] & L^2(I_n,\mu_n,M_n) \ar[r] \ar[rd]^{\iota_n} & \ldots \\
		&&& L^2(Y,\mu,M)
	}
	\]
	Under these embeddings, we view the spaces $L^2(I_n,\mu_n,M_n)$ as an increasing sequence of subspaces of $L^2(Y,\mu,M)$.	Observe that as a subspace, $L^2(I_n,\mu_n,M_n)$ consists of functions on $Y$ which are constant on each member set of the partition $I_n$ and take values in $M_n$. Thus it is convex in $L^2(Y,\mu,M)$ due to Remark~\ref{rmk:continuum-product-geodesics}. 
	By Example~\ref{eg:continuum-product-manifolds}, the space $L^2(I_n,\mu_n,M_n)$ is isometric to a Cartesian product of $M_n$ with a weighted $\ell^2$-metric, which is thus again a finite\-/dimensional complete Riemannian manifold. 
	
	It remains to show that 
	\[
	L^2(Y,\mu,M) = \overline{ \bigcup_{n \in \nbbd}  L^2(I_n,\mu_n,M_n) } \; .
	\]
	Thanks to Lemma~\ref{lem:continuum-product-simple-functions}, it suffices to show that for any simple function $\xi$ in $L^2(Y,\mu,M)$ and any $\varepsilon > 0$, we can find $m \in \nbbd$ and $\eta \in L^2(I_m,\mu_m,M_m)$ such that $d(\xi, \eta) < \varepsilon$. Since $\operatorname{im} \xi$ is a finite subset in $M$, by our choice of the sequence $(M_n)_{n \in \nbbd}$, there is $m_0 \in \nbbd$ such that $\operatorname{im} \xi$ is in the $\varepsilon'$-neighborhood of $M_{m_0}$ for $\varepsilon' = {\varepsilon}/\sqrt{2 \mu(Y)}$, that is, there is a map $g \colon \operatorname{im} \xi \to M_{m_0}$ such that $d(x, g(x)) \leq \varepsilon'$ for any $x \in \operatorname{im} \xi$. 
	On the other hand, since $\xi$ provides a finite measurable partition of $(Y,\mu)$, by our choice of the sequence $(\mathcal{P}_n)_{n \in \nbbd}$, there is a natural number $m \geq m_0$ such that the finite measurable partition $\xi$ is $\varepsilon''$-refined by $\mathcal{P}_{m}$ for $\varepsilon'' = {\varepsilon}^2/ \left( 2 {(\operatorname{diam} (\operatorname{im} \xi ) + \varepsilon' )^2}\right)$, that is, there is a map $f \colon I_m \to \operatorname{im} \xi$, such that $\delta(\xi, f \circ \mathcal{P}_{m}) < \varepsilon''$. Here $\operatorname{diam} (\operatorname{im} \xi ) = \sup_{x,x' \in \operatorname{im} \xi } d_M(x,x')$. 
	
	We thus define $\eta = g \circ f \circ \mathcal{P}_{m} \in L^2(I_m,\mu_m,M_m)$ and compute 
	\begin{align*}
	&~ d(\xi, \eta)^2 \\
	=&~ \int_{y \in Y} d_M(\xi(y), g \circ f \circ \mathcal{P}_{m} (y) )^2 \, d \mu (y) \\
	=&~ \sum_{x \in \operatorname{im} \xi } \int_{y \in \xi^{-1}(x) } d_M( x, g \circ f \circ \mathcal{P}_{m} (y) )^2 \, d \mu (y) \\
	=&~ \sum_{x \in \operatorname{im} \xi } \left( \int_{y \in \xi^{-1}(x) \cap \left( f \circ \mathcal{P}_{m} \right)^{-1}(x) } d_M( x, g(x) )^2 \, d \mu (y) \right. \\
	&~~~\qquad + \left. \int_{y \in \xi^{-1}(x) \setminus \left( f \circ \mathcal{P}_{m} \right)^{-1}(x) } d_M( x, g \circ f \circ \mathcal{P}_{m} (y) )^2 \, d \mu (y)  \right) \\
	\leq &~ \sum_{x \in \operatorname{im} \xi } \left( \int_{y \in \xi^{-1}(x) \cap \left( f \circ \mathcal{P}_{m} \right)^{-1}(x) } (\varepsilon')^2 \, d \mu (y) \right. \\
	&~~~\qquad + \left. \int_{y \in \xi^{-1}(x) \setminus \left( f \circ \mathcal{P}_{m} \right)^{-1}(x) } (\operatorname{diam} (\operatorname{im} \xi ) + \varepsilon' )^2 \, d \mu (y)    \right) \\
	\leq &~ \int_{y \in Y} \frac{\varepsilon^2}{2 \mu(Y)} \, d\mu(y) + \sum_{x \in \operatorname{im} \xi } \int_{y \in \xi^{-1}(x) \bigtriangleup \left( f \circ \mathcal{P}_{m} \right)^{-1}(x) } (\operatorname{diam} (\operatorname{im} \xi ) + \varepsilon' )^2 \, d \mu (y) \\
	\leq &~ \mu(Y) \cdot  \frac{\varepsilon^2}{2 \mu(Y)} + \delta(\xi, f \circ \mathcal{P}_{m}) \cdot  (\operatorname{diam} (\operatorname{im} \xi ) + \varepsilon' )^2 \\
	< &~ \varepsilon^2 \; ,
	\end{align*}
	as desired. 
\end{proof}

With similar techniques as in the above proof of Proposition~\ref{prop:continuum-product-afdhhm}, one can prove the following result. We omit the details as we will not make explicit use of this. 

\begin{prop}\label{prop:continuum-product-separable-separable}
	For any separable {\hhs} $M$ and any separable finite measure space $(Y, \mu)$, the continuum product $L^2(Y,\mu,M)$ is again a separable {\hhs}. \qed
\end{prop}

\begin{rmk}\label{rmk:not-Hilbert-manifold}
	The constructions in the proof of Proposition~\ref{prop:continuum-product-afdhhm} may be used to show that even if we start with a finite-dimensional Hadamard manifold $M$, as long as the sectional curvatures of $M$ are not all zero and the measure space $(Y, \mu)$ contains sets with arbitrarily small nonzero measures, the continuum product $L^2(Y,\mu,M)$ is not (isometric to) a Riemannian-Hilbertian manifold, because it is impossible to construct a Riemann curvature tensor. We explain this point in the next few paragraphs. 
	
	To begin with, we fix:
	\begin{enumerate}
		\item a base point $x \in M$ together with orthonormal tangent vectors $v, w$ in $T_x M$ such that the sectional curvature $K(v,w) < 0$, and 
		\item a sequence $\left(Y_n\right)_{n \in \nbbd}$ of disjoint measurable subsets of $Y$ such that $0 < \mu \left( Y_n \right) \leq \left(\frac{1}{3}\right)^n \mu \left( Y \right)$ for any $n \in \nbbd$. 
	\end{enumerate}
	To construct $\left(Y_n\right)_{n \in \nbbd}$, we use our assumption on $(Y,\mu)$ to choose a sequence of measurable subsets $\left(Z_n\right)_{n \in \nbbd}$ such that 
	$0 < \mu \left( Z_{n+1} \right) \leq \frac{1}{3} \mu \left( Z_n \right)$ for any $n \in \nbbd$. Then for any $n \in \nbbd$, we define 
	\[
		Y_n =  Z_n \setminus \left( \bigcup_{k = n+1}^{\infty} Z_k \right) 
	\]
	and observe that they are disjoint subsets of $Y$ and 
	\[
		0 < \frac{1}{2} \mu \left( Z_n \right) \leq \mu \left( Y_n \right) \leq \mu \left( Z_n \right) \leq \left(\frac{1}{3}\right)^n \mu \left( Y \right) \; .
	\]
	by properties of geometric series. Thus the sequence $\left(Y_n\right)_{n \in \nbbd}$ satisfies our requirements. 
	
	Now consider the sequence of finite measurable partitions $\left( \mathcal{P}_n \colon Y \to I_n \right)_{n \in \nbbd}$ where, for any $n \in \nbbd$, $I_n = \left\{ Y_0, \ldots, Y_n, Y \setminus \bigcup_{k=0}^{n} Y_k \right\}$ and $\mathcal{P}_n$ is the obvious quotient map. Clearly $\mathcal{P}_{n+1}$ refines  $\mathcal{P}_{n}$ for any $n \in \nbbd$. Hence following the proof of Proposition~\ref{prop:continuum-product-afdhhm}, we have a sequence of closed convex subsets 
	\[
		N_0 \subseteq N_1 \subseteq \ldots \subseteq L^2(Y,\mu, M)
	\]
	such that $N_n$ is canonically identified with $L^2(I_n, \mu_n, M)$, where $\mu_n \left( \{A\} \right) = \mu(A)$ for any $A \in I_n$. Observe that $N_0$ contains 
	the function $\xi \in L^2(Y,\mu, M)$ taking constant value $x$. Thus we have a sequence of tangent cones 
	\[
		T_\xi N_0 \subseteq T_\xi N_1 \subseteq \ldots \subseteq T_\xi L^2(Y,\mu, M) \; .
	\] 
	
	For any $n \in \nbbd$, we define vectors $\eta_n, \theta_n \in T_\xi N_n$ such that under the canonical identification $T_\xi N_n \simeq L^2(I_n, \mu_n, T_x M)$, we have 
	\[
		\eta_n \left( Y_k \right) = \left( \mu \left( Y_k \right) \right)^{-\frac{1}{4}} v \quad \text{and} \quad \theta_n \left( Y_k \right) = \left( \mu \left( Y_k \right) \right)^{-\frac{1}{4}} w  \quad \text{for } k \in \{0, \ldots, n \} 
	\]
	and $\eta_n \left( Y \setminus \bigcup_{k=0}^{n} Y_k \right) = \theta_n \left( Y \setminus \bigcup_{k=0}^{n} Y_k \right) = 0$. For any $n \in \nbbd$, since $\eta_n$ and $\eta_{n+1}$ only differ on $Y_{n+1}$, it follows that 
	\begin{align*}
		d_{T_\xi L^2(Y,\mu, M)} \left( \eta_n, \eta_{n+1} \right) = \sqrt{ \int_{Y_{n+1}} \left\| \left( \mu \left( Y_{n+1} \right) \right)^{-\frac{1}{4}} v \right\|^2 \, d \mu  } &\\
		= \left( \mu \left( Y_{n+1} \right) \right)^{\frac{1}{4}} &\leq  \left(\frac{1}{3}\right)^{\frac{n}{4}} \left( \mu \left( Y \right) \right)^{\frac{1}{4}}
	\end{align*}
	and thus $\left( \eta_n \right)_{n \in \nbbd}$ is a Cauchy sequence in $T_\xi L^2(Y,\mu, M)$, whose limit we denote by $\eta_\infty$. Similarly, $\left( \theta_n \right)_{n \in \nbbd}$ is a Cauchy sequence in $T_\xi L^2(Y,\mu, M)$, whose limit we denote by $\theta_\infty$.
	
	Let $R_x^M \colon T_x M \times T_x M \to \operatorname{End} \left( T_x M \right)$ be the Riemann curvature tensor of $M$ at $x$. By our choice of $v$ and $w$, we have $\left\langle R_x^M(v,w)w, v \right\rangle = K(v,w) < 0$. For any $n \in \nbbd$, by identifying $N_n$ as the Cartesian product $M^{I_n}$ with a Riemannian metric weighted by $\mu_n$, we see that the Riemann curvature tensor $R_{\xi}^{N_n}$ of $N_n$ at $\xi$ is ``fiberwise'' in the sense that 
	$\left( R_{\xi}^{N_n} (\kappa, \lambda) \nu \right) (A) = R_x^M \left( \kappa(A), \lambda(A) \right) \nu(A)$ for any $A \in I_n$ and $\kappa, \lambda, \nu \in L^2(I_n, \mu_n, T_x M) \simeq T_\xi N_n$, whence
	\begin{align*}
		&\ \left\langle R_{\xi}^{N_n} \left( \eta_n, \theta_n \right) \theta_n, \eta_n \right\rangle \\
		= &\ \sum_{k = 0}^{n} \left\langle R_x^M \left( \left( \mu \left( Y_{k} \right) \right)^{-\frac{1}{4}} v , \left( \mu \left( Y_{k} \right) \right)^{-\frac{1}{4}} w\right) \left( \left( \mu \left( Y_{k} \right) \right)^{-\frac{1}{4}} w \right), \left( \mu \left( Y_{k} \right) \right)^{-\frac{1}{4}} v \right\rangle \cdot  \mu \left( Y_{k} \right) \\
		= &\ (n+1) K(v,w) \; ,
	\end{align*}
	which approaches $-\infty$ as $n \to \infty$. 
	
	Now suppose $L^2(Y,\mu,M)$ were isometric to a Riemannian-Hilbertian manifold. Then each $N_n$ would be a geodesically closed Riemannian submanifold, and thus the Riemann curvature tensor of $L^2(Y,\mu,M)$ at $\xi$, a continuous 3-1 tensor (see \cite[Chapter~IX]{lang}), would coincide with $R_{\xi}^{N_n}$ when restricted to $T_\xi N_n$. However, the above computations show that it is impossible to satisfy continuity around the tuple $\left( \eta_\infty, \theta_\infty, \theta_\infty, \eta_\infty \right)$. 
\end{rmk}


\bibliographystyle{alpha}
\bibliography{Novikov-diffeo-Hilbert}

\begin{thebibliography}{GHW05}

\bibitem[AAS14]{antoniniazzaliskandalis2014}
Paolo Antonini, Sara Azzali, and Georges Skandalis.
\newblock Flat bundles, von {N}eumann algebras and {$K$}-theory with
  {$\Bbb{R}/\Bbb{Z}$}-coefficients.
\newblock {\em J. K-Theory}, 13(2):275--303, 2014.

\bibitem[AAS16]{antoniniazzaliskandalis2016}
Paolo Antonini, Sara Azzali, and Georges Skandalis.
\newblock Bivariant {$K$}-theory with {$\Bbb{R}/\Bbb{Z}$}-coefficients and rho
  classes of unitary representations.
\newblock {\em J. Funct. Anal.}, 270(1):447--481, 2016.

\bibitem[AAS18]{antoniniazzaliskandalis}
Paolo Antonini, Sara Azzali, and Georges Skandalis.
\newblock The {B}aum-{C}onnes conjecture localised at the unit element of a
  discrete group.
\newblock {\em Compositio Mathematica}, to appear\vphantom{2018}.

\bibitem[BBI01]{burago}
Dmitri Burago, Yuri Burago, and Sergei Ivanov.
\newblock {\em A course in metric geometry}, volume~33 of {\em Graduate Studies
  in Mathematics}.
\newblock American Mathematical Society, Providence, RI, 2001.

\bibitem[BC88]{baumconnes88}
Paul Baum and Alain Connes.
\newblock {$K$}-theory for discrete groups.
\newblock In {\em Operator algebras and applications, {V}ol.\ 1}, volume 135 of
  {\em London Math. Soc. Lecture Note Ser.}, pages 1--20. Cambridge Univ.
  Press, Cambridge, 1988.

\bibitem[BCH94]{baumconneshigson}
Paul Baum, Alain Connes, and Nigel Higson.
\newblock Classifying space for proper actions and {$K$}-theory of group
  {$C^\ast$}-algebras.
\newblock In {\em {$C^\ast$}-algebras: 1943--1993 ({S}an {A}ntonio, {TX},
  1993)}, volume 167 of {\em Contemp. Math.}, pages 240--291. Amer. Math. Soc.,
  Providence, RI, 1994.

\bibitem[BGH17]{bestvinaguirardelhorbez}
Mladen Bestvina, Vincent Guirardel, and Camille Horbez.
\newblock Boundary amenability of ${Out(F_N)}$.
\newblock preprint, arXiv:1705.07017, 2017.

\bibitem[BH99]{BridsonHaefliger1999Metric}
Martin~R. Bridson and Andr\'{e} Haefliger.
\newblock {\em Metric spaces of non-positive curvature}, volume 319 of {\em
  Grundlehren der Mathematischen Wissenschaften [Fundamental Principles of
  Mathematical Sciences]}.
\newblock Springer-Verlag, Berlin, 1999.

\bibitem[BT72]{BruhatTits1972Groupes}
F.~Bruhat and J.~Tits.
\newblock Groupes r\'{e}ductifs sur un corps local.
\newblock {\em Inst. Hautes \'{E}tudes Sci. Publ. Math.}, (41):5--251, 1972.

\bibitem[CGM93]{connesgromovmoscovici}
Alain Connes, Mikhail Gromov, and Henri Moscovici.
\newblock Group cohomology with {L}ipschitz control and higher signatures.
\newblock {\em Geom. Funct. Anal.}, 3(1):1--78, 1993.

\bibitem[CH90]{ConnesHigson1990}
Alain Connes and Nigel Higson.
\newblock D\'eformations, morphismes asymptotiques et {$K$}-th\'eorie
  bivariante.
\newblock {\em C. R. Acad. Sci. Paris S\'er. I Math.}, 311(2):101--106, 1990.

\bibitem[CM90]{connesmoscovici}
Alain Connes and Henri Moscovici.
\newblock Cyclic cohomology, the {N}ovikov conjecture and hyperbolic groups.
\newblock {\em Topology}, 29(3):345--388, 1990.

\bibitem[Con86]{connes1986}
Alain Connes.
\newblock Cyclic cohomology and the transverse fundamental class of a
  foliation.
\newblock In {\em Geometric methods in operator algebras ({K}yoto, 1983)},
  volume 123 of {\em Pitman Res. Notes Math. Ser.}, pages 52--144. Longman Sci.
  Tech., Harlow, 1986.

\bibitem[Con94]{connes}
Alain Connes.
\newblock {\em Noncommutative geometry}.
\newblock Academic Press Inc., San Diego, CA, 1994.

\bibitem[Dav96]{davidson}
Kenneth~R. Davidson.
\newblock {\em {$C^*$}-algebras by example}, volume~6 of {\em Fields Institute
  Monographs}.
\newblock American Mathematical Society, Providence, RI, 1996.

\bibitem[FS08]{fishersilberman}
David Fisher and Lior Silberman.
\newblock Groups not acting on manifolds.
\newblock {\em Int. Math. Res. Not. IMRN}, (16):Art. ID rnn060, 11, 2008.

\bibitem[GHT00]{guentnerhigsontrout}
Erik Guentner, Nigel Higson, and Jody Trout.
\newblock Equivariant {$E$}-theory for {$C^*$}-algebras.
\newblock {\em Mem. Amer. Math. Soc.}, 148(703):viii+86, 2000.

\bibitem[GHW05]{guentnerhigsonweinberger}
Erik Guentner, Nigel Higson, and Shmuel Weinberger.
\newblock The {N}ovikov conjecture for linear groups.
\newblock {\em Publ. Math. Inst. Hautes \'Etudes Sci.}, (101):243--268, 2005.

\bibitem[Gre82]{green1982}
Philip Green.
\newblock Equivariant k-theory and crossed product c-star-algebras.
\newblock In {\em Proceedings of Symposia in Pure Mathematics}, volume~38,
  pages 337--338, Providence, RI, 1982. American Mathematical Society.

\bibitem[Ham09]{hamenstadt}
Ursula Hamenst\"{a}dt.
\newblock Geometry of the mapping class groups. {I}. {B}oundary amenability.
\newblock {\em Invent. Math.}, 175(3):545--609, 2009.

\bibitem[HG04]{higsonguentner}
Nigel Higson and Erik Guentner.
\newblock Group {$C^\ast$}-algebras and {$K$}-theory.
\newblock In {\em Noncommutative geometry}, volume 1831 of {\em Lecture Notes
  in Math.}, pages 137--251. Springer, Berlin, 2004.

\bibitem[Hig00]{higson00}
Nigel Higson.
\newblock Bivariant {$K$}-theory and the {N}ovikov conjecture.
\newblock {\em Geom. Funct. Anal.}, 10(3):563--581, 2000.

\bibitem[HK01]{higsonkasparov}
Nigel Higson and Gennadi Kasparov.
\newblock {$E$}-theory and {$KK$}-theory for groups which act properly and
  isometrically on {H}ilbert space.
\newblock {\em Invent. Math.}, 144(1):23--74, 2001.

\bibitem[HKT98]{higsonkasparovtrout}
Nigel Higson, Gennadi Kasparov, and Jody Trout.
\newblock A {B}ott periodicity theorem for infinite-dimensional {E}uclidean
  space.
\newblock {\em Adv. Math.}, 135(1):1--40, 1998.

\bibitem[Jul81]{julg1981}
Pierre Julg.
\newblock {$K$}-th\'{e}orie \'{e}quivariante et produits crois\'{e}s.
\newblock {\em C. R. Acad. Sci. Paris S\'{e}r. I Math.}, 292(13):629--632,
  1981.

\bibitem[Kas88]{kasparov1}
Gennadi Kasparov.
\newblock Equivariant {$KK$}-theory and the {N}ovikov conjecture.
\newblock {\em Invent. Math.}, 91(1):147--201, 1988.

\bibitem[Kas95]{kasparov95}
Gennadi Kasparov.
\newblock {$K$}-theory, group {$C^*$}-algebras, and higher signatures
  (conspectus).
\newblock In {\em Novikov conjectures, index theorems and rigidity, {V}ol.\ 1
  ({O}berwolfach, 1993)}, volume 226 of {\em London Math. Soc. Lecture Note
  Ser.}, pages 101--146. Cambridge Univ. Press, Cambridge, 1995.

\bibitem[Kid08]{kida}
Yoshikata Kida.
\newblock The mapping class group from the viewpoint of measure equivalence
  theory.
\newblock {\em Mem. Amer. Math. Soc.}, 196(916):viii+190, 2008.

\bibitem[KS91]{kasparovskandalis91}
Guennadi Kasparov and Georges Skandalis.
\newblock Groups acting on buildings, operator {$K$}-theory, and {N}ovikov's
  conjecture.
\newblock {\em $K$-Theory}, 4(4):303--337, 1991.

\bibitem[KS03]{kasparovskandalis2}
Gennadi Kasparov and Georges Skandalis.
\newblock Groups acting properly on ``bolic'' spaces and the {N}ovikov
  conjecture.
\newblock {\em Ann. of Math. (2)}, 158(1):165--206, 2003.

\bibitem[KY12]{kasparovyu12}
Gennadi Kasparov and Guoliang Yu.
\newblock The {N}ovikov conjecture and geometry of {B}anach spaces.
\newblock {\em Geom. Topol.}, 16(3):1859--1880, 2012.

\bibitem[Lan99]{lang}
Serge Lang.
\newblock {\em Fundamentals of differential geometry}, volume 191 of {\em
  Graduate Texts in Mathematics}.
\newblock Springer-Verlag, New York, 1999.

\bibitem[Mis74]{miscenko74}
Alexander~S. Mis\llap{\v{s}}\v{c}enko.
\newblock Infinite-dimensional representations of discrete groups, and higher
  signatures.
\newblock {\em Izv. Akad. Nauk SSSR Ser. Mat.}, 38:81--106, 1974.

\bibitem[Mos65]{Moser1965}
J\"{u}rgen Moser.
\newblock On the volume elements on a manifold.
\newblock {\em Trans. Amer. Math. Soc.}, 120:286--294, 1965.

\bibitem[Ros83]{rosenberg1983c}
Jonathan Rosenberg.
\newblock $ {C}^{*} $-algebras, positive scalar curvature, and the {N}ovikov
  conjecture.
\newblock {\em Publications Math{\'e}matiques de l'IH{\'E}S}, 58:197--212,
  1983.

\bibitem[RS87]{RosenbergSchochet1987}
Jonathan Rosenberg and Claude Schochet.
\newblock The {K}\"{u}nneth theorem and the universal coefficient theorem for
  {K}asparov's generalized {$K$}-functor.
\newblock {\em Duke Math. J.}, 55(2):431--474, 1987.

\bibitem[Sch38]{schoenberg38}
Isaac~Jacob Schoenberg.
\newblock Metric spaces and positive definite functions.
\newblock {\em Trans. Amer. Math. Soc.}, 44(3):522--536, 1938.

\bibitem[Ska91]{Skandalis1991Le}
Georges Skandalis.
\newblock Le bifoncteur de {K}asparov n'est pas exact.
\newblock {\em C. R. Acad. Sci. Paris S\'er. I Math.}, 313(13):939--941, 1991.

\bibitem[STY02]{skandalistuyu}
Georges Skandalis, Jean~Louis Tu, and Guoliang Yu.
\newblock The coarse {B}aum-{C}onnes conjecture and groupoids.
\newblock {\em Topology}, 41(4):807--834, 2002.

\bibitem[Val02]{Valette2002}
Alain Valette.
\newblock {\em Introduction to the {B}aum-{C}onnes conjecture}.
\newblock Lectures in Mathematics ETH Z\"urich. Birkh\"auser Verlag, Basel,
  2002.
\newblock From notes taken by Indira Chatterji, With an appendix by Guido
  Mislin.

\bibitem[Yu98]{yu2}
Guoliang Yu.
\newblock The {N}ovikov conjecture for groups with finite asymptotic dimension.
\newblock {\em Ann. of Math. (2)}, 147(2):325--355, 1998.

\bibitem[Yu00]{yu3}
Guoliang Yu.
\newblock The coarse {B}aum-{C}onnes conjecture for spaces which admit a
  uniform embedding into {H}ilbert space.
\newblock {\em Invent. Math.}, 139(1):201--240, 2000.

\end{thebibliography}

\end{document}